\documentclass[11pt,leqno]{amsart}

\usepackage[a4paper, right=20mm, left=20mm, top=25mm]{geometry}

\usepackage{microtype}

\usepackage{amsthm}
\usepackage[english]{babel}
\usepackage[utf8]{inputenc}
\usepackage{hyperref}

\usepackage{amsmath}
\usepackage{amssymb}
\usepackage{mathtools}
\usepackage{mathrsfs}

\usepackage{pdflscape}

\usepackage{tikz}
\usetikzlibrary{matrix,arrows,positioning,cd}

\numberwithin{equation}{section}
\numberwithin{figure}{section}

\theoremstyle{plain}
\newtheorem{Theorem}{Theorem}[section]
\newtheorem{Lemma}[Theorem]{Lemma}
\newtheorem{Proposition}[Theorem]{Proposition}
\newtheorem{Corollary}[Theorem]{Corollary}
\newtheorem{thmalph}{Theorem}

\newtheorem{Conjecture}[thmalph]{Conjecture}
\newtheorem{Hypothesis}[thmalph]{Hypothesis}

\theoremstyle{definition}
\newtheorem{Definition}[Theorem]{Definition}
\newtheorem{Remark}[Theorem]{Remark}
\newtheorem{Example}[Theorem]{Example}

\newcommand{\C}{\mathbb{C}}
\newcommand{\R}{\mathbb{R}}

\newcommand{\Z}{\mathbb{Z}}

\newcommand{\Gr}{\mathcal{G}r}
\newcommand{\Fl}{\mathcal{F}l}

\newcommand{\GL}{\mathrm{GL}}

\providecommand{\abs}[1]{\lvert#1\rvert}

\title[Pseudo-centralizers in affine Hecke algebras]{Pseudo-centralizers in affine Hecke algebras}
\author{Jonathan Gruber}
\address{RWTH Aachen, Chair of Algebra and Representation Theory, Pontdriesch 10--12, 52062 Aachen, Germany}
\email{gruber@art.rwth-aachen.de}
\date{\today}

\begin{document}

\begin{abstract}
	We introduce and study a subalgebra $\mathcal{B}$ of the affine Hecke algebra, which arises from a centralizer construction in the double affine Hecke algebra, and which may be regarded as a $v$-deformation of the affine Fomin--Stanley subalgebra introduced by Lam \cite{LamSchubert} as a combinatorial model for the affine Grassmannian homology ring.
	In types $\mathsf{A}_n$ and $\mathsf{B}_2$ and $\mathsf{G}_2$, we show that $\mathcal{B}_\mathrm{aff}$ admits a canonical basis indexed by the cosets of the finite Weyl group in the affine Weyl group.
	We also discuss conjectural positivity properties of the canonical basis and explain how it can be used to study the center of the affine Hecke algebra.
\end{abstract}

\maketitle

\section*{Introduction}

Let $\Phi$ be a simple root system with finite Weyl group $W_\mathrm{fin}$, affine Weyl group $W_\mathrm{aff} = W_\mathrm{fin} \ltimes \Z\Phi^\vee$ and weight lattice $X$.
The affine Hecke algebra $\mathcal{H}$ and the double affine Hecke algebra $\mathbb{H}$ are deformations the group algebras $\Z[W_\mathrm{aff}]$ and $\Z[ X \rtimes W_\mathrm{aff} ]$, respectively, over the Laurent polynomial ring $\mathcal{A} = \Z[v,v^{-1}]$.
Both $\mathcal{H}$ and the group ring $\mathcal{A}[X] = \bigoplus_{\lambda \in X} \mathcal{A} \cdot X_\lambda$ of $X$ can be identified with $\mathcal{A}$-subalgebras of $\mathbb{H}$, and the multiplication in $\mathbb{H}$ affords an isomorphism of $\mathcal{A}$-modules $\mathbb{H} \cong \mathcal{A}[X] \otimes_\mathcal{A} \mathcal{H}$.
We write $\varphi \colon \mathbb{H} \to \mathcal{H}$ for the unique $\mathcal{A}$-linear map such that $\varphi(X_\lambda h) = h$ for all $\lambda \in X$ and $h \in \mathcal{H}$,%
\footnote{We stress that $\varphi \colon \mathbb{H} \to \mathcal{H}$ is not a homomorphism of $\mathcal{A}$-algebras.}
and we define the \emph{pseudo-centralizer} of $X$ in $\mathcal{H}$ via
\[ \mathcal{B} \coloneqq \big\{ h \in \mathcal{H} \mathop{\big|} \varphi( h X_\lambda ) = h \text{ for all } \lambda \in X \big\} . \]
In this article, we initiate the study of $\mathcal{B}$ using algebraic and combinatorial techniques.
For instance, we prove in Section \ref{sec:pseudocentralizer} below that $\mathcal{B}$ is a subalgebra of $\mathcal{H}$ and that the restriction of $\varphi$ to the centralizer $\mathcal{C} = C_\mathbb{H}(X)$ of $\mathcal{A}[X]$ in $\mathbb{H}$ is an $\mathcal{A}$-algebra homomorphism $\varphi \colon \mathcal{C} \to \mathcal{B}$, giving rise to the following commutative diagram:
\[ \begin{tikzcd} \mathcal{C} \ar[r] \ar[d,"\varphi",swap] & \mathbb{H} \ar[d,"\varphi"] \\ \mathcal{B} \ar[r] & \mathcal{H}  \end{tikzcd} \]
Furthermore, we show that $\mathcal{B}$ contains the center $Z(\mathcal{H})$ of $\mathcal{H}$.

The definition of $\mathcal{B}$ and the aforementioned structural results are strongly inspired by results of Lam \cite{LamSchubert} (and Lam--Schilling--Shimozono \cite{LamSchillingShimozono}), where a subalgebra analogous to $\mathcal{B}$ in the affine nil-Hecke algebra (or the affine $0$-Hecke algebra) is shown to provide a combinatorial model for the homology (or $K$-homology) of the affine Grassmannian.
While a similar geometric interpretation of $\mathcal{B}$ is currently lacking, the results from \cite{LamSchubert,LamSchillingShimozono} still provide an important motivation and some useful guiding principles for the study of $\mathcal{B}$.
We briefly summarize these results below before before proceeding with our discussion of $\mathcal{B}$.

\subsection*{Affine Grassmannian homology}

Let $G$ be the simply connected simple algebraic group over $\C$ with root system $\Phi$, and let $T \subseteq G$ be a maximal torus.
Further let $\Gr_G$ and $\Fl_G$ be the corresponding affine Grassmannian and affine Flag manifold.
The $T$-equivariant homology $H^T_\bullet( \Fl_G )$ can be identified with the double affine nil-Hecke algebra $\mathbb{H}_\mathrm{nil}$, which is generated by the affine nil-Hecke algebra $\mathcal{H}_\mathrm{nil}$ and the symmetric algebra $S = \mathrm{Sym}(X)$, subject to certain commutation relations given explicitly in Section \ref{sec:typeA} below.
As for the double affine Hecke algebra, there is a $\Z$-linear map $\varphi_\mathrm{nil} \colon \mathbb{H}_\mathrm{nil} \to \mathcal{H}_\mathrm{nil}$ with $\varphi_\mathrm{nil}(fh) = f(0) \cdot h$ for all $f \in S$ and $h \in \mathcal{H}_\mathrm{nil}$, which corresponds to ``forgetting equivariance''
\[ \mathrm{for} \colon H^T_\bullet( \Fl_G ) \longrightarrow H_\bullet( \Fl_G ) \cong \mathcal{H}_\mathrm{nil} \]
in homology.
Lam proves in \cite[Theorems 4.4 and 5.5]{LamSchubert} that the equivariant homology $H^T_\bullet( \Gr_G )$ can be identified with the centralizer $\mathcal{C}_\mathrm{nil} = C_{\mathbb{H}_\mathrm{nil}}(S)$ of $S$ in $\mathbb{H}_\mathrm{nil}$ and that the non-equivariant homology $H_\bullet( \Gr_G )$ can be identified with the subalgebra
\[ \mathcal{B}_\mathrm{nil} \coloneqq \big\{ h \in \mathcal{H}_\mathrm{nil} \mathop{\big|} \varphi( h f ) = f(0) \cdot h \text{ for all } f \in S \big\} \]
of $\mathcal{H}_\mathrm{nil}$.
In summary, we have the following commutative diagrams, which correspond to each other via the isomorphism $\mathbb{H}_\mathrm{nil} \cong H^T_\bullet( \Fl_G )$:
\[ \begin{tikzcd}[baseline={(current bounding box.center)}] \mathcal{C}_\mathrm{nil} \ar[r] \ar[d,"\varphi_\mathrm{nil}",swap] & \mathbb{H}_\mathrm{nil} \ar[d,"\varphi_\mathrm{nil}"] \\ \mathcal{B}_\mathrm{nil} \ar[r] & \mathcal{H}_\mathrm{nil}  \end{tikzcd}
\hspace{2cm}
\begin{tikzcd}[baseline={(current bounding box.center)}] H^T_\bullet( \Gr_G ) \ar[r] \ar[d,"\mathrm{for}",swap] & H^T_\bullet( \Fl_G ) \ar[d,"\mathrm{for}"] \\ H_\bullet( \Gr_G ) \ar[r] & H_\bullet( \Fl_G )  \end{tikzcd} \]
The horizontal arrows in the left hand diagram are the obvious inclusions, whereas the horizontal arrows in the right hand diagram are the embeddings given by the ``wrong way'' map also considered in \cite[Section 5.1]{LamSchillingShimozono}.
Replacing the double affine nil-Hecke algebra by the double affine $0$-Hecke algebra, Lam--Schilling--Shimozono prove analogous results about the equivariant and non-equivariant $K$-homology of $\Gr_G$ and $\Fl_G$ in \cite{LamSchillingShimozono}.

As a consequence of the isomorphism $\mathcal{B}_\mathrm{nil} \cong H_\bullet( \Gr_G )$, the subalgebra $\mathcal{B}_\mathrm{nil}$ of $\mathcal{H}_\mathrm{nil}$ has a basis corresponding to the basis of $H_\bullet( \Gr_G )$ given by the fundamental classes of Schubert varieties.
To make this more specific, let us write $\ell \colon W_\mathrm{aff} \to \Z_{\geq 0}$ for the length function, and let
\[ W_\mathrm{aff}^+ = \{ x \in W_\mathrm{aff} \mid x \text{ has minimal length in } W_\mathrm{fin} x \} \]
be the set of minimal $W_\mathrm{fin}$-coset representatives in $W_\mathrm{aff}$.
Then the affine Grassmannian has a Bruhat decomposition
\[ \Gr_G = \bigsqcup_{w \in W_\mathrm{aff}^+} \Omega_w = \bigcup_{w \in W_\mathrm{aff}^+} X_w , \]
where $\Omega_w$ is a Schubert cell and its closure $X_w = \overline{\Omega_w}$ is a Schubert variety, for $w \in W_\mathrm{aff}^+$.
The fundamental classes of the Schubert varieties form a basis $\{ [X_w] \mid w \in W_\mathrm{aff}^+ \}$ of $H_\bullet( \Gr_G )$, and the latter corresponds via the isomorphism $\mathcal{B}_\mathrm{nil} \cong H_\bullet( \Gr_G )$ to a basis $\{ b_w \mid w \in W_\mathrm{aff}^+ \}$ of $\mathcal{B}_\mathrm{nil}$, which we call the \emph{Schubert basis}.
The Schubert basis can also described in a purely algebro-combinatorial way:

Writing $\{ A_x \mid x \in W_\mathrm{aff} \}$ for the standard basis of $\mathcal{H}_\mathrm{nil}$ such that $A_x A_y = A_{xy}$ if $\ell(xy) = \ell(x)+\ell(y)$ and $A_x A_y = 0$ if $\ell(xy) < \ell(x)+\ell(y)$, Lam proves in \cite[Proposition 5.4]{LamSchubert} that $b_w$ can be characterized as the unique element of $\mathcal{B}_\mathrm{nil}$ whose expansion
\[ b_w = \sum_{x \in W_\mathrm{aff}} p_{x,w} \cdot A_x \]
in terms of the standard basis satisfies $p_{w,w} = 1$ and $p_{x,w} = 0$ for all $w \neq x \in W_\mathrm{aff}^+$.

\subsection*{Bases of \texorpdfstring{$\mathcal{B}$}{B}}

Although we do not know a geometric interpretation of the subalgebra $\mathcal{B}$ of $\mathcal{H}$ (in terms of $\Gr_G$ or otherwise), we may still ask whether there are bases of $\mathcal{B}$ with the same algebro-combinatorial behavior as the Schubert basis of $\mathcal{B}_\mathrm{nil}$.
A new feature that arises in the study of $\mathcal{B}$ in this context (in contrast to the study of $\mathcal{B}_\mathrm{nil}$ or its $K$-theory version) is that the affine Hecke algebra has two different bases (the \emph{standard basis} and the \emph{Kazhdan--Lusztig basis}) that could both play the role that the standard basis $\{A_x \mid x \in W_\mathrm{aff}\}$ of $\mathcal{H}_\mathrm{nil}$ plays in the above characterization of $b_w$, and both choices of basis potentially lead to interesting classes of elements of $\mathcal{B}$.

Let us write $\{ H_x \mid x \in W_\mathrm{aff} \}$ for the standard basis and $\{ \underline{H}_x \mid x \in W_\mathrm{aff} \}$ for the Kazhdan--Lusztig basis of $\mathcal{H}$ (as defined in Subsections \ref{subsec:Heckealgebras} and \ref{subsec:KLbasis}).
In Section \ref{sec:standardbasis} below, we prove that for all $w \in W_\mathrm{aff}^+$, there is at most one element $B_w \in \mathcal{B}$ whose expansion
\[ B_w = \sum_{x \in W_\mathrm{aff}} P_{x,w} \cdot H_x \]
in terms of the standard basis satisfies $P_{w,w} = 1$ and $P_{x,w} = 0$ for $w \neq x \in W_\mathrm{aff}^+$.
Furthermore, if an element $B_w \in \mathcal{B}$ as above exists for all $w \in W_\mathrm{aff}$, then the elements $\{ B_w \mid w \in W_\mathrm{aff}^+ \}$ form a basis of $\mathcal{B}$, which we call the \emph{standard basis}.
Lacking a geometric interpretation of $\mathcal{B}$, we do not know how to prove the existence of the standard basis of $\mathcal{B}$ in full generality, but we can construct it using explicit computations for root systems of type $\mathsf{A}_n$ or of rank $2$.

\begin{thmalph} \label{thm:introstandardbasis}
	Suppose that $\Phi$ is of type $\mathsf{A}_n$ (for some $n \geq 1$) or $\mathsf{B}_2$ or $\mathsf{G}_2$.
	Then for all $w \in W_\mathrm{aff}^+$, there exists an element $B_w \in \mathcal{B}$ whose expansion
	\[ B_w = \sum_{x \in W_\mathrm{aff}} P_{x,w} \cdot H_x \]
	 in terms of the standard basis of $\mathcal{H}$ satisfies $P_{w,w} = 1$ and $P_{x,w} = 0$ for $w \neq x \in W_\mathrm{aff}^+$.
	In particular, the standard basis $\{ B_w \mid w \in W_\mathrm{aff}^+ \}$ of $\mathcal{B}$ exists.
\end{thmalph}

The computations that lead to Theorem \ref{thm:introstandardbasis} (in Section \ref{sec:typeA} and Appendix \ref{app:rank2}) seem infeasible outside of type $\mathsf{A}_n$ or small rank root system.
Nevertheless, we expect that the standard basis of $\mathcal{B}$ should exist independently of the type of $\Phi$.

In the cases where the standard basis of $\mathcal{B}$ is known to exist (in types $\mathsf{A}_n$ and $\mathsf{B}_2$ and $\mathsf{G}_2$), its elements additionally satisfy the condition
\begin{equation} \label{eq:Bwfiltration}
	B_w \in \mathcal{H}^{\leq \ell(w)} \coloneqq \mathrm{span}_\mathcal{A}\{ H_x \mid x \in W_\mathrm{aff} \text{ with } \ell(x) \leq \ell(w) \} \tag{$*$}
\end{equation}
for all $w \in W_\mathrm{aff}^+$.
Assuming the existence of the standard basis and the condition \eqref{eq:Bwfiltration}, we can also define a canonical basis of $\mathcal{B}$ as in the following theorem.

\begin{thmalph}
	Suppose that the standard basis of $\mathcal{B}$ exists and that $B_w \in \mathcal{H}^{\leq \ell(w)}$ for all $w \in W_\mathrm{aff}^+$.
	Then for all $w \in W_\mathrm{aff}^+$, there is a unique element $\underline{B}_w \in \mathcal{B}$ whose expansion
	\[ \underline{B}_w = \sum_{x \in W_\mathrm{aff}} Q_{x,w} \cdot \underline{H}_x \]
	 in terms of the Kazhdan--Lusztig basis of $\mathcal{H}$ satisfies $Q_{w,w} = 1$ and $Q_{x,w} = 0$ for $w \neq x \in W_\mathrm{aff}^+$.
	 Furthermore, the elements $\{ \underline{B}_w \mid w \in W_\mathrm{aff}^+ \}$ form a basis of $\mathcal{H}$, which we call the \emph{canonical basis}.
\end{thmalph}

The canonical basis of $\mathcal{B}$ (when it exists) combines aspects of Schubert calculus (being defined as an analogue of the Schubert basis of $\mathcal{B}_\mathrm{nil}$) and of Kazhdan--Lusztig theory (being invariant under the bar involution of $\mathcal{H}$).
Furthermore, we have observed in sample computations that the canonical basis seems to have interesting positivity properties.
By analogy with the Kazhdan--Lusztig positivity conjectures (a theorem due to Elias--Williamson \cite{EliasWilliamsonHodgetheory}), we propose the following positivity conjectures:

\begin{Conjecture} \label{conj:intropositivity}
	Suppose that the standard basis of $\mathcal{B}$ exists and that $B_w \in \mathcal{H}^{\leq \ell(w)}$ for all $w \in W_\mathrm{aff}^+$.
	\begin{enumerate}
	\item The Laurent polynomials $Q_{x,w} \in \Z[v,v^{-1}]$ in the expansion
	\[ \underline{B}_w = \sum_{x \in W_\mathrm{aff}} Q_{x,w} \cdot \underline{H}_x \]
	have non-negative coefficients for all $x \in W_\mathrm{aff}$ and $w \in W_\mathrm{aff}^+$.
	\item The Laurent polynomials $\mu_{x,y}^z \in \Z[v,v^{-1}]$ in the expansion
	\[ \underline{B}_x \cdot \underline{B}_y = \sum_{z \in W_\mathrm{aff}^+} \mu_{x,y}^z \cdot \underline{B}_z \]
	have non-negative coefficients for all $x,y,z \in W_\mathrm{aff}^+$. 
	\end{enumerate}
\end{Conjecture}

In Section \ref{sec:canonicalbases} below, we prove that the first part of Conjecture \ref{conj:intropositivity} implies the second part.
If true, the conjecture would suggest the existence of a geometric interpretation or a categorification of $\mathcal{B}$ and its canonical basis, possibly related to known geometric interpretations or categorifications of the affine Hecke algebra and the Kazhdan--Lusztig basis (e.g.\ via perverse sheaves on $\Fl_G$ or via Soergel bimodules).
We leave the existence of such a categorification as a further open problem.

\subsection*{The center of \texorpdfstring{$\mathcal{H}$}{H}}

Besides the analogies with Schubert calculus and Kazhdan--Lusztig theory, another reason for out interest in $\mathcal{B}$ is that the canonical basis of $\mathcal{B}$ can be used to study combinatorial properties of the center $Z(\mathcal{H})$ of $\mathcal{H}$ via the chain of inclusions $Z(\mathcal{H}) \subseteq \mathcal{B} \subseteq \mathcal{H}$. 
By a result of Bernstein (see \cite[Theorem 8.1]{LusztigSpherical}), the center of $\mathcal{H}$ can be described as the ring of $W_\mathrm{fin}$-invariants $Z(\mathcal{H}) = \mathcal{A}[\Z\Phi^\vee]^{W_\mathrm{fin}}$ in the group ring of the coroot lattice, and it admits a canonical basis given by the characters $\{ \chi_\lambda \mid \lambda \in \Z \Phi^\vee \text{ dominant} \}$ of the irreducible representation of the complex simple algebraic group $G_\mathrm{ad}^\vee$ of adjoint type with root system $\Phi^\vee$.
To determine the Laurent polynomials $a_{x,\lambda} \in \Z[v,v^{-1}]$ in the expansion
\[ \chi_\lambda = \sum_{x \in W_\mathrm{aff}} a_{x,\lambda} \cdot \underline{H}_x \]
is equivalent to computing graded composition multiplicities in Gaitsgory's central sheaves on $\Fl_G$ \cite{Gaitsgorycentral} (see Section \ref{sec:center} below for more details), a problem which has previously been considered by Görtz--Haines in \cite{GoertzHainesBoundsWeightsNearbyCycles,GoertzHainesJordanHoelderNearbyCycles}.
Assuming the existence of the canonical basis of $\mathcal{B}$, we can also write
\[ \chi_\lambda = \sum_{w \in W_\mathrm{aff}^+} b_{w,\lambda} \cdot \underline{B}_w \]
for certain Laurent polynomials $b_{w,\lambda} \in \Z[v,v^{-1}]$, and writing $\underline{B}_w = \sum_{x \in W_\mathrm{aff}} Q_{x,w} \cdot \underline{H}_x$ as above, it follows that
\[ a_{x,\lambda} = \sum_{w \in W_\mathrm{aff}^+} Q_{x,w} \cdot b_{w,\lambda} \]
for all $x \in W_\mathrm{aff}$.
In Theorem \ref{thm:centercanonicalbasiscoefficients}, we give an explicit formula for the Laurent polynomials $b_{w,\lambda}$ as a twisted sum of spherical and anti-spherical inverse Kazhdan--Lusztig polynomials.
Thus, the problem of computing the Laurent polynomials $a_{x,\lambda}$ can be reduced to computing the Laurent polynomials $Q_{x,w}$.
This, in our opinion, further underlines the importance of understanding the canonical basis of $\mathcal{B}$.

\subsection*{Outline}

We conclude the introduction with a brief outline of the contents of this article.
Section 1 is mainly concerned with setting up our notation and recalling the definitions of affine Weyl groups and (double) affine Hecke algebras.
In Section \ref{sec:pseudocentralizer}, we introduce the pseudo-centralizer subalgebra $\mathcal{B}$ and prove some of its elementary properties, mostly generalizing results of \cite{LamSchubert,LamSchillingShimozono}.
We then begin the study of the standard basis of $\mathcal{B}$ in Section \ref{sec:standardbasis}, proving its uniqueness (if it exists), and in Section \ref{sec:typeA}, we prove the existence of the standard basis of $\mathcal{B}$ for root systems of type $\mathsf{A}_n$.
In Section \ref{sec:canonicalbases}, we prove the existence of the canonical basis of $\mathcal{B}$ (assuming the existence of the standard basis) and compute some examples that motivate Conjecture \ref{conj:intropositivity}.
Section \ref{sec:center} explains how $\mathcal{B}$ can be used to study the center of the affine Hecke algebra.
Finally, in Appendix \ref{app:rank2}, we prove that the existence of the standard basis of $\mathcal{B}$ can be reduced to a finite computational problem for any given root system, and we use this reduction to establish the existence of the standard basis for root systems of type $\mathsf{B}_2$ and $\mathsf{G}_2$.

\subsubsection*{Acknowledgements}

I would like to thank Pramod Achar, Stefan Dawydiak, Ben Elias, Thomas Lam, Arun Ram, Peng Shan, Philip Thomas and Geordie Williamson for helpful discussions and comments at different stages of this project.
This work was partially supported by the Swiss National Science Foundation via grant P500PT\_206751.

\section{Affine Weyl group and (double) affine Hecke algebra}
\label{sec:Heckealgebras}

In this section, we set up our notation and recall some important results on affine Weyl groups and the corresponding affine and double affine Hecke algebras.
The main reference is \cite{CherednikDAHA}, although our normalization of Hecke algebras is that of \cite{SoergelKL} (because we find it better adapted to discussing canonical bases in Section \ref{sec:canonicalbases} below).

\subsection{}

Let $\Phi$ be an irreducible root system in a Euclidean space $X_\R$ with inner product $(-\,,-)$.
For $\alpha \in \Phi$, we write $\alpha^\vee = 2\alpha / (\alpha,\alpha)$ for the coroot corresponding to $\alpha$ and $s_\alpha \in \GL(X_\R)$ for the reflection corresponding to $\alpha$, with $s_\alpha(x) = x - (x,\alpha^\vee) \cdot \alpha$ for $x \in X_\R$.
We further denote by $\Phi^\vee = \{ \alpha^\vee \mid \alpha \in \Phi \}$ the dual root system and by $W_\mathrm{fin} = \langle s_\alpha \mid \alpha \in \Phi \rangle$ the finite Weyl group of $\Phi$.
Let us fix a base $\Pi \subseteq \Phi$ and write $\Phi^+ \subseteq \Phi$ for the corresponding positive system and $S_\mathrm{fin} = \{ s_\alpha \mid \alpha \in \Pi \}$ for the corresponding set of simple reflections in $W_\mathrm{fin}$.
For $\alpha \in \Pi$ and $s = s_\alpha \in S_\mathrm{fin}$, we also write $\alpha_s = \alpha$.
The fundamental dominant weight $\varpi_\alpha$ corresponding to $\alpha \in \Pi$ is defined via $(\varpi_\alpha,\beta^\vee) = \delta_{\alpha,\beta}$ for all $\beta \in \Pi$, and the root lattice and weight lattice of $\Phi$ are given by
\[ Q = \bigoplus_{\alpha \in \Pi} \Z \alpha , \hspace{2cm} X = \bigoplus_{\alpha \in \Pi} \Z \varpi_\alpha . \]
Analogously, we define the fundamental dominant coweights $\varpi_\alpha^\vee \in X_\R$ via $(\beta,\varpi_\alpha^\vee) = \delta_{\alpha,\beta}$ for $\alpha,\beta \in \Phi$, and we consider the coroot lattice and the coweight lattice defined via
\[ Q^\vee = \bigoplus_{\alpha \in \Pi} \Z \alpha^\vee , \hspace{2cm} Y = \bigoplus_{\alpha \in \Pi} \Z \varpi_\alpha^\vee . \]
We will often fix an indexing set $I_\mathrm{fin}$ such that $\Pi = \{ \alpha_i \mid i \in I_\mathrm{fin} \}$ and write $s_{\alpha_i} = s_i$ and $\varpi_{\alpha_i} = \varpi_i$ for $i \in I_\mathrm{fin}$ accordingly.

\subsection{}
\label{subsec:WaffWext}

The \emph{affine Weyl group} and the \emph{extended affine Weyl group} are defined via
\[ W_\mathrm{aff} = W_\mathrm{fin} \ltimes Q^\vee , \hspace{2cm} W_\mathrm{ext} = W_\mathrm{fin} \ltimes Y , \]
and we write $\gamma \mapsto t_\gamma$ for the canonical embedding $Y \to W_\mathrm{fin} \ltimes Y = W_\mathrm{ext}$.
The affine Weyl group is generated by the (affine) reflections $s_{\alpha,r} = t_{r\alpha^\vee} s_\alpha$, for $\alpha \in \Phi$ and $r \in \Z$.
Let $\alpha_\mathrm{h} \in \Phi^+$ be the highest root and consider the affine reflection
\[ s_0 = s_{\alpha_\mathrm{h},1} = t_{\alpha_\mathrm{h}^\vee} s_{\alpha_\mathrm{h}} \in W_\mathrm{aff} . \]
Then $W_\mathrm{aff}$ is a Coxeter group with simple reflections $S_\mathrm{aff} = S_\mathrm{fin} \cup \{ s_0 \}$.
For $s,t \in W_\mathrm{aff}$, we write $m_{st}$ for the order of the product $st$.
Furthermore, $W_\mathrm{aff}$ is a normal subgroup in $W_\mathrm{ext}$, and the action of $W_\mathrm{ext}$ on $W_\mathrm{aff}$ by conjugation preserves the set of reflections in $W_\mathrm{aff}$.
The subgroup
\[ \Omega = \{ \omega \in W_\mathrm{ext} \mid \omega S_\mathrm{aff} \omega^{-1} = S_\mathrm{aff} \} \]
of $W_\mathrm{ext}$ is a complement of $W_\mathrm{aff}$ in $W_\mathrm{ext}$, so we have $W_\mathrm{ext} \cong W_\mathrm{aff} \rtimes \Omega$ and
\[ \Omega \cong W_\mathrm{ext} / W_\mathrm{aff} \cong Y / Q^\vee . \]
The length function $\ell \colon W_\mathrm{aff} \to \Z_{\geq 0}$ can be extended to a length function $\ell \colon W_\mathrm{ext} \to \Z_{\geq 0}$ via $\ell(x\omega) = \ell(x)$ for $x \in W_\mathrm{aff}$ and $\omega \in \Omega$, so $\Omega = \{ x \in W_\mathrm{ext} \mid \ell(x) = 0 \}$, and we also extend the Bruhat order on $W_\mathrm{aff}$ to a partial order on $W_\mathrm{ext}$ via $x \omega \leq x' \omega'$ if and only if $x \leq x'$ and $\omega = \omega'$, for $x,x' \in W_\mathrm{aff}$ and $\omega,\omega' \in \Omega$.

\subsection{}
\label{subsec:affineweights}

We define an action of $W_\mathrm{ext}$ on the space $\tilde X_\R = X_\R \oplus \R \delta$ via
\[ w t_\gamma( \lambda + a \delta ) = w(\lambda) + \big( a - ( \lambda , \gamma ) \big) \cdot \delta \]
for all $w \in W_\mathrm{fin}$, $\gamma \in Y$, $\lambda \in X_\R$ and $a \in \R$.
Let $m$ be the least positive integer such that $m \cdot (\lambda,\gamma) \in \Z$ for all $\lambda \in X$ and  $\gamma \in Y$, and observe that
\begin{itemize}
	\item the action of $W_\mathrm{ext}$ preserves the lattices $\tilde X_m = X \oplus \Z \frac{\delta}{m}$ and $\tilde Q = Q \oplus \Z \delta$ in $\tilde X_\R$,
	\item the action of $W_\mathrm{aff}$ preserves the lattice $\tilde X = X \oplus \Z \delta$ in $\tilde X_\R$.
\end{itemize}
Alternatively, $m$ can be defined as the exponent of $Y / Q^\vee$, so $m=2$ if $\Phi$ is of type $\mathsf{D}_{2n}$ for some $n \geq 2$ and $m = \abs{ Y / Q^\vee }$ for all other root systems.
The affine simple root corresponding to $s_0$ is
\[ \alpha_0 = \alpha_{s_0} \coloneqq \delta - \alpha_\mathrm{h} \in \tilde Q , \]
and the set of simple roots of $W_\mathrm{aff}$ is given by
\[ \Pi_\mathrm{aff} = \Pi \cup \{ \alpha_0 \} = \{ \alpha_s \mid s \in S_\mathrm{aff} \} . \]
We also extend the inner product $( - \,, - )$ on $X_\R$ to a symmetric bilinear form on $\tilde X_\R$ via $(\delta,\lambda) = 0$ for all $\lambda \in \tilde X_\R$, so that the coroot corresponding to $\alpha_0$ is
\[ \alpha_0^\vee = \alpha_{s_0}^\vee \coloneqq 2 \alpha_0 / (\alpha_0,\alpha_0) = - \alpha_\mathrm{h}^\vee . \]

\subsection{}
\label{subsec:Heckealgebras}

Now we introduce different Hecke algebras over the Laurent polynomial ring $\mathcal{A} = \Z[v,v^{-1}]$.

\begin{Definition}
	The affine Hecke algebra $\mathcal{H}$ is the $\mathcal{A}$-algebra with generators
	\[ \{ H_s \mid s \in S_\mathrm{aff} \} , \]
	subject to the \emph{quadratic relations}
	\begin{equation} \label{eq:relationsquadratic}
	( H_s + v ) ( H_s - v^{-1} ) = 0
	\end{equation}
	and the \emph{braid relations}
	\begin{equation} \label{eq:relationsbraid}
	\underbrace{H_s H_t H_s \cdots}_{m_{st} \text{ factors}} = \underbrace{H_t H_s H_t \cdots}_{m_{st} \text{ factors}} ,
	\end{equation}
	for $s,t \in S_\mathrm{aff}$.
\end{Definition}

\begin{Remark}
	For an element $w \in W_\mathrm{aff}$ with reduced expression $w = s_1 \cdots s_n$, the braid relations \eqref{eq:relationsbraid} imply that the element $H_w = H_{s_1} \cdots H_{s_n} \in \mathcal{H}$ is independent of the choice of reduced expression.
	The elements $\{ H_w \mid w \in W_\mathrm{aff} \}$ form an $\mathcal{A}$-basis of $\mathcal{H}$, which we call the \emph{standard basis}.
\end{Remark}

\begin{Definition}
	The extended affine Hecke algebra $\mathcal{H}_\mathrm{ext}$ is the $\mathcal{A}$-algebra with generators
	\[ \{ H_s \mid s \in S_\mathrm{aff} \} \cup \{ H_\omega \mid \omega \in \Omega \} , \]
	subject to the quadratic relations \eqref{eq:relationsquadratic}, the braid relations \eqref{eq:relationsbraid}, and the \emph{$\Omega$-relations}
	\begin{equation} \label{eq:relationsomega}
	H_\omega H_{\omega'} = H_{\omega\omega'} , \hspace{2cm} H_\omega H_s H_{\omega^{-1}} = H_{\omega s \omega^{-1}}
	\end{equation}
	for $\omega,\omega' \in \Omega$ and $s \in S_\mathrm{aff}$.
\end{Definition}

\begin{Remark}
	For all $x \in W_\mathrm{ext}$, we can uniquely write $x = w \omega$ with $w \in W_\mathrm{aff}$ and $\omega \in \Omega$, and we define $H_x = H_w H_\omega$.
	The elements $\{ H_x \mid x \in W_\mathrm{ext} \}$ form the \emph{standard basis} of $\mathcal{H}_\mathrm{ext}$.
	For all $x,y \in W_\mathrm{ext}$ with $\ell(xy) = \ell(x) + \ell(y)$, we have $H_{xy} = H_x H_y$.
\end{Remark}

\begin{Definition}
	The double affine Hecke algebra $\mathbb{H}$ is the $\mathcal{A}$-algebra with generators
	\[ \{ H_s \mid s \in S_\mathrm{aff} \} \cup \{ X_\lambda \mid \lambda \in \tilde X \} , \]
	subject to the quadratic relations \eqref{eq:relationsquadratic}, the braid relations \eqref{eq:relationsbraid}, the \emph{weight relations}
	\begin{equation} \label{eq:relationsweight}
	X_\lambda X_\mu = X_{\lambda+\mu}
	\end{equation}
	for $\lambda,\mu \in \tilde X$, and the \emph{Bernstein relations}
	\begin{equation} \label{eq:relationsDL}
	H_s X_\lambda = X_{s(\lambda)} H_s + ( v - v^{-1} ) \cdot \frac{ X_\lambda - X_{s(\lambda)} }{ X_{\alpha_s} - 1 } 
	\end{equation}
	for $\lambda \in \tilde X$ and $s \in S_\mathrm{aff}$.
\end{Definition}

\begin{Remark} \label{rem:Demazureoperator}
	Note that $X_\lambda - X_{s(\lambda)}$ is a multiple of $X_{\alpha_s} - 1$ for all $\lambda \in \tilde X$ and $s \in S_\mathrm{aff}$, so that the expression on the right hand side of \eqref{eq:relationsDL} is well-defined.
	Indeed, if $r = (\lambda,\alpha_s^\vee) \geq 0$ then we have
	\[ X_\lambda - X_{s(\lambda)} = ( X_{\alpha_s} - 1 ) \cdot ( X_{\lambda - \alpha_s} + X_{\lambda - 2 \alpha_s} + \cdots + X_{\lambda - r \alpha_s} ) , \]
	and if $r = (\lambda,\alpha_s^\vee) \leq 0$ then
	\[ X_\lambda - X_{s(\lambda)} = - ( X_{\alpha_s} - 1 ) \cdot ( X_\lambda + X_{\lambda + \alpha_s} + \cdots + X_{\lambda + (1-r) \alpha_s} ) . \]
	For the expression in \eqref{eq:relationsDL}, we obtain
	\[ \frac{ X_\lambda - X_{s(\lambda)} }{ X_{\alpha_s} - 1 } = \begin{cases}
	X_\lambda \cdot ( X_{-\alpha_s} + X_{-2\alpha_s} + \cdots + X_{-r\alpha_s} ) & \text{if } r = (\lambda,\alpha_s^\vee) \geq 0 \\
	- X_\lambda \cdot ( 1 + X_{\alpha_s} + \cdots + X_{ - (r+1) \cdot \alpha_s} ) & \text{if } r = (\lambda,\alpha_s^\vee) \leq 0 .
	\end{cases} \]
\end{Remark}

\begin{Remark} \label{rem:PBWbasis}
	By the weight relations \eqref{eq:relationsweight}, the elements $X_\lambda$ for $\lambda \in \tilde X$ span a commutative $\mathcal{A}$-subalgebra $\mathcal{A}[\tilde X] = \bigoplus_\lambda \mathcal{A} X_\lambda$ in $\mathbb{H}$.
	The multiplication in $\mathbb{H}$ induces isomorphisms of $\mathcal{A}$-modules
	\[ \mathbb{H} \cong \mathcal{A}[\tilde X] \otimes_\mathcal{A} \mathcal{H} \cong \mathcal{H} \otimes_\mathcal{A} \mathcal{A}[\tilde X] ; \]
	this result (or variations thereof) is often referred to as the \emph{Poincaré--Birckhoff--Witt theorem} for double affine Hecke algebras.
	In particular, the elements $\{ X_\lambda \cdot H_x \mid \lambda \in \tilde X , x \in W_\mathrm{ext} \}$ form an $\mathcal{A}$-basis of $\mathbb{H}$.
	
	It is customary to set $q = X_\delta$ and view $\mathbb{H}$ as an algebra over $\mathcal{A}_q = \mathcal{A}[q^{\pm 1}] = \Z[v^{\pm 1},q^{\pm 1}]$.
	As such, an $\mathcal{A}_q$-basis of $\mathbb{H}$ is given by $\{ X_\lambda \cdot H_x \mid \lambda \in X , x \in W_\mathrm{ext} \}$.
	This remains true in the specialization $q \mapsto 1$, which we will consider later on.
\end{Remark}

\begin{Definition}
	The extended double affine Hecke algebra $\mathbb{H}_\mathrm{ext}$ is the $\mathcal{A}$-algebra with generators
	\[ \{ H_s \mid s \in S_\mathrm{aff} \} \cup \{ H_\omega \mid \omega \in \Omega \} \cup \{ X_\lambda \mid \lambda \in \tilde X_m \} , \]
	subject to the quadratic relations \eqref{eq:relationsquadratic}, the braid relations \eqref{eq:relationsbraid}, the $\Omega$-relations \eqref{eq:relationsomega}, the weight relations \eqref{eq:relationsweight} (for $\lambda,\mu \in \tilde X_m$) the Bernstein relations \eqref{eq:relationsDL} (for $\lambda \in \tilde X_m$ and $s \in S_\mathrm{aff}$), and the \emph{$\Omega$-weight relations}
	\begin{equation} \label{eq:relationsomegaweight}
	H_\omega X_\lambda H_{\omega^{-1}} = X_{\omega(\lambda)} ,
	\end{equation}
	for $\lambda \in \tilde X_m$ and $\omega \in \Omega$.
\end{Definition}

\begin{Remark} \label{rem:omegainHext}
	In view of the $\Omega$-relations \eqref{eq:relationsomega}, the elements $\{ H_\omega \mid \omega \in \Omega \}$ form a subgroup isomorphic to $\Omega$ in $\mathbb{H}_\mathrm{ext}$ (and in $\mathcal{H}_\mathrm{ext}$).
	By abuse of notation, we write $H_\omega = \omega$ for $\omega \in \Omega$, so that the relations in \eqref{eq:relationsomega} and \eqref{eq:relationsomegaweight} become
	\[ \omega H_s \omega^{-1} = H_{\omega s \omega^{-1}} , \hspace{2cm} \omega X_\lambda \omega^{-1} = X_{\omega(\lambda)} \]
	for $s \in S_\mathrm{aff}$ and $\lambda \in \tilde X_m$.
\end{Remark}

\begin{Remark} \label{rem:PBWbasisextended}
	As in Remark \ref{rem:PBWbasis}, the multiplication in $\mathbb{H}_\mathrm{ext}$ induces isomorphisms of $\mathcal{A}$-modules
	\[ \mathbb{H}_\mathrm{ext} \cong \mathcal{A}[\tilde X_m] \otimes_\mathcal{A} \mathcal{H}_\mathrm{ext} \cong \mathcal{H}_\mathrm{ext} \otimes_\mathcal{A} \mathcal{A}[\tilde X_m] , \]
	and an $\mathcal{A}$-basis of $\mathbb{H}_\mathrm{ext}$ is given by $\{ X_\lambda \cdot H_x \mid \lambda \in \tilde X_m , x \in W_\mathrm{ext} \}$.
	If we set $q' = X_{\delta/m}$ and view $\mathbb{H}_\mathrm{ext}$ as an algebra over $\mathcal{A}_{q'} = \mathcal{A}[q',q'^{-1}]$ then an $\mathcal{A}_{q'}$-basis of $\mathbb{H}_\mathrm{ext}$ is given by $\{ X_\lambda \cdot H_x \mid \lambda \in X , x \in W_\mathrm{ext} \}$, and this remains true in the specialization $q' \mapsto 1$.
\end{Remark}

\section{The pseudo-centralizer subalgebra}
\label{sec:pseudocentralizer}

In this section, we introduce the main object of interest in this article, namely, a subalgebra $\mathcal{B}$ of the affine Hecke algebra $\mathcal{H}$ that arises from a centralizer construction in the double affine Hecke algebra.
We also establish some elementary properties of $\mathcal{B}$ and exhibit some explicit elements.

\subsection{}

From now on and for the remainder of this article, we consider the double affine Hecke algebra $\mathbb{H}$ and the extended double affine Hecke algebra $\mathbb{H}_\mathrm{ext}$ with the parameters $q = X_\delta$ and $q' = X_{\delta/m}$ specialized to $1$ (see Remarks \ref{rem:PBWbasis} and \ref{rem:PBWbasisextended}).
Thus, $\mathbb{H}$ and $\mathbb{H}_\mathrm{ext}$ are algebras over the Laurent polynomial ring $\mathcal{A} = \Z[v,v^{-1}]$, and the multiplication in $\mathbb{H}$ and $\mathbb{H}_\mathrm{ext}$ induces isomorphisms of $\mathcal{A}$-modules
\[ \mathbb{H} \cong \mathcal{A}[X] \otimes_\mathcal{A} \mathcal{H} , \hspace{2cm} \mathbb{H}_\mathrm{ext} \cong \mathcal{A}[X] \otimes_\mathcal{A} \mathcal{H}_\mathrm{ext} . \]
In particular, an $\mathcal{A}$-basis of $\mathbb{H}$ is given by $\{ X_\lambda H_w \mid \lambda \in X , w \in W_\mathrm{aff} \}$, and an $\mathcal{A}$-basis of $\mathbb{H}_\mathrm{ext}$ is given by $\{ X_\lambda H_x \mid \lambda \in X , x \in W_\mathrm{ext} \}$.

\subsection{}

Every element $x \in \mathbb{H}_\mathrm{ext}$ can be written uniquely in the form
\[ x = \sum_{\mu \in X} X_\mu \cdot c_\mu(x) , \]
with $c_\mu(x) \in \mathcal{H}_\mathrm{ext}$ for all $\mu \in X$, and if $x \in \mathbb{H}$ then we further have $c_\mu(x) \in \mathcal{H}$ for all $\mu \in X$.
We consider the $\mathcal{A}$-linear map
\begin{equation} \label{eq:phi}
	\varphi \colon \mathbb{H}_\mathrm{ext} \longrightarrow \mathcal{H}_\mathrm{ext} , \qquad x \longmapsto \sum_{\mu \in Y} c_\mu(x) ,
\end{equation}
which restricts to a map $\varphi \colon \mathbb{H} \to \mathcal{H}$.
In other words, $\varphi \colon \mathbb{H}_\mathrm{ext} \to \mathcal{H}_\mathrm{ext}$ is the unique $\mathcal{A}$-linear map with $\varphi( X_\lambda h ) = h$ for all $\lambda \in X$ and $h \in \mathcal{H}_\mathrm{ext}$.
Further consider the $\mathcal{A}$-submodule
\[ \mathcal{B}_\mathrm{ext} \coloneqq \Big\{ h \in \mathcal{H}_\mathrm{ext} \mathrel{\Big|} \varphi(h X_\lambda) = h \text{ for all } \lambda \in X \Big\} \subseteq \mathcal{H}_\mathrm{ext} . \]
Observe that for $\lambda \in X$ and $h \in \mathbb{H}_\mathrm{ext}$, we have $\varphi(X_\lambda h) = h$, whence we can equivalently define $\mathcal{B}_\mathrm{ext}$ via
\[ \mathcal{B}_\mathrm{ext} = \Big\{ h \in \mathcal{H}_\mathrm{ext} \mathrel{\Big|} \varphi(h X_\lambda - X_\lambda h) = 0 \text{ for all } \lambda \in X \Big\} . \]
In view of this description, we refer to $\mathcal{B}_\mathrm{ext}$ as the \emph{pseudo-centralizer} of $X$ in $\mathcal{H}_\mathrm{ext}$.

\begin{Lemma} \label{lem:Bsubalgebra}
	The pseudo-centralizer $\mathcal{B}_\mathrm{ext} \subseteq \mathcal{H}_\mathrm{ext}$ is an $\mathcal{A}$-subalgebra.
\end{Lemma}
\begin{proof}
	For all $g,h \in \mathbb{H}_\mathrm{ext}$ and $\lambda \in X$, we have
	\[ g h X_\lambda = \sum_\mu g X_\mu c_\mu(hX_\lambda) = \sum_{\mu,\nu} X_\nu c_\nu(g X_\mu) c_\mu(hX_\lambda) \]
	and therefore $c_\nu(gh X_\lambda) = \sum_\mu c_\nu(g X_\mu) c_\mu(hX_\lambda)$ for all $\nu \in X$.
	For $g,h \in \mathcal{B}_\mathrm{ext}$, we obtain
	\[ \sum_\nu c_\nu(g h X_\lambda) = \sum_{\mu,\nu} c_\nu(g X_\mu) c_\mu(hX_\lambda) = \sum_\mu g c_\mu(hX_\lambda) = gh , \]
	and it follows that $gh \in \mathcal{B}_\mathrm{ext}$, as required.
\end{proof}

We further define the pseudo-centralizer of $X$ in $\mathcal{H}$ via
\[ \mathcal{B} \coloneqq \Big\{ h \in \mathcal{H} \mathrel{\Big|} \varphi(h X_\lambda) = h \text{ for all } \lambda \in X \Big\} = \mathcal{B}_\mathrm{ext} \cap  \mathcal{H} \subseteq \mathcal{H} . \]
It is obvious from Lemma \ref{lem:Bsubalgebra} that $\mathcal{B}$ is an $\mathcal{A}$-subalgebra of $\mathcal{H}$.

\subsection{}

Recall from Remark \ref{rem:omegainHext} that the elements of length zero in $W_\mathrm{ext}$ afford a subgroup $\Omega$ in the group of units of $\mathbb{H}_\mathrm{ext}$.
Using the relations \eqref{eq:relationsomega} and \eqref{eq:relationsomegaweight}, it is straightforward to see that we have $\mathcal{H} \omega = \omega \mathcal{H}$ and $\mathbb{H} \omega = \omega \mathbb{H}$ for all $\omega \in \Omega$, and so $\mathcal{H}_\mathrm{ext}$ and $\mathbb{H}_\mathrm{ext}$ are $\Omega$-graded $\mathcal{A}$-algebras with
\begin{equation} \label{eq:omegagrading}
	\mathcal{H}_\mathrm{ext} = \bigoplus_{\omega \in \Omega} \mathcal{H} \omega , \hspace{2cm} \mathbb{H}_\mathrm{ext} = \bigoplus_{\omega \in \Omega} \mathbb{H} \omega .
\end{equation}
Furthermore, it is straightforward to see that the map $\varphi \colon \mathbb{H}_\mathrm{ext} \to \mathcal{H}_\mathrm{ext}$ is homogeneous with respect to the $\Omega$-grading, that is, we have $\varphi( \mathbb{H} \omega ) \subseteq \mathcal{H} \omega$ for all $\omega \in \Omega$.
More precisely, we have the following compatibility of $\varphi$ with the action of $\Omega$.

\begin{Lemma} \label{lem:phiomega}
	For all $x \in \mathbb{H}_\mathrm{ext}$ and $\omega,\omega' \in \Omega$, we have $\varphi(\omega x \omega') = \omega \varphi(x) \omega'$.
\end{Lemma}
\begin{proof}
	Using \eqref{eq:relationsomegaweight} and the definition of $\varphi$ in \eqref{eq:phi}, we compute
	\[ \varphi( \omega x \omega' ) = \sum_{\mu \in X} \varphi\big( \omega X_\mu c_\mu(x) \omega' \big) = \sum_{\mu \in X} \varphi\big( X_{\omega(\mu)} \omega c_\mu(x) \omega' \big) = \sum_{\mu \in X} \omega c_\mu(x) \omega' = \omega \varphi(x) \omega' , \]
	as claimed.
\end{proof}

\begin{Lemma} \label{lem:Bomega}
	We have $\Omega \subseteq \mathcal{B}_\mathrm{ext}$, and $\mathcal{B}_\mathrm{ext}$ is an $\Omega$-graded $\mathcal{A}$-algebra with
	\[ \mathcal{B}_\mathrm{ext} = \bigoplus_{\omega \in \Omega} \mathcal{B} \omega . \]
\end{Lemma}
\begin{proof}
	For $\omega \in \Omega$ and $\lambda \in X$, we have $\varphi(\omega X_\lambda) = \omega \varphi(X_\lambda) = \omega$ by Lemma \ref{lem:phiomega}, and therefore $\omega \in \mathcal{B}_\mathrm{ext}$.
	As we also have $\mathcal{B} \subseteq \mathcal{B}_\mathrm{ext}$, it follows that $\bigoplus_\omega \mathcal{B} \omega \subseteq \mathcal{B}_\mathrm{ext}$.
	
	Now let $h \in \mathcal{B}_\mathrm{ext}$ and write $h = \sum_\omega h_\omega \omega$ with $h_\omega \in \mathcal{H}$ for all $\omega \in \Omega$.
	For all $\lambda \in X$, we have
	\[ \sum_{\omega \in \Omega} h_\omega \omega = h = \varphi( h X_\lambda ) = \sum_{\omega \in \Omega} \varphi( h_\omega \omega X_\lambda ) , \]
	where $\varphi( h_\omega \omega X_\lambda ) \in \mathcal{H} \omega$ for $\omega \in \Omega$ because $h_\omega \omega X_\lambda = h_\omega X_{\omega(\lambda)} \omega \in \mathbb{H} \omega$ by \eqref{eq:relationsomegaweight}.
	Using the $\Omega$-grading from \eqref{eq:omegagrading}, it follows that $\varphi( h_\omega \omega X_\lambda ) = h_\omega \omega$ for all $\lambda \in X$ and $\omega \in \Omega$, so $h_\omega \omega \in \mathcal{B}_\mathrm{ext} \cap \mathcal{H} \omega$ and
	\[ h_\omega = h_\omega \omega \cdot \omega^{-1} \in \mathcal{B}_\mathrm{ext} \cap \mathcal{H} = \mathcal{B} . \]
	In particular, we have $h = \sum_\omega h_\omega \omega \in \bigoplus_\omega \mathcal{B} \omega$, as required.
\end{proof}

In the Sections \ref{sec:standardbasis}, \ref{sec:typeA} and \ref{sec:canonicalbases} below, we will mostly restrict our attention to the pseudo-centralizer $\mathcal{B}$ in $\mathcal{H}$, but most of our results can easily be generalized to $\mathcal{B}_\mathrm{ext}$ using Lemma \ref{lem:Bomega}.
We will only spell this out explicitly when the results about $\mathcal{B}_\mathrm{ext}$ are relevant for the rest of this article.

\subsection{}

For computational purposes, it is useful to record the following result.

\begin{Lemma} \label{lem:BgeneratingsetX}
	Let $\Gamma \subseteq X$ be a generating set.
	Then we have
	\[ \mathcal{B}_\mathrm{ext} = \Big\{ h \in \mathcal{H}_\mathrm{ext} \mathrel{\Big|} \varphi(h X_\lambda) = h \text{ for all } \lambda \in \Gamma \Big\} \]
\end{Lemma}
\begin{proof}
	Let $\lambda,\mu \in X$ and $h \in \mathcal{H}_\mathrm{ext}$, and suppose that $\varphi(h X_\lambda) = h$ and $\varphi(h X_\mu) = h$.
	We claim that $\varphi(h X_{\lambda+\mu}) = h$ and $\varphi(h X_{-\lambda}) = h$.
	Indeed, observe that we have
	\[ h = h X_\lambda X_{-\lambda} = \sum_\nu X_\nu c_\nu(h X_\lambda) X_{-\lambda} = \sum_{\nu,\delta} X_{\nu+\delta} c_\delta\big( c_\nu(h X_\lambda) X_{-\lambda} \big) \]
	and therefore
	\[ \sum_{\nu+\delta = \vartheta} c_\delta\big( c_\nu( h X_\lambda) X_{-\lambda} \big) = \begin{cases} h & \text{if } \vartheta = 0 , \\ 0 & \text{otherwise} ,  \end{cases} \]
	for all $\vartheta \in X$.
	This implies that
	\begin{multline*}
		\qquad h = \sum_{\vartheta \in X} \sum_{\nu+\delta=\vartheta} c_\delta\big( c_\nu( h X_\lambda) X_{-\lambda} \big) = \sum_{\nu,\delta \in X} c_\delta\big( c_\nu( h X_\lambda ) X_{-\lambda} \big) \\
		 = \sum_{\delta \in X} c_\delta\big( \varphi( h X_\lambda) X_{-\lambda} \big) = \sum_{\delta \in X} c_\delta(h X_{-\lambda}) = \varphi( h X_{-\lambda} ) , \qquad
	\end{multline*}
	as claimed.
	Similarly, we have
	\[ \sum_\vartheta X_\vartheta c_\vartheta(h X_{\lambda+\mu}) = h X_{\lambda+\mu} = h X_\lambda X_\mu = \sum_\nu X_\nu c_\nu(h X_\lambda) X_\mu = \sum_{\nu,\delta} X_{\nu+\delta} c_\delta( c_\nu(h X_\lambda) X_\mu ) , \]
	and therefore
	\[ c_\vartheta( h X_{\lambda+\mu} ) = \sum_{\nu+\delta=\vartheta} c_\delta\big( c_\nu(h X_\lambda) X_\mu \big) . \]
	This implies that
	\begin{multline*}
		\quad \varphi(h X_{\lambda+\mu}) = \sum_\vartheta c_\vartheta(X_{\lambda+\mu}) = \sum_\vartheta \sum_{\nu+\delta=\vartheta} c_\delta\big( c_\nu(h X_\lambda) X_\mu \big) = \sum_{\nu,\delta} c_\delta\big( c_\nu(h X_\lambda) X_\mu \big) \\
		= \sum_\delta c_\delta\big( \varphi(h X_\lambda) X_\mu \big) = \sum_\delta c_\delta\big( h X_\mu \big) = \varphi(h X_\mu) = h , \quad
	\end{multline*}
	as claimed.
	We conclude that for all $h \in \mathcal{H}_\mathrm{ext}$, the set $X_h = \{ \lambda \in X \mid \varphi(h X_\lambda) = h \}$ is a subgroup of $X$.
	Since $\Gamma$ is a generating set of $X$, we have $\Gamma \subseteq X_h$ if and only if $X_h = X$, and it follows that
	\[ \Big\{ h \in \mathcal{H}_\mathrm{ext} \mathrel{\Big|} \varphi(h X_\lambda) = h \text{ for all } \lambda \in \Gamma \Big\} = \{ h \in \mathcal{H}_\mathrm{ext} \mid \Gamma \subseteq X_h \} = \{ h \in \mathcal{H}_\mathrm{ext} \mid X_h = X \} = \mathcal{B}_\mathrm{ext} , \]
	as required.
\end{proof}

Using the preceding lemma, we can exhibit some explicit elements of $\mathcal{B}$.
Recall that $\alpha_\mathrm{h} \in \Phi$ denotes the highest root, and for $\alpha \in \Pi$, let $c_\alpha = (\varpi_\alpha,\alpha_\mathrm{h}^\vee)$, so that $\alpha_\mathrm{h}^\vee = \sum_{\alpha \in \Pi} c_\alpha \alpha^\vee$.

\begin{Lemma} \label{lem:Bs0}
	We have
	\[ B_{s_0} \coloneqq H_{s_0} + \sum_{\alpha \in \Pi} c_\alpha H_{s_\alpha} \in \mathcal{B} . \]
\end{Lemma}
\begin{proof}
	For all $\alpha \in \Pi$, we have
	\[ H_{s_\alpha} X_{\varpi_\alpha} = X_{\varpi_\alpha-\alpha} H_{s_\alpha} + (v-v^{-1}) \cdot X_{\varpi_\alpha-\alpha} \]
	by \eqref{eq:relationsDL} and Remark \ref{rem:Demazureoperator}, and it follows that $\varphi( H_{s_\alpha} X_{\varpi_\alpha} ) = H_{s_\alpha} + (v-v^{-1})$.
	Fur $\beta \in \Pi$ with $\beta \neq \alpha$, we have $(\beta,\alpha^\vee) = 0$, so $H_{s_\alpha} X_{\varpi_\beta} = X_{\varpi_\beta} H_{s_\alpha}$ and $\varphi(H_{s_\alpha} X_{\varpi_\beta}) = H_{s_\alpha}$.
	Furthermore, we have
	\[ H_{s_0} X_{\varpi_\beta} = X_{ \varpi_\beta - c_\beta \alpha_\mathrm{h} } H_{s_0} - (v-v^{-1}) \cdot X_{\varpi_\beta} \cdot ( 1 + X_{-\alpha_\mathrm{h}} + \cdots + X_{ - ( c_\beta - 1 ) \cdot \alpha_\mathrm{h}} ) \]
	and therefore $\varphi( H_{s_0} X_{\varpi_\beta} ) = H_{s_0} - c_\beta \cdot ( v - v^{-1} )$.
	Thus we have
	\[ \varphi( B_{s_0} X_{\varpi_\beta} ) = H_{s_0} - c_\beta \cdot ( v - v^{-1} ) + c_\beta \cdot \big( H_{s_\beta} + ( v - v^{-1} ) \big) + \sum_{\alpha \neq \beta} c_\alpha H_{s_\alpha} = B_{s_0} \]
	for all $\beta \in \Pi$.
	As the fundamental dominant weights generate $X$, Lemma \ref{lem:BgeneratingsetX} implies that $B_{s_0} \in \mathcal{B}$, as claimed.
\end{proof}

\subsection{}

Now we consider the centralizer
\[ \mathcal{C}_\mathrm{ext} = C_{\mathbb{H}_\mathrm{ext}}(X) = \{ h \in \mathbb{H}_\mathrm{ext} \mid h X_\lambda = X_\lambda h \text{ for all } \lambda \in X \} \]
of $\mathcal{A}[X]$ in $\mathbb{H}_\mathrm{ext}$.
Recall the map $\varphi \colon \mathbb{H}_\mathrm{ext} \to \mathcal{H}_\mathrm{ext}$ defined in \eqref{eq:phi}.

\begin{Lemma} \label{lem:Bcentralizer}
	The restriction of $\varphi$ to $\mathcal{C}_\mathrm{ext}$ is an $\mathcal{A}$-algebra homomorphism with $\varphi(\mathcal{C}_\mathrm{ext}) \subseteq \mathcal{B}_\mathrm{ext}$.
\end{Lemma}
\begin{proof}
	For $x \in \mathcal{C}_\mathrm{ext}$ and $\lambda \in X$, we have
	\[ \sum_\mu X_{\lambda+\mu} c_\mu(x) = X_\lambda x = x X_\lambda = \sum_\nu X_\nu c_\nu(x) X_\lambda = \sum_{\nu,\delta} X_{\nu+\delta} c_\delta\big( c_\nu(x) X_\lambda \big) , \]
	and it follows that $c_\mu(x) = \sum_{\nu+\delta=\lambda+\mu} c_\delta\big( c_\nu(x) X_\lambda \big)$ for all $\lambda,\mu \in X$.
	This implies that
	\[ \varphi\big( \varphi(x) X_\lambda \big)  = \sum_\delta c_\delta\big( \varphi(x) X_\lambda \big) = \sum_{\nu,\delta} c_\delta\big( c_\nu(x) X_\lambda \big) = \sum_\mu c_\mu(x) = \varphi(x) , \]
	and we conclude that $\varphi(x) \in \mathcal{B}_\mathrm{ext}$.
	For $x \in \mathcal{C}_\mathrm{ext}$ and $y \in \mathbb{H}_\mathrm{ext}$, we have
	\[ xy = \sum_\nu x X_\nu c_\nu(y) = \sum_\nu X_\nu x c_\nu(y) = \sum_{\nu,\delta} X_{\nu+\delta} c_\delta(x) c_\nu(y) \]
	and therefore $c_\mu(xy) = \sum_{\nu+\delta=\mu} c_\delta(x) c_\nu(y)$ for all $\mu \in X$,
	and it follows that
	\[ \varphi(xy) = \sum_\mu c_\mu(xy) = \sum_{\nu,\delta} c_\delta(x) c_\nu(y) = \sum_\delta c_\delta(x) \cdot \sum_\nu c_\nu(y) = \varphi(x) \cdot \varphi(y) , \]
	as required.
\end{proof}

\begin{Remark}
	We can view $\mathcal{H}_\mathrm{ext}$ as a left $\mathcal{B}_\mathrm{ext}$-module via left-multiplication, and by Lemma \ref{lem:Bcentralizer}, this induces a left $\mathcal{C}_\mathrm{ext}$-module structure on $\mathcal{H}_\mathrm{ext}$ via the homomorphism $\varphi \colon \mathcal{C}_\mathrm{ext} \to \mathcal{B}_\mathrm{ext}$.
	The proof of Lemma \ref{lem:Bcentralizer} shows that $\varphi \colon \mathbb{H}_\mathrm{ext} \to \mathcal{H}_\mathrm{ext}$ is a homomorphism of left $\mathcal{C}_\mathrm{ext}$-modules.
\end{Remark}

\subsection{}
\label{subsec:Bcontainscenter}

To conclude this section, we show that the subalgebra $\mathcal{B}_\mathrm{ext}$ of $\mathcal{H}_\mathrm{ext}$ contains the center of $\mathcal{H}_\mathrm{ext}$.
This fact will be crucial for one of the main applications of $\mathcal{B}_\mathrm{ext}$, which we explain in Section \ref{sec:center} below.
We first recall a classical description of the center of $\mathcal{H}_\mathrm{ext}$ due to Bernstein.

Let us write
\[ Y^+ = \{ \lambda \in Y \mid (\alpha,\lambda) \geq 0 \text{ for all } \alpha \in \Pi \} \]
for the set of dominant coweights.
For every coweight $\lambda \in Y$, we can choose dominant coweights $\mu,\nu \in Y^+$ such that $\lambda = \mu - \nu$, and the element
\[ Y_\lambda \coloneqq H_{t_\mu} \cdot H_{t_\nu}^{-1} \in \mathcal{H}_\mathrm{ext} \]
is independent of the choice of $\mu$ and $\nu$.
Furthermore, we have $Y_\lambda \cdot Y_\mu = Y_{\lambda+\mu}$ for all $\lambda,\mu \in Y$, and so the elements $Y_\lambda$ with $\lambda \in Y$ span a commutative $\mathcal{A}$-subalgebra
\[ \mathcal{A}[Y] = \bigoplus_{\lambda \in Y} \mathcal{A} \cdot Y_\lambda \]
in $\mathcal{H}_\mathrm{ext}$.
Now the center $Z(\mathcal{H}_\mathrm{ext})$ of $\mathcal{H}_\mathrm{ext}$ can be described as the ring of $W_\mathrm{fin}$-invariants in $\mathcal{A}[Y]$, that is
\[ Z(\mathcal{H}_\mathrm{ext}) = \mathcal{A}[Y]^{W_\mathrm{fin}} = \Big\{ \sum\nolimits_\lambda a_\lambda \cdot Y_\lambda \in \mathcal{A}[Y] \mathop{\Big|} a_\lambda = a_{w(\lambda)} \text{ for all } w \in W_\mathrm{fin} \text{ and } \lambda \in Y \Big\} ; \]
see Proposition 3.11 in \cite{LusztigAffineHeckeGraded}.

\begin{Corollary} \label{cor:centerinB}
	The center $Z(\mathcal{H}_\mathrm{ext})$ of $\mathcal{H}_\mathrm{ext}$ is contained in $\mathcal{B}_\mathrm{ext}$.
\end{Corollary}
\begin{proof}
	The center $Z(\mathbb{H}_\mathrm{ext})$ of $\mathbb{H}_\mathrm{ext}$ contains $\mathcal{A}[X]^{W_\mathrm{fin}} \otimes_\mathcal{A} \mathcal{A}[Y]^{W_\mathrm{fin}}$ by Lemma 5.1 in \cite{OblomkovDAHA}.
	By the above discussion, we further have $Z(\mathcal{H}_\mathrm{ext}) = \mathcal{A}[Y]^{W_\mathrm{fin}}$, and so the embedding $\mathcal{H}_\mathrm{ext} \subseteq \mathbb{H}_\mathrm{ext}$ gives rise to an embedding
	\[ Z(\mathcal{H}_\mathrm{ext}) \subseteq Z(\mathbb{H}_\mathrm{ext}) \subseteq \mathcal{C}_\mathrm{ext} . \]
	As the map $\varphi \colon \mathbb{H}_\mathrm{ext} \to \mathcal{H}_\mathrm{ext}$ restricts to the identity on the subalgebra $\mathcal{H}_\mathrm{ext} \subseteq \mathbb{H}_\mathrm{ext}$, we conclude that
	\[ Z( \mathcal{H}_\mathrm{ext} ) = \varphi\big( Z(\mathcal{H}_\mathrm{ext}) \big) \subseteq \varphi( \mathcal{C}_\mathrm{ext} ) \subseteq \mathcal{B}_\mathrm{ext} \]
	by Lemma \ref{lem:Bcentralizer}, as claimed.
\end{proof}

\section{The standard basis}
\label{sec:standardbasis}

This section and the following are concerned with the question whether the subalgebra $\mathcal{B}$ of $\mathcal{H}$ admits a basis with an algebro-combinatorial characterization analogous to that of the Schubert basis of $\mathcal{B}_\mathrm{nil}$ in the introduction.
More precisely, we ask whether for $w \in W_\mathrm{aff}^+$, there is an element $B_w \in \mathcal{B}$ whose expansion
\[ B_w = \sum_{x \in W_\mathrm{aff}} P_{x,w} \cdot H_x \]
in terms of the standard basis of $\mathcal{H}$ satisfies $P_{w,w} = 1$ and $P_{x,w} = 0$ for $w \neq x \in W_\mathrm{aff}^+$.
While we cannot establish the existence of these elements in full generality, we prove below that an element $B_w \in \mathcal{B}$ with the properties listed above is necessarily unique (if it exists).
Furthermore, we show that the elements $\{ B_w \mid w \in W_\mathrm{aff}^+ \}$ form an $\mathcal{A}$-basis of $\mathcal{B}$ (if they exist), which we call the \emph{standard basis}.
These results are obtained by relating $\mathcal{B}$ to an analogously defined subalgebra $\mathcal{B}_0$ of the affine $0$-Hecke algebra and relating the sought after standard basis of $\mathcal{B}$ to the $K$-homological Schubert basis of $\mathcal{B}_0$ from \cite{LamSchillingShimozono}.
We conjecture that the standard basis of $\mathcal{B}$ exists for arbitrary root systems, and we prove its existence for root systems of type $\mathsf{A}_n$ in Section \ref{sec:typeA} below.

\subsection{}

Recall that the double affine Hecke algebra $\mathbb{H}$ and the affine Hecke algebra $\mathcal{H}$ are defined as algebras over the Laurent polynomial ring $\mathcal{A} = \Z[v,v^{-1}]$.
We now set $t = v^{-1}$ and define
\[ T_w = v^{-\ell(w)} \cdot H_w = t^{\ell(w)} \cdot H_w \in \mathcal{H} \]
for all $w \in W_\mathrm{aff}$, in particular $T_s = v^{-1} \cdot H_s = t \cdot H_s$ for $s \in S_\mathrm{aff}$.
With this renormalization, the quadratic relations \eqref{eq:relationsquadratic} become
\[ ( T_s + 1 ) ( T_s - t^2 ) = 0 \]
and the Bernstein relations \eqref{eq:relationsDL} become
\[ T_s X_\lambda = X_{s(\lambda)} T_s + ( 1 - t^2 ) \cdot \tfrac{ X_\lambda - X_{s(\lambda)} }{ X_{\alpha_s} - 1 } \]
for $s \in S_\mathrm{aff}$ and $\lambda \in X$.
This implies that the $\Z[t]$-submodules
\[ \mathcal{H}_t = \bigoplus_{w \in W_\mathrm{aff}} \Z[t] \cdot T_w , \hspace{2cm} \mathbb{H}_t = \bigoplus_{\lambda \in X, w \in W_\mathrm{aff}} \Z[t] \cdot X_\lambda T_w \]
are $\Z[t]$-subalgebras of $\mathcal{H}$ and $\mathbb{H}$, respectively.
By specializing $t$ to $0$, we obtain the affine $0$-Hecke algebra $\mathcal{H}_0$ and the double affine $0$-Hecke algebra $\mathbb{H}_0$,%
\footnote{More precisely, the double affine $0$-Hecke algebra is $\mathbb{H}_0 = \Z \otimes_{\Z[t]} \mathbb{H}_t$, where $\Z$ becomes a $\Z[t]$-algebra via the unique ring homomorphism $\Z[t] \to \Z$ with $t \mapsto 0$.
The definition of $\mathcal{H}_0$ is analogous.
By a slight abuse of notation, we write $T_w = 1 \otimes T_w \in \mathcal{H}_0$ and $X_\lambda = 1 \otimes X_\lambda \in \mathbb{H}_0$, for $w \in W_\mathrm{aff}$ and $\lambda \in X$.}
where the quadratic relations and the Bernstein relations become
\[ T_s^2 = - T_s , \hspace{2cm} T_s X_\lambda = X_{s(\lambda)} T_s + \tfrac{ X_\lambda - X_{s(\lambda)} }{ X_{\alpha_s} - 1 } \]
for $s \in S_\mathrm{aff}$ and $\lambda \in X$.
(This matches the relations (2.1), (2.3) and (2.6) in \cite{LamSchillingShimozono}.)
There is a unique ring homomorphism
\[ \mathrm{ev} \colon \mathbb{H}_t \longrightarrow \mathbb{H}_0 \hspace{2cm} \mathrm{ev}(f) = f(0) , \quad \mathrm{ev}(T_w) = T_w , \quad \mathrm{ev}(X_\lambda) = X_\lambda \]
for $f \in \Z[t]$, $w \in W_\mathrm{aff}$ and $\lambda \in X$, and we have $\mathrm{ev}( \mathcal{H}_t ) \subseteq \mathcal{H}_0$.

\subsection{}

The $\mathcal{A}$-linear map $\varphi \colon \mathbb{H} \to \mathcal{H}$ from \eqref{eq:phi} restricts to a $\Z[t]$-linear map $\varphi \colon \mathbb{H}_t \to \mathcal{H}_t$ and induces a $\Z$-linear map $\varphi_0 \colon \mathbb{H}_0 \to \mathcal{H}_0$ upon specialization at $0$.
It is straightforward to see that $\varphi_0 \circ \mathrm{ev} = \mathrm{ev} \circ \varphi$.
Alternatively, $\varphi_0 \colon \mathbb{H}_0 \to \mathcal{H}_0$ can be defined as the unique $\Z$-linear map with $\varphi_0(X_\lambda T_w) = T_w$ for all $\lambda \in X$ and $w \in W_\mathrm{aff}$.
We now consider the $\Z[t]$-subalgebra
\[ \mathcal{B}_t = \mathcal{B} \cap \mathcal{H}_t = \{ h \in \mathcal{H}_t \mid \varphi( h X_\lambda ) = h \text{ for all } \lambda \in X \} \]
of $\mathcal{H}_t$ and the $\Z$-subalgebra
\[ \mathcal{B}_0 = \{ h \in \mathcal{H}_0 \mid \varphi( h X_\lambda ) = h \text{ for all } \lambda \in X \} \]
of $\mathcal{H}_0$; see Section 6.1 in \cite{LamSchillingShimozono} or the proof of Lemma \ref{lem:Bsubalgebra}.

\begin{Lemma} \label{lem:B0evaluation}
	The ring homomorphism $\mathrm{ev} \colon \mathcal{H}_t \to \mathcal{H}_0$ satisfies $\mathrm{ev}( \mathcal{B}_t ) \subseteq \mathcal{B}_0$.
\end{Lemma}
\begin{proof}
	This is easily verified using the identity $\varphi_0 \circ \mathrm{ev} = \mathrm{ev} \circ \varphi$.
\end{proof}

\subsection{}
\label{subsec:standardbasisB0}

By results of Lam--Schilling--Shimozono \cite{LamSchillingShimozono}, the subalgebra $\mathcal{B}_0$ of $\mathcal{H}_0$ admits a Schubert basis with an algebro-combinatorial characterization analogous to that of the Schubert basis of $\mathcal{B}_\mathrm{nil}$ from the introduction.
(The elements of this basis correspond to the Schubert classes in the $K$-homology of the affine Grassmannian $\Gr_G$, but this alternative characterization will not be relevant for us.)
To state these results, note that for all $x \in W_\mathrm{aff}$, the coset $W_\mathrm{fin} x$ has a unique element of minimal length, and recall that we write
\[ W_\mathrm{aff}^+ = \{ x \in W_\mathrm{aff} \mid x \text{ has minimal length in } W_\mathrm{fin} x \} . \]
For later use, we also define
\[ W_\mathrm{ext}^+ = \{ x \in W_\mathrm{ext} \mid x \text{ has minimal length in } W_\mathrm{fin} x \} . \]
We call $\{ T_x \mid x \in W_\mathrm{aff} \}$ the standard basis of $\mathcal{H}_0$.
The following result is Theorem 6.4 in \cite{LamSchillingShimozono}.

\begin{Theorem}[Lam--Schilling--Shimozono] \label{thm:basisB0}
	For all $w \in W_\mathrm{aff}^+$, there is a unique element $B_{w,0} \in \mathcal{B}_0$ whose expansion
	\[ B_{w,0} = \sum_{x \in W_\mathrm{aff}} P_{x,w,0} \cdot T_x \]
	in terms of the standard basis of $\mathcal{H}_0$ satisfies $P_{w,w,0}=1$ and $P_{x,w,0} = 0$ for all $w \neq x \in W_\mathrm{aff}^+$.
	Furthermore, the elements $\{ B_{w,0} \mid w \in W_\mathrm{aff}^+ \}$ form a $\Z$-basis of $\mathcal{B}_0$.
\end{Theorem}

In order to derive some consequences of Theorem \ref{thm:basisB0} for the subalgebra $\mathcal{B}$ of $\mathcal{H}$, we will use the following corollary.

\begin{Corollary} \label{cor:B0basiscoefficients}
	Let $b \in \mathcal{B}_0$ and write $b = \sum_{x \in W_\mathrm{aff}} a_x \cdot T_x$, with $a_x \in \Z$ for $x \in W_\mathrm{aff}$.
	Then we have
	\[ b = \sum_{w \in W_\mathrm{aff}^+} a_w \cdot B_{w,0} . \]
\end{Corollary}
\begin{proof}
	By Theorem \ref{thm:basisB0}, we can write $b = \sum_{y \in W_\mathrm{aff}^+} a_y' \cdot B_{y,0}$, with $a_y' \in \Z$ for $y \in W_\mathrm{aff}^+$, and so
	\[ \sum_{x \in W_\mathrm{aff}} a_x \cdot T_x = b = \sum_{y \in W_\mathrm{aff}^+} a_y' \cdot B_{y,0} = \sum_{x \in W_\mathrm{aff}} \Big( \sum_{y \in W_\mathrm{aff}^+} P_{x,y,0} \cdot a_y' \Big) \cdot T_x . \]
	This implies that
	\[ a_x = \sum_{y \in W_\mathrm{aff}^+} P_{x,y,0} \cdot a_y' \]
	for all $x \in W_\mathrm{aff}$.
	As $P_{w,y,0} = \delta_{w,y}$ for $w,y \in W_\mathrm{aff}^+$, we conclude that $a_w = a_w'$ for $w \in W_\mathrm{aff}^+$ and
	\[ b = \sum_{w \in W_\mathrm{aff}^+} a_w' \cdot B_{w,0} = \sum_{w \in W_\mathrm{aff}^+} a_w \cdot B_{w,0} , \]
	as claimed.
\end{proof}

\subsection{}

Now we can prove our first result about $\mathcal{B}$ that is obtained via a comparison with $\mathcal{B}_0$.

\begin{Lemma} \label{lem:Bcoefficientnonzero}
	For every non-zero element $B \in \mathcal{B}$ with expansion $B = \sum_{x \in W_\mathrm{aff}} P_x \cdot H_x$ in terms of the standard basis of $\mathcal{H}$, there is an element $w \in W_\mathrm{aff}^+$ such that $P_w \neq 0$.
\end{Lemma}
\begin{proof}
	For $x \in W_\mathrm{aff}$, let $P_x' = t^{-\ell(x)} \cdot P_x$, so that $B = \sum_x P_x' \cdot T_x$.
	Since $B$ is non-zero, we can define
	\[ k = \min\{ i \in \Z \mid t^i \cdot B \in \mathcal{H}_t \} = \min\{ i \in \Z \mid t^i \cdot P_x' \in \Z[t] \text{ for all } x \in W_\mathrm{aff} \} . \]
	Then there exists some $x \in W_\mathrm{aff}$ such that $t^k \cdot P_x' \in \Z[t]$ and $t^{k-1} \cdot P_x' \notin \Z[t]$, and so $t^k \cdot P_x'$ is a polynomial with non-zero constant term.
	This implies that $\mathrm{ev}(t^k \cdot P_x') \neq 0$ and $\mathrm{ev}(t^k \cdot B) \neq 0$.
	Further observe that we have
	\[ \sum_{x \in W_\mathrm{aff}} \mathrm{ev}(t^k \cdot P_x') \cdot T_x = \mathrm{ev}(t^k \cdot B) \in \mathcal{B}_0 \]
	by Lemma \ref{lem:B0evaluation}, and so Corollary \ref{cor:B0basiscoefficients} yields
	\[ \mathrm{ev}(t^k \cdot B) = \sum_{w \in W_\mathrm{aff}^+} \mathrm{ev}(t^k \cdot P_w') \cdot b_w . \]
	As $\mathrm{ev}(t^k \cdot B) \neq 0$, we conclude that $\mathrm{ev}(t^k \cdot P_w') \neq 0$ for some $w \in W_\mathrm{aff}^+$, whence $P_w \neq 0$, as claimed. 
\end{proof}

\subsection{}

Motivated by Theorem \ref{thm:basisB0} and Lemma \ref{lem:Bcoefficientnonzero}, we propose the following conjecture.

\begin{Conjecture} \label{conj:standardbasis}
	For all $w \in W_\mathrm{aff}^+$, there exists an element $B_w \in \mathcal{B}$ whose expansion
	\[ B_w = \sum_{x \in W_\mathrm{aff}} P_{x,w} \cdot H_x \]
	in terms of the standard basis of $\mathcal{H}$ satisfies $P_{w,w}=1$ and $P_{x,w} = 0$ for all $w \neq x \in W_\mathrm{aff}^+$.
\end{Conjecture}

\begin{Remark}
	Note that Conjecture \ref{conj:standardbasis} is valid for the element $w = s_0 \in W_\mathrm{aff}^+$ by Lemma \ref{lem:Bs0}.
For root systems of type $\mathsf{A}_n$, we verify the conjecture in Section \ref{sec:typeA} below (see Theorem \ref{thm:BwexistencetypeA}), and for root systems of type $\mathsf{B}_2$ and $\mathsf{G}_2$, we verify it in Appendix \ref{app:rank2}.
\end{Remark}

\subsection{}

To conclude this section, we prove that the elements $B_w \in \mathcal{B}$ from Conjecture \ref{conj:standardbasis} are uniquely determined if they exist, and that they form an $\mathcal{A}$-basis of $\mathcal{B}_\mathrm{aff}$ if they exist.

\begin{Lemma} \label{lem:Bwunique}
	For all $w \in W_\mathrm{aff}^+$, there is at most one element $B_w \in \mathcal{B}$ whose expansion
	\[ B_w = \sum_{x \in W_\mathrm{aff}} P_{x,w} \cdot H_x \]
	in terms of the standard basis of $\mathcal{H}$ satisfies $P_{w,w}=1$ and $P_{x,w} = 0$ for all $w \neq x \in W_\mathrm{aff}^+$.
\end{Lemma}
\begin{proof}
	Suppose that $B_w,B_w' \in \mathcal{B}$ are two elements that satisfy the conditions of the lemma.
	Then Lemma \ref{lem:Bcoefficientnonzero} implies that $B_w - B_w' = 0$, and so $B_w = B_w'$, as required.
\end{proof}

\begin{Lemma} \label{lem:Bwbasis}
	If Conjecture \ref{conj:standardbasis} holds, then the elements $\{ B_w \mid w \in W_\mathrm{aff}^+ \}$ form an $\mathcal{A}$-basis of $\mathcal{B}$.
\end{Lemma}
\begin{proof}
	The linear independence of $\{ B_w \mid w \in W_\mathrm{aff}^+ \}$ is immediate from the fact that $\{ H_x \mid x \in W_\mathrm{aff} \}$ is an $\mathcal{A}$-basis of $\mathcal{H}$.
	For an arbitrary element $B \in \mathcal{B}$ with expansion
	\[ B = \sum_{x \in W_\mathrm{aff}} a_x \cdot H_x \]
	in terms of the standard basis of $\mathcal{H}$, one can use Lemma \ref{lem:Bcoefficientnonzero} to see that
	\[ B - \sum_{w \in W_\mathrm{aff}^+} a_w \cdot B_w = 0 , \]
	whence $\{ B_w \mid w \in W_\mathrm{aff}^+ \}$ spans $\mathcal{B}$ over $\mathcal{A}$, as required.
\end{proof}

\begin{Definition} \label{def:standardbasis}
	If Conjecture \ref{conj:standardbasis} holds, then we call $\{ B_w \mid w \in W_\mathrm{aff}^+ \}$ the \emph{standard basis} of $\mathcal{B}$.
\end{Definition}

\section{The standard basis in type \texorpdfstring{$\mathsf{A}_n$}{An}}
\label{sec:typeA}

The goal of this section is to prove Conjecture \ref{conj:standardbasis}, and thereby establish the existence of the standard basis of $\mathcal{B}$, for root systems of type $\mathsf{A}_n$.
Our strategy is to once again exploit an analogy between $\mathcal{B}$ and the subalgebra $\mathcal{B}_\mathrm{nil}$ of $\mathcal{H}_\mathrm{nil}$ from the introduction.
Namely, if $\Phi$ is of type $\mathsf{A}_n$, then $\mathcal{B}_\mathrm{nil}$ is isomorphic to a polynomial ring $\Z[b_1,\cdots,b_n]$ with generators of degree $\deg(b_i) = i$ that can be described explicitly in terms of the basis $\{ A_w \mid w \in W_\mathrm{aff} \}$ of $\mathcal{H}_\mathrm{nil}$, see \cite[Theorems 6.3 and 7.4]{LamSchubert}.
In Theorem \ref{thm:BrinB} below, we prove that the generators $b_1,\ldots,b_n$ of $\mathcal{B}_\mathrm{nil}$ can be lifted to elements $B_1,\ldots,B_n \in \mathcal{B}$, and in Theorem \ref{thm:BwexistencetypeA}, we construct elements $B_w \in \mathcal{B}$ for $w \in W_\mathrm{aff}^+$ as in Conjecture \ref{conj:standardbasis} as suitable polynomials in the elements $B_1,\ldots,B_r \in \mathcal{B}$.
Finally, we also prove in Theorem \ref{thm:Bcommutative} that $\mathcal{B}$ is commutative.

\subsection{}
\label{subsec:typeAsetup}

Throughout this section, we assume that $\Phi$ is of type $\mathsf{A}_n$ for some $n \geq 1$.
The Coxeter diagram of $W_\mathrm{aff}$ is a cycle of length $n+1$, as in Figure \ref{fig:CoxeterdiagramtypeAn}.
The subgroup $\Omega = \{ \omega \in W_\mathrm{ext} \mid \omega S_\mathrm{aff} \omega^{-1} = S_\mathrm{aff} \}$ of $W_\mathrm{ext}$ from Subsection \ref{subsec:WaffWext} is cyclic of order $n+1$, and the action of $\Omega$ on $S_\mathrm{aff}$ by conjugation corresponds to the graph automorphisms which rotate the Coxeter diagram.
\begin{figure}[htbp]
\begin{center}
\caption{The Coxeter diagram of $W_\mathrm{aff}$ in type $\mathsf{A}_n$}
\label{fig:CoxeterdiagramtypeAn}
\begin{tikzpicture}
	\node[circle,fill=black,inner sep=2pt] (A1) at (0,0) {};
	\node[circle,fill=black,inner sep=2pt] (A2) at (1,0) {};
	\node[circle,fill=black,inner sep=2pt] (A3) at (2.5,0) {};
	\node[circle,fill=black,inner sep=2pt] (A4) at (3.5,0) {};
	\node[circle,fill=black,inner sep=2pt] (A5) at (1.75,1) {};
	\node[below=.1cm of A1] {\scriptsize $1$};
	\node[below=.1cm of A2] {\scriptsize $2$};
	\node[below=.1cm of A3] {\scriptsize $n-1$};
	\node[below=.1cm of A4] {\scriptsize $n \vphantom{1}$};
	\node[above=.1cm of A5] {\scriptsize $0$};
	\draw (A1) -- (A2);
	\draw (A2) -- node[fill=white] {$\,\dots$} (A3);
	\draw (A3) -- (A4);
	\draw (A1) -- (A5);
	\draw (A4) -- (A5);
\end{tikzpicture}
\end{center}
\end{figure}
We identify the set of vertices in the Coxeter diagram with $I = \Z / (n+1) \Z$, and for $i \in I$, we write $s_i \in S_\mathrm{aff}$ for the corresponding simple reflection, $\alpha_i \in \Pi_\mathrm{aff}$ for the corresponding (affine) simple root and $\varpi_i$ for the corresponding fundamental dominant weight, with the convention that $\varpi_0 = 0$.
For an element $w \in W_\mathrm{aff}$ with reduced expression $w = s_{i_1} s_{i_2} \cdots s_{i_r}$, we will often write $H_w = H_{i_1 i_2 \cdots i_r}$.

\subsection{}

A proper subset $X \subsetneq I$ is called an \emph{interval} if
\[ X = \{ i , i-1 , i-2 , \ldots , j+1 , j \} \eqqcolon [i,j] \]
for some $i,j \in I$ with $j \neq i+1$, or equivalently, if the induced subgraph of the Coxeter diagram with vertex set $X$ is connected.
By convention, we set $[i,i+1] = \varnothing$ for $i \in I$.
For an interval $X = [i,j]$ as above, we define
\[ w_X \coloneqq s_i s_{i-1} \cdots s_{j+1} s_j \in W_\mathrm{aff} . \] 
Every proper subset $Y \subsetneq I$ can be written uniquely as a disjoint union
\[ Y = X_1 \sqcup \cdots \sqcup X_r \]
of intervals $X_1, \ldots , X_r \subsetneq I$ such that the number $r$ of intervals is minimal,%
\footnote{The intervals $X_1,\ldots,X_r$ in the disjoint union $Y = X_1 \sqcup \cdots \sqcup X_r$ are precisely the connected components of the induced subgraph of the Coxeter diagram with vertex set $Y$.
Note that the set $\{ X_1 , \ldots , X_r \}$ is uniquely determined, but the order of the intervals $X_1,\ldots,X_r$ may obviously be permuted.}
and we call $X_1,\ldots,X_r$ the \emph{interval decomposition} of $Y$.
For a proper subset $Y \subsetneq I$ with interval decomposition $Y = X_1 \sqcup \cdots \sqcup X_r$, we further define
\begin{equation} \label{eq:cycdecelements}
	w_Y \coloneqq w_{X_1} \cdots w_{X_r} \in W_\mathrm{aff} , \hspace{2cm} H_Y \coloneqq H_{w_Y} \in \mathcal{H} .
\end{equation}
Observe that the product $w_Y$ does not depend on the order of the intervals.
The elements $w_Y \in W_\mathrm{aff}$ for subsets $Y \subsetneq I$ are called \emph{cyclically decreasing} in \cite[Definition 6.2]{LamSchubert}.

\subsection{}

For $0 \leq r \leq n$, consider the element
\begin{equation} \label{eq:generatorsB}
	B_r \coloneqq \sum_{ \substack{ Y \subset I \\ \abs{Y} = r } } H_Y \in \mathcal{H} .
\end{equation}
In the following, we will show that
\begin{enumerate}
	\item the elements $B_1,\ldots,B_n \in \mathcal{H}$ belong to the subalgebra $\mathcal{B}$ of $\mathcal{H}$ (see Theorem \ref{thm:BrinB});
	\item the elements $B_1,\ldots,B_n \in \mathcal{B}$ generate $\mathcal{B}$ as an $\mathcal{A}$-algebra (see Corollary \ref{cor:Bgenerators});
	\item the elements $B_1,\ldots,B_n \in \mathcal{B}$ commute, whence $\mathcal{B}$ is commutative (see Theorem \ref{thm:Bcommutative}).
\end{enumerate}
Along the way, we will also show that 
\begin{enumerate}
	\setcounter{enumi}{3}
	\item the standard basis of $\mathcal{B}$ exists (see Theorems \ref{thm:BwexistencetypeA} and \ref{thm:BwbasistypeA}).
\end{enumerate}
The claim (1) will be established by a direct computation, whereas (2)--(4) will be obtained by relating $\mathcal{B}$ to $\mathcal{B}_\mathrm{nil}$ and using results of Lam \cite{LamSchubert} about the Schubert basis of $\mathcal{B}_\mathrm{nil}$.

\begin{Example}
	For $n=3$, we have
	\begin{align*}
		B_0 & = H_e , \\
		B_1 & = H_0 + H_2 + H_3 , \\
		B_2 & = H_{03} + H_{32} + H_{21} + H_{10} + H_{02} + H_{31} , \\
		B_3 & = H_{032} + H_{321} + H_{210} + H_{103} .
	\end{align*}
\end{Example}

\subsection{}

In order to show that $B_1,\ldots,B_n \in \mathcal{B}$, we need to prove some preliminary results that involve slightly technical computations in $\mathbb{H}$.

\begin{Lemma} \label{lem:commutationpi1interval}
	For all $i \in I$ with $i \neq 0$, we have
	\[ H_{[i,1]} X_{-\varpi_1} = X_{\varpi_i-\varpi_{i+1}} H_{[i,1]} - (v-v^{-1}) \cdot \sum_{\ell=1}^i X_{\varpi_{\ell-1} - \varpi_\ell} \cdot H_{[i,1] \setminus \{ \ell \}} . \]
\end{Lemma}
\begin{proof}
	By \eqref{eq:relationsDL} and Remark \ref{rem:Demazureoperator}, we have
	\[ H_1 X_{-\varpi_1} = X_{\varpi_1-\varpi_2} H_1 - (v-v^{-1}) \cdot X_{-\varpi_1} \]
	and this established the claim for $i=1$.
	For $i \neq 0$, we similarly have
	\[ H_i X_{\varpi_{i-1}-\varpi_i} = X_{\varpi_i-\varpi_{i+1}} H_i - (v-v^{-1}) \cdot X_{\varpi_{i-1} - \varpi_i} \]
	and $H_i X_{\varpi_{\ell-1}-\varpi_\ell} = X_{\varpi_{\ell-1}-\varpi_\ell} H_i$ for $\ell=1,\ldots,i-1$, and by induction, it follows that
	\begin{align*}
		H_{[i,1]} X_{-\varpi_1} & = H_i H_{[i-1,1]} \cdot X_{-\varpi_1} \\
		& = H_i \cdot \Big( X_{\varpi_{i-1} - \varpi_i} H_{[i-1,1]} - ( v - v^{-1} ) \cdot \sum_{\ell=1}^{i-1} X_{\varpi_{\ell-1}-\varpi_\ell} H_{[i-1,1] \setminus \{ \ell \}} \Big) \\
		& = \big( X_{\varpi_i-\varpi_{i+1}} H_i - ( v - v^{-1} ) \cdot X_{\varpi_{i-1}-\varpi_i} \big) \cdot H_{[i-1,1]} - ( v - v^{-1} ) \cdot \sum_{\ell=1}^{i-1} X_{\varpi_{\ell-1} - \varpi_\ell} H_i H_{[i-1,1] \setminus \{ \ell \}} \\
		& = X_{\varpi_i-\varpi_{i+1}} H_{[i,1]} - ( v - v^{-1} ) \cdot \sum_{\ell=1}^i X_{\varpi_{\ell-1} - \varpi_\ell} \cdot H_{[i,1] \setminus \{ \ell \}} ,
	\end{align*}
	as claimed.
\end{proof}

\begin{Proposition} \label{prop:commutationpi1}
	Let $Y \subsetneq I$ be a proper subset.
	\begin{enumerate}
		\item If $Y \cap \{0,1\} = \varnothing$ then we have $H_Y X_{-\varpi_1} = X_{-\varpi_1} H_Y$.
		\item If $0 \in Y$ then we have $H_Y X_{-\varpi_1} = X_{\varpi_n} H_Y + (v-v^{-1}) \cdot X_{\varpi_n} H_{Y \setminus \{0\}}$.
		\item If $1 \in Y$ and $0 \notin Y$ then in the interval decomposition $Y = X_1 \sqcup \cdots \sqcup X_m$ of $Y$, the unique interval $X_k$ with $1 \in X_k$ is of the form $X_k = [i,1]$ for some $i \in I$.
		We have
		\[ H_Y X_{-\varpi_1} = X_{\varpi_i - \varpi_{i+1}} H_Y - (v-v^{-1}) \cdot \sum_{\ell=1}^{i} X_{\varpi_{\ell-1} - \varpi_\ell} H_{Y \setminus \{\ell\}} . \]
	\end{enumerate}	 
\end{Proposition}
\begin{proof}
	We consider the three claims in turn.
	\begin{enumerate}
		\item If $Y \cap \{0,1\} = \varnothing$ then $H_i$ commutes with $X_{-\varpi_1}$ for all $i \in Y$.
		As $H_Y$ is a product of the $H_i$ with $i \in Y$, it follows that $H_Y X_{\varpi_1} = X_{\varpi_1} H_Y$.
		\item Suppose that $0 \in Y$ and fix an interval decomposition $Y = X_1 \sqcup \cdots \sqcup X_m$ such that $0 \in X_1$.
		Further let $i,j \in I$ such that
		\[ X_1 = [i,j] = \{ i , i-1 , \ldots , 1 , 0 , n , n-1 , \ldots , j+1 , j \} , \]
		and observe that $X_{-\varpi_1}$ commutes with $H_{[n,j]}$ and with $H_{X_k}$ for $1 < k \leq m$.
		We further have
		\[ H_0 X_{-\varpi_1} = X_{\varpi_n} H_0 + (v-v^{-1}) \cdot X_{\varpi_n} , \]
		and as $X_{\varpi_n}$ commutes with $H_{[i,1]}$, we conclude that
		\begin{align*}
			H_Y X_{-\varpi_1} & = H_{[i,1]} H_0 H_{[n-j]} H_{X_2} \cdots H_{X_m} X_{-\varpi_1} \\
			& = H_{[i,1]} H_0 X_{-\varpi_1} H_{[n-j]} H_{X_2} \cdots H_{X_m} \\
			& = H_{[i,1]} \cdot \big( X_{\varpi_n} H_0 + (v-v^{-1}) \cdot X_{\varpi_n} \big) \cdot H_{[n-j]} H_{X_2} \cdots H_{X_m} \\
			& = X_{\varpi_n} H_{[i,1]} H_0 H_{[n-j]} H_{X_2} \cdots H_{X_m} + (v-v^{-1}) \cdot X_{\varpi_n} H_{[i,1]} H_0 H_{[n-j]} H_{X_2} \cdots H_{X_m} \\
			& = X_{\varpi_n} H_Y + (v-v^{-1}) \cdot X_{\varpi_n} H_{Y \setminus \{0\}} ,
		\end{align*}
		as claimed.
		\item Suppose that $1 \in Y$ and $0 \notin Y$, and fix an interval decomposition $Y = X_1 \sqcup \cdots \sqcup X_m$ such that $1 \in X_1$.
		As $0 \notin Y$, we must have $X_1 = [i,1]$ for some $i \in I$, and $X_{-\varpi_1}$ commutes with $H_{X_k}$ for all $1 < k \leq m$ because $X_k \cap \{0,1\} = \varnothing$.
		By Lemma \ref{lem:commutationpi1interval}, we further have
		\[ H_{[i,1]} X_{-\varpi_1} = X_{\varpi_i-\varpi_{i+1}} H_{[i,1]} - (v-v^{-1}) \cdot \sum_{\ell=1}^i X_{\varpi_{\ell-1} - \varpi_\ell} H_{[i,1] \setminus \{ \ell \}} , \]
		and we conclude that
		\begin{align*}
			H_Y X_{-\varpi_1} & = H_{[i,1]} H_{X_2} \cdots H_{X_m} X_{-\varpi_1} \\
			& = H_{[i,1]} X_{-\varpi_1} H_{X_2} \cdots H_{X_m} \\
			& = \Big( X_{\varpi_i-\varpi_{i+1}} H_{[i,1]} - (v-v^{-1}) \cdot \sum_{\ell=1}^i X_{\varpi_{\ell-1} - \varpi_\ell} H_{[i,1] \setminus \{ \ell \}} \Big) \cdot H_{X_2} \cdots H_{X_m} \\
			& = X_{\varpi_i-\varpi_{i+1}} H_{[i,1]} H_{X_2} \cdots H_{X_m} - (v-v^{-1}) \cdot \sum_{\ell=1}^i X_{\varpi_{\ell-1} - \varpi_\ell} H_{[i,1] \setminus \{ \ell \}} H_{X_2} \cdots H_{X_m} \\
			& = X_{\varpi_i-\varpi_{i+1}} H_Y - (v-v^{-1}) \cdot \sum_{\ell=1}^i X_{\varpi_{\ell-1} - \varpi_\ell} H_{Y \setminus \{ \ell \}} ,
		\end{align*}
		as claimed.
		\qedhere
	\end{enumerate}
\end{proof}

\begin{Proposition} \label{prop:Brcommutationpi1}
	For $0 \leq r \leq n$, we have $\varphi( B_r X_{-\varpi_1} ) = B_r$.
\end{Proposition}
\begin{proof}
	Let us write $C_r = \{ Y \subseteq I \mid \abs{Y} = r \}$ and recall that $B_r = \sum_{Y \in C_r} H_Y$.
	We further define
	\begin{align*}
		C_r' & = \{ Y \in C_r \mid Y \cap \{ 0 , 1 \} = \varnothing \} , \\
		C_r^0 & = \{ Y \in C_r \mid 0 \in Y \} , \\
		D_r & = \{ Y \in C_r \mid 1 \in Y \text{ and } 0 \notin Y \} ,
	\end{align*}
	so that $C_r = C_r' \sqcup C_r^0 \sqcup D_r$, and for $Y \in D_r$, we let $i_Y \in I$ such that $[i_Y,1]$ is the unique interval containing $1$ in the interval decomposition of $Y$ (cf.\ part (3) of Proposition \ref{prop:commutationpi1}).
	By Proposition \ref{prop:commutationpi1}, we have
	\begin{align*}
		B_r X_{-\varpi_1} & = \sum_{Y \in C_r} H_Y X_{-\varpi_1} \\
		& = \sum_{Y \in C_r'} H_Y X_{-\varpi_1} + \sum_{Y \in C_r^0} H_Y X_{-\varpi_1} + \sum_{Y \in D_r} H_Y X_{-\varpi_1} \\
		& = \sum_{Y \in C_r'} X_{-\varpi_1} H_Y + \sum_{Y \in C_r^0} \big( X_{\varpi_n} H_Y + (v-v^{-1}) \cdot X_{\varpi_n} \cdot H_{Y \setminus \{ 0 \} } \\
		& \hspace{2cm} + \sum_{Y \in D_r} \Big( X_{\varpi_{i_Y}-\varpi_{i_Y+1}} H_Y - (v-v^{-1}) \cdot \sum_{\ell=1}^{i_Y} X_{\varpi_{\ell-1} - \varpi_\ell} H_{Y \setminus \{ \ell \}} \Big) .
	\end{align*}
	This implies that
	\[ \varphi(B_r X_{\varpi_1}) = \sum_{Y \in C_r} H_Y + R = B_r + R , \]
	where we set
	\[ R \coloneqq (v-v^{-1}) \cdot \Big( \sum_{Y \in C_r^0} H_{Y \setminus \{0\}} - \sum_{Y \in D_r} \sum_{\ell=1}^{i_Y} H_{Y \setminus \{ \ell \}} \Big) , \]
	and it remains to show that $R=0$.
	To that end, let
	\[ D_r' \coloneqq \{ (Y,\ell) \mid Y \in D_r , \, 1 \leq \ell \leq i_Y \} \]
	and observe that the map
	\[ \Psi \colon D_r' \longrightarrow C_r^0 , \qquad (Y,\ell) \mapsto Y  \cup \{0\} \setminus \{\ell\} \]
	is a bijection.
	The inverse map is given by $Y' \mapsto Y' \cup \{i+1\} \setminus \{0\}$ for $Y' \in C_r^0$ and $i,j \in I$ such that $[i,j]$ is the unique interval containing $0$ in the interval decomposition of $Y'$.
	For $(Y,\ell) \in D_r'$, we further have $Y \setminus \{\ell\} = \Psi( (Y,\ell) ) \setminus \{0\}$, and it follows that
	\[ R = (v-v^{-1}) \cdot \Big( \sum_{Y \in C_r^0} H_{Y \setminus \{0\}} - \sum_{Y \in D_r} \sum_{\ell=1}^{i_Y} H_{Y \setminus \{ \ell \}} \Big) = (v-v^{-1}) \cdot \sum_{(Y,\ell) \in D_r'} \big( H_{\Psi( (Y,\ell) ) \setminus \{0\}} - H_{Y \setminus \{ \ell \}} \big) = 0 , \]
	as required.
\end{proof}

\subsection{}

Using the technical results from the preceding subsection, we can now show that $B_r \in \mathcal{B}$ for all $1 \leq r \leq n$.
More precisely, we will prove this by a reduction to Proposition \ref{prop:Brcommutationpi1}, using the two following elementary observations about the subgroup $\Omega = \{ \omega \in W_\mathrm{ext} \mid \omega S_\mathrm{aff} \omega^{-1} = S_\mathrm{aff} \}$ of $W_\mathrm{ext}$ from Subsection \ref{subsec:WaffWext}.

\begin{Lemma} \label{lem:omegaBr}
	For all $\omega \in \Omega$ and $1 \leq r \leq n$, we have $\omega B_r \omega^{-1} = B_r$.
\end{Lemma}
\begin{proof}
	Recall from Subsection \ref{subsec:typeAsetup} that the action of $\omega$ on $S_\mathrm{aff}$ by conjugation is given by a permutation $\sigma_\omega$ of $I$ that correspond to a graph automorphisms which rotates the Coxeter diagram in Figure \ref{fig:CoxeterdiagramtypeAn}.
	In other words, we have $\omega s_i \omega^{-1} = s_{\sigma_\omega(i)}$ for all $i \in I$, where $\sigma(i) = i+j_\omega$ for some $j_\omega \in I$.
	For a proper subset $Y \subsetneq I$, it is straightforward to see that $\omega H_Y \omega^{-1} = H_{\sigma_\omega(Y)}$, and we conclude that
	\[ \omega B_r \omega^{-1} = \sum_{ \substack{ Y \subset I \\ \abs{Y} = r } } \omega H_Y \omega^{-1} = \sum_{ \substack{ Y \subset I \\ \abs{Y} = r } } H_{\sigma_\omega(Y)} = \sum_{ \substack{ Y \subset I \\ \abs{Y} = r } } H_Y = B_r , \]
	as claimed.
\end{proof}

\begin{Lemma} \label{lem:omegaorbitpi1}
	The orbit of the weight $-\varpi_1$ under the action of $\Omega$ on $X$ is given by
	\[ \Omega(-\varpi_1) = \{ -\varpi_1 , \varpi_1-\varpi_2 , \cdots , \varpi_{n-1} - \varpi_n , \varpi_n \} . \]
	In particular, $\Omega(-\varpi_1)$ is a generating set of the weight lattice $X$.
\end{Lemma}
\begin{proof}
	For $\omega = t_{\varpi_1^\vee} s_1 s_2 \cdots s_n$ and $i \in I$, one easily checks that $\omega(\alpha_i) = \alpha_{i+1}$ with respect to the action of $W_\mathrm{ext}$ on $\tilde Q$ from Subsection \ref{subsec:affineweights}.
	In particular, we have $\omega s_i \omega^{-1} = s_{i+1}$ for all $i \in I$ and so $\omega \in \Omega$.
	As $\Omega$ is cyclic of order $n+1$, it also follows that $\omega$ generates $\Omega$.
	Now for the action of $W_\mathrm{ext}$ on $X$, the action of the translation lattice $Y$ in $W_\mathrm{ext}$ is trivial, so $\omega$ acts via $s_1 \cdots s_n$, and we compute
	\[ \omega(-\varpi_1) = \varpi_1 - \varpi_2 , \qquad \omega( \varpi_i - \varpi_{i+1} ) = \varpi_{i+1} - \varpi_{i+2} , \qquad \omega( \varpi_{n-1} - \varpi_n ) = \varpi_n , \qquad \omega(\varpi_n) = -\varpi_1 \]
	for $1 \leq i < n-1$.
	In particular, we have $\Omega(-\varpi_1) = \{ -\varpi_1 , \varpi_1-\varpi_2 , \cdots , \varpi_{n-1} - \varpi_n , \varpi_n \}$.
	The submodule of $X$ generated by $\Omega(-\varpi_1)$ clearly contains all fundamental dominant weights, so $\Omega(-\varpi_1)$ generates $X$, as claimed.
\end{proof}

Now we are ready to prove the first main result of this section.

\begin{Theorem} \label{thm:BrinB}
	For $0 \leq r \leq n$, we have $B_r \in \mathcal{B}$.
\end{Theorem}
\begin{proof}
	First observe that for all $h \in \mathcal{H}$, $\omega \in \Omega$ and $\lambda \in X$, we have
	\[ \varphi( \omega h \omega^{-1} X_{\omega(\lambda)} ) = \varphi( \omega h  X_\lambda \omega^{-1} ) = \omega \varphi( h X_\lambda ) \omega^{-1} \]
	by \eqref{eq:relationsomegaweight} and Lemma \ref{lem:phiomega}.
	As $\omega B_r \omega^{-1} = B_r$ by Lemma \ref{lem:omegaBr} and $\varphi( B_r X_{-\varpi_1} ) = B_r$ by Proposition \ref{prop:Brcommutationpi1}, we conclude that
	\[ \varphi( B_r X_{\omega(-\varpi_1)} ) = \varphi( \omega B_r \omega^{-1} X_{\omega(-\varpi_1)} ) = \omega \varphi( B_r X_{-\varpi_1} ) \omega^{-1} = \omega B_r \omega^{-1} = B_r \]
	for all $\omega \in \Omega$, and as the orbit $\Omega(-\varpi_1)$ of $-\varpi_1$ under the action of $\Omega$ is a generating set of the weight lattice $X$ by Lemma \ref{lem:omegaorbitpi1}, Lemma \ref{lem:BgeneratingsetX} implies that $B_r \in \mathcal{B}$, as claimed. 
\end{proof}

\subsection{}

In order to show that the elements $B_0,\ldots,B_r$ generate $\mathcal{B}$, we will relate $\mathcal{B}$ to the subalgebra $\mathcal{B}_\mathrm{nil}$ of the affine nil-Hecke algebra $\mathcal{H}_\mathrm{nil}$ that was already discussed in the introduction.
To that end, we first need to introduce some additional notation relating to the (double) affine nil-Hecke algebra.
The \emph{affine nil-Hecke algebra} $\mathcal{H}_\mathrm{nil}$ is the $\Z$-algebra with generators $A_i$, for $i \in I$, subject to the quadratic relations $A_i^2 = 0$ and the braid relations of $W_\mathrm{aff}$.
For an element $w \in W_\mathrm{aff}$ with reduced expression $w = s_{i_1} \cdots s_{i_r}$, we write $A_w = A_{i_1} \cdots A_{i_r}$.
Then the elements $\{ A_w \mid w \in W_\mathrm{aff} \}$ form a $\Z$-basis of $\mathcal{H}_\mathrm{nil}$, which we call the \emph{standard basis}.

Now let us write $S = \mathrm{Sym}(X) = \Z[\varpi_1,\ldots,\varpi_n]$ for the symmetric algebra of the weight lattice $X$ over $\Z$.
The \emph{double affine nil-Hecke algebra} $\mathbb{H}_\mathrm{nil}$ (first introduced by Kostant--Kumar \cite{KostantKumarnilHecke}) is generated by $\mathcal{H}_\mathrm{nil}$ and $S$, subject to the commutation relation
\[ A_i \cdot \lambda = s_i(\lambda) \cdot A_i + (\lambda,\alpha_i^\vee) \cdot 1 \]
for $i \in I$ and $\lambda \in X$.
\begin{Remark}
	Our terminology differs from that of \cite{KostantKumarnilHecke,LamSchubert}, where $\mathcal{H}_\mathrm{nil}$ is called the \emph{affine nil-Coxeter ring} and and $\mathbb{H}_\mathrm{nil}$ is called the \emph{affine nil-Hecke ring}.
	We have chosen non-standard terminology to keep the names of $\mathcal{H}_\mathrm{nil}$ and $\mathbb{H}_\mathrm{nil}$ similar to those of $\mathcal{H}$ and $\mathbb{H}$ and thereby stress the analogies between these algebras and the roles they play in the definition of $\mathcal{B}_\mathrm{nil}$ and $\mathcal{B}$, respectively.
\end{Remark}
The multiplication map induces an isomorphism
\[ S \otimes_\Z \mathcal{H}_\mathrm{nil} \cong \mathbb{H}_\mathrm{nil} \]
of $\Z$-modules.
Using the canonical evaluation map $S \to \Z$, $f \mapsto f(0)$ with $\lambda \mapsto 0$ for all $\lambda \in X$, we can define a $\Z$-linear map
\[ \varphi_\mathrm{nil} \colon \mathbb{H}_\mathrm{nil} \longrightarrow \mathcal{H}_\mathrm{nil} \]
such that $\varphi(fh) = f(0) \cdot h$ for all $f \in S$ and $h \in \mathcal{H}_\mathrm{nil}$.
By Lemma 5.2 in \cite{LamSchubert}, the $\Z$-submodule
\[ \mathcal{B}_\mathrm{nil} = \{ h \in \mathcal{H}_\mathrm{nil} \mid \varphi_\mathrm{nil}( h \cdot f ) = f_0 \cdot h \text{ for all } f \in S(X) \} \]
is in fact a $\Z$-subalgebra of $\mathcal{H}_\mathrm{nil}$, called the \emph{affine Fomin--Stanley subalgebra}.
As explained in the introduction, the subalgebra $\mathcal{B}_\mathrm{nil}$ of $\mathcal{H}_\mathrm{nil}$ provides a combinatorial model for the affine Grassmannian homology $H_\bullet(\Gr_G)$, and the fundamental classes of the Schubert varieties afford a basis of $\mathcal{B}_\mathrm{nil}$ with the following algebro-combinatorial characterization (see \cite[Proposition 5.4]{LamSchubert}):

\begin{Theorem} \label{thm:basisBnil}
	For all $w \in W_\mathrm{aff}^+$, there is a unique element $b_w \in \mathcal{B}_\mathrm{nil}$ whose expansion
	\[ b_w = \sum_{x \in W_\mathrm{aff}} p_{x,w} \cdot A_x \]
	in terms of the standard basis of $\mathcal{H}_\mathrm{nil}$ satisfies $p_{w,w} = 1$ and $p_{x,w} = 0$ for all $w \neq x \in W_\mathrm{aff}^+$.
	The elements $\{ b_w \mid w \in W_\mathrm{aff}^+ \}$ form a $\Z$-basis of $\mathcal{B}_\mathrm{nil}$.
	\end{Theorem}

Now for $0 \leq r \leq n$, let
\begin{equation} \label{eq:generatorsBnil}
	b_r \coloneqq \sum_{ \substack{ Y \subset I \\ \abs{Y} = r } } A_{w_Y} \in \mathcal{H}_\mathrm{nil} ,
\end{equation}
as in \eqref{eq:generatorsB}, where $w_Y$ is as in \eqref{eq:cycdecelements}.
Then we have $b_r \in \mathcal{B}_\mathrm{nil}$ for $0 \leq r \leq n$, and the elements $b_1,\ldots,b_n$ are algebraically independent and generate $\mathcal{B}_\mathrm{nil}$ as a $\Z$-algebra by Theorems 6.3 and 7.4 in \cite{LamSchubert}.%
\footnote{The elements $b_i$ are denoted by $h_i$ in \cite{LamSchubert}, and the root system is of type $\mathrm{A}_{n-1}$ there.}
In particular, $\mathcal{B}_\mathrm{nil}$ is isomorphic to a polynomial ring $\mathcal{B} \cong \Z[b_1,\ldots,b_n]$.

\subsection{}

Observe that $\mathcal{H}_\mathrm{nil}$ is a graded $\Z$-algebra, with graded pieces defined by
\[ \mathcal{H}_\mathrm{nil}^i = \bigoplus_{ \substack{ x \in W_\mathrm{aff} \\ \ell(x) = i } } \Z \cdot A_x \]
for $i \geq 0$.
As the subalgebra $\mathcal{B}_\mathrm{nil}$ of $\mathcal{H}_\mathrm{nil}$ is generated by the homogeneous elements $b_1,\ldots,b_n$, it inherits a grading from $\mathcal{H}_\mathrm{nil}$, with graded pieces defined by $\mathcal{B}_\mathrm{nil}^i = \mathcal{B}_\mathrm{nil} \cap \mathcal{H}_\mathrm{nil}^i$ for $i \geq 0$.
It is straightforward to see that we must have $b_w \in \mathcal{B}_{\mathrm{nil}}^{\ell(w)}$ for all $w \in W_\mathrm{aff}^+$,%
\footnote{In other words, the coefficients $p_{x,w} \in \Z$ in $b_w = \sum_x p_{x,w} \cdot A_x$ satisfy $p_{x,w} = 0$ unless $\ell(x) = \ell(w)$.}
and therefore
\begin{equation} \label{eq:BnilgradingWaff}
	\mathcal{B}_\mathrm{nil}^i = \bigoplus_{ \substack{ w \in W_\mathrm{aff}^+ \\ \ell(w) = i } } \Z \cdot b_w
\end{equation}
for $i \geq 0$.

A different description of the grading pieces $\mathcal{B}_\mathrm{nil}^i$ for $i \geq 0$ is as follows:
Recall that a partition is a sequence $\lambda = (\lambda_1,\lambda_2,\ldots)$ of non-negative integers with $\lambda_i \geq \lambda_{i+1}$ for all $i$ and $\lambda_r = 0$ for some $r > 0$.
We write $|\lambda| = \lambda_1 + \lambda_2 + \cdots$ for the size of $\lambda$ and call $\lambda$ \emph{$n$-bounded} if $\lambda_1 \leq n$, and we write $\mathcal{P}_{\leq n}$ for the set of $n$-bounded partitions.
For $\lambda \in \mathcal{P}_{\leq n}$, we can uniquely write $\lambda = (n^{i_n},\ldots,2^{i_2},1^{i_1})$ with $i_1,\ldots,i_n \geq 0$ (i.e.\ $\lambda$ has $i_j$ parts equal to $j$), and we define
\begin{equation} \label{eq:Blambda}
	b_\lambda \coloneqq b_1^{i_1} b_2^{i_2} \cdots b_n^{i_n} \in \mathcal{B}_\mathrm{nil} , \hspace{2cm} B_\lambda \coloneqq B_1^{i_1} B_2^{i_2} \cdots B_n^{i_n} \in \mathcal{B} .
\end{equation}
Then we clearly have $b_\lambda \in \mathcal{B}_\mathrm{nil}^{\abs{\lambda}}$ for $\lambda \in \mathcal{P}_{\leq n}$, and as the elements $b_1,\ldots,b_r \in \mathcal{B}_\mathrm{nil}$ are algebraically independent and generate $\mathcal{B}_\mathrm{nil}$, it follows that
\begin{equation} \label{eq:Bnilgradingpartitions}
	\mathcal{B}_\mathrm{nil}^i = \bigoplus_{ \substack{ \lambda \in \mathcal{P}_{\leq n} \\ \abs{\lambda} = i } } \Z \cdot b_\lambda .
\end{equation}
for all $i \geq 0$.

\begin{Remark}
	The equations \eqref{eq:BnilgradingWaff} and \eqref{eq:Bnilgradingpartitions} are related via the fact that there is a canonical bijection $\vartheta \colon W_\mathrm{aff}^+ \longrightarrow \mathcal{P}_{\leq n}$ such that $\ell(w) = \abs{\vartheta(w)}$ for all $w \in W_\mathrm{aff}^+$, see \cite[Corollary 40]{LapointeMorseCoresAffinePermutations}.
\end{Remark}

\subsection{}
\label{subsec:filtrationHaff}

The affine Hecke algebra $\mathcal{H}$ is not graded, but it is still filtered, with filtration pieces defined by
\[ \mathcal{H}^{\leq i} = \bigoplus_{ \substack{ w \in W_\mathrm{aff} \\ \ell(w) \leq i } } \mathcal{A} \cdot H_w \]
for $i \geq 0$.
Let us set $\mathcal{H}^{\leq -1} = \{0\}$ and write
	\[ \mathrm{Gr}(\mathcal{H}) = \bigoplus_{i \geq 0} \mathrm{Gr}_i(\mathcal{H}) , \qquad \mathrm{Gr}_i(\mathcal{H}) = \mathcal{H}^{\leq i} / \mathcal{H}^{\leq i-1} , \]
	for the associated graded $\mathcal{A}$-algebra of $\mathcal{H}$ with respect to the filtration $\mathcal{H}^{\leq 0} \subseteq \mathcal{H}^{\leq 1} \subseteq \cdots$.
For all $x \in W_\mathrm{aff}$, we have $H_x \in \mathcal{H}^{\leq \ell(x)}$, and we define
\[ [H_x] \coloneqq H_x + \mathcal{H}^{\leq \ell(x)-1} \in \mathrm{Gr}_{\ell(x)}(\mathcal{H}) . \]
It is straightforward to see that for $i \geq 0$, we have
\[ \mathrm{Gr}_i(\mathcal{H}) = \bigoplus_{ \substack{ x \in W_\mathrm{aff} \\ \ell(x) = i } } \mathcal{A} \cdot [H_x] . \]

\begin{Lemma} \label{lem:homomorphismHnilGrH}
	There is a unique homomorphism
	\[ \psi \colon \mathcal{H}_\mathrm{nil} \longrightarrow \mathrm{Gr}(\mathcal{H}) \]
	such that $\psi(A_i) = [H_i]$ for all $i \in I$, and we have $\psi(A_w) = [H_w]$ for all $w \in W_\mathrm{aff}$.
\end{Lemma}
\begin{proof}
	For all $i \in I$, the quadratic relation $H_i^2 = (v^{-1}-v) \cdot H_i + 1$ implies that $H_i^2 \in \mathcal{H}^{\leq 1}$ and so $[H_i]^2 = 0$.
	The elements $\{ [H_i] \mid i \in I \}$ of $\mathrm{Gr}(\mathcal{H})$ satisfy the braid relations because the elements $\{ H_i \mid i \in I \}$ of $\mathcal{H}$ satisfy the braid relations, and it follows that there is a unique homomorphism $\psi \colon \mathcal{H}_\mathrm{nil} \to \mathrm{Gr}(\mathcal{H})$ such that $\psi(A_i) = [H_i]$ for all $i \in I$.
	For $w \in W_\mathrm{aff}$ with reduced expression $w = s_{i_1} \cdots s_{i_r}$, we have
	\[ \psi(A_w) = \psi(A_{i_1}) \cdots \psi(A_{i_r}) = [H_{i_1}] \cdots [H_{i_r}] = [H_w] , \]
	as claimed.
\end{proof}

For all $\lambda \in \mathcal{P}_{\leq n}$, we have $B_\lambda \in \mathcal{H}^{\leq \abs{\lambda}}$, and we define
\[ [B_\lambda] \coloneqq B_\lambda + \mathcal{H}^{\leq \abs{\lambda}-1} \in \mathrm{Gr}_{\abs{\lambda}}(\mathcal{H}) . \]

\begin{Lemma} \label{lem:BlambdaHnil}
	For all $\lambda \in \mathcal{P}_{\leq n}$, the homomorphism $\psi \colon \mathcal{H}_\mathrm{nil} \to \mathrm{Gr}(\mathcal{H})$ from Lemma \ref{lem:homomorphismHnilGrH} satisfies
	\[ \psi(b_\lambda) = [B_\lambda] . \]
\end{Lemma}
\begin{proof}
	For $0 \leq r \leq n$, we have $\psi(b_r) = [B_r]$ by the definitions of $b_r$ and $B_r$, see \eqref{eq:generatorsB} and \eqref{eq:generatorsBnil}.
	With $\lambda = (n^{i_n},\ldots,2^{i_2},1^{i_1})$, it follows that
	\[ \psi( b_\lambda ) = \psi( b_1 )^{i_1} \cdots \psi( b_n )^{i_n} = [B_1]^{i_1} \cdots [B_n]^{i_n} = [B_\lambda] , \]
	as claimed.
\end{proof}

\subsection{}

Now we can establish more of the main results of this section, namely, that the standard basis of $\mathcal{B}$ exists, and that the elements $B_1,\ldots,B_r \in \mathcal{B}$ generate $\mathcal{B}$.
The following theorem verifies Conjecture \ref{conj:standardbasis} for $\Phi$ of type $\mathsf{A}_n$.

\begin{Theorem} \label{thm:BwexistencetypeA}
	For all $w \in W_\mathrm{aff}^+$, there is a unique element $B_w \in \mathcal{B}$ whose expansion
	\[ B_w = \sum_{x \in W_\mathrm{aff}} P_{x,w} \cdot H_x \]
	in terms of the stadard basis of $\mathcal{H}$ satisfies $P_{w,w} = 1$ and $P_{x,w} = 0$ for all $w \neq x \in W_\mathrm{aff}^+$.
\end{Theorem}
\begin{proof}
	We prove the claim by induction on $\ell(w)$.
	If $\ell(w) = 0$ then we have $w = e$ and the claim follows with $B_e = H_e \in \mathcal{B}$.
	Now suppose that $\ell(w) \geq 1$ and that there are elements $B_u \in \mathcal{B}$ as in the statement of the theorem for all $u \in W_\mathrm{aff}^+$ with $\ell(u) < \ell(w)$.
	By \eqref{eq:BnilgradingWaff} and \eqref{eq:Bnilgradingpartitions}, we can choose integers $\tilde{p}_{\lambda,w} \in \Z$ for $\lambda \in \mathcal{P}_{\leq n}$ with $\abs{\lambda} = \ell(w)$ such that
	\[ b_w = \sum_{ \substack{ \lambda \in \mathcal{P}_{\leq n} \\ \abs{\lambda} = \ell(w) } } \tilde{p}_{\lambda,w} \cdot b_\lambda , \]
	and we consider the element
	\[ \hat B_w = \sum_{ \substack{ \lambda \in \mathcal{P}_{\leq n} \\ \abs{\lambda} = \ell(w) } } \tilde{p}_{\lambda,w} \cdot B_\lambda \in \mathcal{B} \cap \mathcal{H}^{\leq \ell(w)} . \]
	Further recall that $b_w = \sum_x p_{x,w} \cdot A_x$ with $p_{x,w} \in \Z$ for $x \in W_\mathrm{aff}$ (and $p_{x,w}=0$ unless $\ell(x) = \ell(w)$), and define $\hat P_{x,w} \in \mathcal{A}$ for $x \in W_\mathrm{aff}$ via $\hat B_w = \sum_x \hat P_{x,w} \cdot H_x$.
	For the element
	\[ [\hat B_w] = \hat B_w + \mathcal{H}^{\leq \ell(w)-1} \in \mathrm{Gr}_{\ell(w)}(\mathcal{H}) , \]
	we compute using Lemmas \ref{lem:homomorphismHnilGrH} and \ref{lem:BlambdaHnil} that
	\begin{multline*}
		\sum_{ \substack{x \in W_\mathrm{aff} \\ \ell(x) = \ell(w) } } \hat P_{x,w} \cdot [ H_x ] = [\hat B_w] = \sum_\lambda \tilde p_{\lambda,w} \cdot [B_\lambda] \\[-2pt]
		= \sum_\lambda \tilde p_{\lambda,w} \cdot \psi( b_\lambda ) = \psi( b_w ) = \sum_x p_{x,w} \cdot \psi(A_x) = \sum_x p_{x,w} \cdot [ H_x ] ,
	\end{multline*}
	and it follows that $\hat P_{x,w} = p_{x,w}$ for all $x \in W_\mathrm{aff}$ with $\ell(x) = \ell(w)$.
	In particular, we have $\hat P_{w,w} = 1$ and $\hat P_{x,w} = 0$ for all $w \neq x \in W_\mathrm{aff}^+$ with $\ell(x) = \ell(w)$, by the definition of $b_w$ (see Theorem \ref{thm:basisBnil}).
	We further have $\hat P_{x,w} = 0$ for all $x \in W_\mathrm{aff}$ with $\ell(x) > \ell(w)$ because $\hat B_w \in \mathcal{H}^{\leq \ell(w)}$, and it is straightforward to check that the element
	\[ B_w \coloneqq \hat B_w - \sum_{ \substack{ u \in W_\mathrm{aff}^+ \\ \ell)(u) < \ell(w) } } \hat P_{u,w} \cdot B_u \in \mathcal{B} \]
	satisfies the conditions stated in the theorem.
	The uniqueness of $B_w$ follows from Lemma \ref{lem:Bwunique}.
\end{proof}

\begin{Remark} \label{rem:Bwproperties}
	The proof of Theorem \ref{thm:BwexistencetypeA} implies that for $w \in W_\mathrm{aff}^+$, the element $B_w \in \mathcal{B}$ is an $\mathcal{A}$-linear combination of the elements $B_\lambda \in \mathcal{B}$, for $\lambda \in \mathcal{P}_{\leq n}$ with $\abs{\lambda} \leq \ell(w)$.
	In particular, we have $B_w \in \mathcal{H}^{\leq \ell(w)}$ and $P_{x,w} = 0$ for all $x \in W_\mathrm{aff}$ with $\ell(x)>\ell(w)$.
	Furthermore, the proof implies that we have $P_{x,w} = p_{x,w} \in \Z$ for all $x \in W_\mathrm{aff}$ with $\ell(x) = \ell(w)$.
\end{Remark}

\begin{Remark} \label{rem:BrasBw}
	Using the notations from Theorem \ref{thm:BwexistencetypeA}, we have $B_0 = B_e$ and $B_r = B_{w_{[0,n-r+1]}}$ for $1 \leq r \leq n$, where $w_{[0,n-r+1]} = s_0 s_n s_{n-1} \cdots s_{n-r+1}$ as in \eqref{eq:cycdecelements}.
\end{Remark}

\begin{Theorem} \label{thm:BwbasistypeA}
	The elements $\{ B_w \mid w \in W_\mathrm{aff}^+ \}$ form an $\mathcal{A}$-basis of $\mathcal{B}$.
\end{Theorem}
\begin{proof}
	This follows from Lemma \ref{lem:Bwbasis} and Theorem \ref{thm:BwexistencetypeA}.
\end{proof}

As in Section \ref{sec:standardbasis}, we call $\{ B_w \mid w \in W_\mathrm{aff}^+ \}$ the \emph{standard basis} of $\mathcal{B}$.
For $w \in W_\mathrm{aff}^+$ with reduced expression $w = s_{i_1} s_{i_2} \cdots s_{i_r}$, we sometimes write $B_w = B_{i_1 i_2 \cdots i_r}$.

\begin{Corollary} \label{cor:Bgenerators}
	The elements $B_1,\ldots,B_n \in \mathcal{B}$ generate $\mathcal{B}$ as an $\mathcal{A}$-algebra.
\end{Corollary}
\begin{proof}
	The subalgebra generated by $B_1,\ldots,B_n$ contains $\{ B_w \mid w \in W_\mathrm{aff} \}$ by Remark \ref{rem:Bwproperties} and the definition of the elements $B_\lambda \in \mathcal{B}$ for $\lambda \in \mathcal{P}_{\leq n}$ in \eqref{eq:Blambda}.
	The claim follows because $\{ B_w \mid w \in W_\mathrm{aff} \}$ is an $\mathcal{A}$-basis of $\mathcal{B}$ by Theorem \ref{thm:BwbasistypeA}.
\end{proof}

\subsection{}

In order to prove that $\mathcal{B}$ is commutative, we will again pass to the associated graded $\mathcal{A}$-algebra $\mathrm{Gr}(\mathcal{H})$ of $\mathcal{H}$ and use the following result of Lam; see Corollary 5.6 in \cite{LamSchubert}.

\begin{Proposition} \label{prop:Bnilcommutative}
	The $\Z$-subalgebra $\mathcal{B}_\mathrm{nil}$ of $\mathcal{H}_\mathrm{nil}$ is commutative.
\end{Proposition}

In order to prove an analogous result for the $\mathcal{A}$-subalgebra $\mathcal{B}$ of $\mathcal{H}$, note that there is an involutive anti-automorphism
\[ \imath \colon \mathcal{H} \longrightarrow \mathcal{H} , \hspace{1.5cm} \imath(H_x) = H_{x^{-1}} \quad \text{for } x \in W_\mathrm{aff} , \]
see the proof of Theorem 2.7 in \cite{SoergelKL}.
Furthermore, there is an involutive automorphism
\[ \zeta \colon \mathcal{H} \longrightarrow \mathcal{H} , \hspace{1.5cm} \zeta(H_i) = H_{-i} \quad \text{for } i \in I = \Z / (n+1) \Z , \]
corresponding to the diagram automorphism that reflects the Coxeter diagram in Figure \ref{fig:CoxeterdiagramtypeAn} in a vertical axis through the vertex labeled by $0$.

\begin{Lemma} \label{lem:zetaiBr}
	The anti-automorphism $\zeta \imath \colon \mathcal{H} \to \mathcal{H}$ satisfies $\zeta \imath(B_r) = B_r$ for all $0 \leq r \leq n$.
\end{Lemma}
\begin{proof}
	For a proper subset $Y \subsetneq I$, let $Y^- = \{-i \mid i \in Y \}$.
	We have $\zeta \imath(H_Y) = H_{Y^-}$ for all $Y \subsetneq I$, and using the definition of $B_r$ in \eqref{eq:generatorsB}, it follows that
	\[ \zeta \imath( B_r ) = \sum_{ \substack{ Y \subset I \\ \abs{Y} = r } } \zeta \imath( H_Y ) = \sum_{ \substack{ Y \subset I \\ \abs{Y} = r } } H_{Y^-} = \sum_{ \substack{ Y \subset I \\ \abs{Y} = r } } H_Y = B_r , \]
	as claimed.
\end{proof}

\begin{Lemma} \label{lem:zetaiBw}
	Let $w \in W_\mathrm{aff}^+$ and suppose that $B_i$ and $B_j$ commute for $1 \leq i,j \leq n$ with $i+j \leq \ell(w)$.
	Then we have $\zeta \imath(B_w) = B_w$.
\end{Lemma}
\begin{proof}
	Recall from Remark \ref{rem:Bwproperties} that $B_w$ is an $\mathcal{A}$-linear combination of elements $B_\lambda \in \mathcal{B}$, for $\lambda \in \mathcal{P}_{\leq n}$ with $\abs{\lambda} \leq \ell(w)$.
	Therefore, it suffices to prove that $\zeta \imath(B_\lambda) = B_\lambda$ for all $\lambda \in \mathcal{P}_{\leq n}$ with $\abs{\lambda} \leq \ell(w)$.
	Indeed, let $\lambda = ( \lambda_1 , \lambda_2 , \ldots ) \in \mathcal{P}_{\leq n}$ with $\abs{\lambda} \leq \ell(w)$, and observe that $B_{\lambda_i}$ and $B_{\lambda_j}$ commute for all $i > j \geq 1$ because $\lambda_i + \lambda_j \leq \abs{\lambda} \leq \ell(w)$.
	Using Lemma \ref{lem:zetaiBr} and the fact that $\zeta \imath$ is an anti-automorphism, it follows that
	\[ \zeta \imath( B_\lambda ) = \zeta \imath( \cdots B_{\lambda_3} B_{\lambda_2} B_{\lambda_1} ) = \zeta \imath( B_{\lambda_1} ) \zeta \imath( B_{\lambda_2} ) \zeta \imath( B_{\lambda_3} ) \cdots = B_{\lambda_1} B_{\lambda_2} B_{\lambda_3} \cdots = \cdots B_{\lambda_3} B_{\lambda_2} B_{\lambda_1} = B_\lambda , \]
	as required.
\end{proof}

\begin{Lemma} \label{lem:BicommuteGrH}
	For all $1 \leq i,j \leq n$, we have $B_i B_j - B_j B_i \in \mathcal{H}^{\leq i+j-1}$.
\end{Lemma}
\begin{proof}
	In the associated graded $\mathcal{A}$-algebra $\mathrm{Gr}(\mathcal{H})$ of $\mathcal{H}$, we have
	\[ [B_i] \cdot [B_j] - [B_j] \cdot [B_i] = \psi( b_i b_j - b_j b_i ) = 0 \]
	in $\mathrm{Gr}_{i+j}(\mathcal{H}) = \mathcal{H}^{i+j} / \mathcal{H}^{i+j-1}$, by Lemma \ref{lem:BlambdaHnil} and the fact that $\mathcal{B}_\mathrm{nil}$ is commutative by Proposition \ref{prop:Bnilcommutative}.
	This implies that $B_i B_j - B_j B_i \in \mathcal{H}^{\leq i+j-1}$, as claimed.
\end{proof}

\begin{Theorem} \label{thm:Bcommutative}
	The $\mathcal{A}$-subalgebra $\mathcal{B}$ of $\mathcal{H}$ is commutative.
\end{Theorem}
\begin{proof}
	By Corollary \ref{cor:Bgenerators}, it suffices to show that $B_i$ commutes with $B_j$ for $1 \leq i,j \leq n$.
	We prove this claim by induction on $i+j$.
	If $i+j=2$ then $i=j=1$, and $B_1$ clearly commutes with $B_1$.
	Now suppose that $i+j \geq 3$ and that $B_k$ commutes with $B_\ell$ for all $k,\ell \geq 1$ with $k+\ell < i+j$.
	Then we have $\zeta \imath(B_w) = B_w$ for all $w \in W_\mathrm{aff}^+$ with $\ell(w) < i+j$ by Lemma \ref{lem:zetaiBw}.
	We can write
	\[ B_i B_j - B_j B_i = \sum_{w \in W_\mathrm{aff}^+} a_{i,j}^w \cdot B_w \]
	with $a_{i,j}^w \in \mathcal{A}$ for $w \in W_\mathrm{aff}^+$, and we claim that $a_{i,j}^w = 0$ for all $w \in W_\mathrm{aff}^+$ with $\ell(w) \geq i+j$.
	Indeed, if we further write
	\[ B_i B_j - B_j B_i = \sum_{x \in W_\mathrm{aff}} \tilde a_{i,j}^x \cdot H_x \]
	with $\tilde a_{i,j}^w \in \mathcal{A}$ for $x \in W_\mathrm{aff}$ then we must have $\tilde a_{i,j}^w = a_{i,j}^w$ for all $w \in W_\mathrm{aff}^+$ by the definition of $B_w$, and as $B_i B_j - B_j B_i \in \mathcal{H}^{\leq i+j-1}$ by Lemma \ref{lem:BicommuteGrH}, we obtain $a_{i,j}^w = \tilde a_{i,j}^w = 0$ for all $w \in W_\mathrm{aff}^+$ with $\ell(w) \geq i+j$.
	Now using Lemma \ref{lem:zetaiBr} and the fact that $\zeta \imath$ is an anti-automorphism, we compute
	\[ B_i B_j - B_j B_i = \sum_{ \substack{ w \in W_\mathrm{aff}^+ \\ \ell(w) < i+j } } a_{i,j}^w \cdot B_w = \sum_{ \substack{ w \in W_\mathrm{aff}^+ \\ \ell(w) < i+j } } a_{i,j}^w \cdot \zeta \imath( B_w ) = \zeta \imath( B_i B_j - B_j B_i) = B_j B_i - B_i B_j , \]
	and it follows that $B_i B_j = B_j B_i$, as required.
\end{proof}

\subsection{}

Let us briefly explain how our results about $\mathcal{B}$ can be extended to the subalgebra $\mathcal{B}_\mathrm{ext}$ of $\mathcal{H}_\mathrm{ext}$.
Recall from Lemma \ref{lem:Bomega} that we have $\mathcal{B}_\mathrm{ext} = \bigoplus_{\omega \in \Omega} \mathcal{B}_\mathrm{aff} \omega$, so $\mathcal{B}_\mathrm{ext}$ is generated by $\Omega$ and $B_1,\ldots,B_n$ by Corollary \ref{cor:Bgenerators}, and $\mathcal{B}_\mathrm{ext}$ is commutative by Lemma \ref{lem:omegaBr} and Theorem \ref{thm:Bcommutative}.
For $x \in W_\mathrm{ext}^+$, we can uniquely write $x = y \omega$ with $y \in W_\mathrm{aff}^+$ and $\omega \in \Omega$, and we define $B_x = B_y \omega$.
Then by Theorem \ref{thm:BwbasistypeA}, the elements $\{ B_x \mid x \in W_\mathrm{ext}^+ \}$ form an $\mathcal{A}$-basis of $\mathcal{B}_\mathrm{aff}$.
Furthermore, if we write
\[ B_x = \sum_{z \in W_\mathrm{aff}} P_{z,x} \cdot H_z , \]
with $P_{z,x} \in \mathcal{A}$ for all $z \in W_\mathrm{aff}$, then $B_x$ is uniquely determined by the requirements that $P_{x,x} = 1$ and $P_{y,x} = 0$ for $x \neq y \in W_\mathrm{ext}^+$, again by Theorem \ref{thm:BwexistencetypeA}.
We have $P_{z,x} = 0$ for all $z \in W_\mathrm{ext}$ with $\ell(z) > \ell(x)$ and $P_{z,x} \in \Z$ for all $z \in W_\mathrm{ext}$ with $\ell(z) = \ell(x)$ by Remark \ref{rem:Bwproperties}.

\subsection{}
\label{subsec:examplesstandardbasis}

To conclude this section, we provide examples of the standard basis of $\mathcal{B}$ for the root systems of type $\mathsf{A}_1$ and $\mathsf{A}_2$.

\begin{Example} \label{ex:A1standardbasis}
	Let $\Phi$ be of type $\mathsf{A}_1$, and for $m \geq 0$, consider the products $x_m = s_0 s_1 s_0 \cdots$ and $y_m = s_1 s_0 s_1 \cdots$ with $m$ factors.
	Then $x_m$ is the unique element of length $m$ in $W_\mathrm{aff}^+$, and we have
	\[ W_\mathrm{aff}^+ = \{ x_m \mid m \geq 0 \} , \qquad W_\mathrm{aff} = \{ e , x_m , y_m \mid m \geq 1 \} \] 
	We claim that $B_{x_m} = H_{x_m} + H_{y_m}$ for all $m \geq 1$.
	Indeed, we have $B_{x_1} = B_{s_0} = B_1 = H_0 + H_1$ by Remark \ref{rem:BrasBw}.
	Now let $m \geq 2$ and suppose that $B_{x_k} = H_{x_k} + H_{y_k}$ for $1 \leq k \leq m-1$.
	Then we have $B_1 \cdot B_{x_{m-1}} \in \mathcal{B}$, and we compute
	\begin{align*}
		B_1 \cdot B_{x_{m-1}} & = ( H_0 + H_1 ) \cdot ( H_{x_{m-1}} + H_{y_{m-1}} ) \\
		& = H_{x_m} + H_{y_m} + (v^{-1}-v) \cdot H_{x_{m-1}} + H_{y_{m-2}} + (v^{-1}-v) \cdot H_{y_{m-1}} +  H_{x_{m-2}} \\
		& = ( H_{x_m} + H_{y_m} ) + (v^{-1}-v) \cdot B_{x_{m-1}} + (1+\delta_{m,2}) \cdot B_{x_{m-2}} .
	\end{align*}
	In particular, we have $H_{x_m} + H_{y_m} \in \mathcal{B}$, and the uniqueness of $B_{x_m}$ implies that $B_{x_m} = H_{x_m} + H_{y_m}$.
\end{Example}

\begin{Example} \label{ex:A2standardbasis}
	Let $\Phi$ be of type $\mathsf{A}_2$.
	The elements $B_w$ for $w \in W_\mathrm{aff}^+$ with $\ell(w) \leq 4$ are given by
	\[
	\begin{gathered}
		B_e = H_e , \hspace{3cm}
		B_{s_0} = H_0 + H_1 + H_2 , \\
		B_{01} = H_{01} + H_{12} + H_{20} , \hspace{2cm}
		B_{02} = H_{02} + H_{21} + H_{10} , \\
		B_{012} = H_{012} + H_{120} + H_{201} + H_{010} + H_{121} + H_{202} , \\
		B_{021} = H_{021} + H_{210} + H_{102} + H_{010} + H_{121} + H_{202} , \\
		B_{0120} = H_{0120} + H_{1201} + H_{2012} , \hspace{2cm}
		B_{0210} = H_{0210} + H_{2102} + H_{1021} , \\
		B_{0121} = H_{0121} + H_{1202} + H_{2010} + H_{1210} + H_{2021} + H_{0102} + (v^{-1}-v) \cdot ( H_{010} + H_{121} + H_{202} ) .
	\end{gathered}
	\]
	\end{Example}

\section{The canonical basis}
\label{sec:canonicalbases}

\subsection{}

In this section, we discuss the existence of a \emph{canonical basis} in the subalgebra $\mathcal{B}$ of $\mathcal{H}$, which is stable under the bar involution of $\mathcal{H}$, assuming the existence of the standard basis $\{B_w \mid w \in W_\mathrm{aff}^+\}$ of $\mathcal{B}$ (see Conjecture \ref{conj:standardbasis}).
In fact, we impose a slightly stronger assumption.
Recall that the affine Hecke algebra $\mathcal{H}$ has a filtration with filtration pieces
\[ \mathcal{H}^{\leq i} = \mathrm{span}_\mathcal{A}\{ H_w \mid w \in W_\mathrm{aff} \text{ with } \ell(w) \leq i \} . \]
For the entire section, we work under the following hypothesis:
\smallskip

\begin{Hypothesis} \label{hyp:standardbasisfiltration}
	We assume that Conjecture \ref{conj:standardbasis} holds and that $B_w \in \mathcal{H}^{\leq \ell(w)}$ for all $w \in W_\mathrm{aff}^+$.
\end{Hypothesis}

The hypothesis is satisfied for root systems of type $\mathsf{A}_n$ by Remark \ref{rem:Bwproperties} and for root systems of type $\mathsf{B}_2$ and $\mathsf{G}_2$ by Appendix \ref{app:rank2}.

\subsection{}
\label{subsec:KLbasis}

The \emph{bar involution} on $\mathcal{H}$ is the unique involutive ring automorphism $\overline{\phantom{A}} \colon \mathcal{H} \to \mathcal{H}$ such that $\overline{v} = v^{-1}$ and $\overline{H_x} = H_{x^{-1}}^{-1}$ for all $x \in W_\mathrm{aff}$, and we say that $h \in H$ is \emph{self-dual} if $\overline{h} = h$.
For all $x \in W_\mathrm{aff}$, there is a unique self-dual element
\[ \underline{H}_x = \sum_{y \in W_\mathrm{aff}} h_{y,x} \cdot H_x \in \mathcal{H} \]
such that $h_{x,x} = 1$ and $h_{y,x} \in v \cdot \Z[v]$ for all $y \in W_\mathrm{aff}$ with $y \neq x$, and the elements $\{ \underline{H}_x \mid x \in W_\mathrm{aff} \}$ form the \emph{Kazhdan--Lusztig basis} of $\mathcal{H}$, see \cite{SoergelKL}.
Furthermore, the base change coefficients $h_{y,x}$ are the \emph{Kazhdan--Lusztig polynomials}, and they satisfy $h_{y,x} = 0$ unless $y \leq x$.
The key observation which will enable us to define a similar canonical basis in $\mathcal{B}$ is the following lemma.

\begin{Lemma} \label{lem:Bbarinvolution}
	The subalgebra $\mathcal{B} \subseteq \mathcal{H}$ is stable under the bar involution $\overline{\phantom{A}} \colon \mathcal{H} \to \mathcal{H}$.
\end{Lemma}
\begin{proof}
	By Section 3.2.2 in \cite{CherednikDAHA}, the bar involution $\overline{\phantom{A}} \colon \mathcal{H} \to \mathcal{H}$ extends to an involutive ring automorphism $\overline{\phantom{A}} \colon \mathbb{H} \to \mathbb{H}$ such that $\overline{ X_\lambda } = X_{-\lambda}$ for all $\lambda \in X$, and it is straightforward to see that we have $c_\mu(\overline{x}) = \overline{c_{-\mu}(x)}$ and $\varphi(\overline{x}) = \overline{\varphi(x)}$ for all $x \in \mathbb{H}$ and $\mu \in X$.
	In particular, for $b \in \mathcal{B}$ and $\lambda \in X$, we have
	\[ \varphi( \overline{b} X_\lambda ) = \varphi( \overline{b X_{-\lambda}} ) = \overline{ \varphi( b X_{-\lambda} ) } = \overline{b} , \]
	and it follows that $\overline{b} \in \mathcal{B}$, as claimed.
\end{proof}

\subsection{}

We now prove the existence of certain elements $\underline{\tilde B}_w \in \mathcal{B}$ for $w \in W_\mathrm{aff}^+$ whose definition mimics the definition of the Kazhdan--Lusztig basis of $\mathcal{H}$.
Note that these elements are not what we call the \emph{canonical basis} of $\mathcal{B}$; see Corollary \ref{cor:canonicalbasisexistence} and Definition \ref{def:canonicalbasis} below. 
We consider the partial order $\leq_\ell$ on $W_\mathrm{aff}$ that is defined by $x \leq y$ if and only if $x=y$ or $\ell(x) < \ell(y)$.

\begin{Theorem} \label{thm:KLbasisB}
	For all $w \in W_\mathrm{aff}^+$, there is a unique self-dual element $\underline{\tilde B}_w \in \mathcal{B}$ such that
	\[ \underline{\tilde B}_w = \sum_{x \in W_\mathrm{aff}^+} \tilde R_{x,w} \cdot B_x \]
	with $\tilde R_{x,w} = 0$ unless $x \leq_\ell w$ and $\tilde R_{w,w} = 1$ and $\tilde R_{x,w} \in v \cdot \Z[v]$ for $x <_\ell w$.
\end{Theorem}
\begin{proof}
	By \cite[Proposition 2.1]{LaniniRamSobaje}, it suffices to show that $\mathcal{B}$ is a \emph{KL-module} in the sense of \cite[Section 2]{LaniniRamSobaje} with respect to the bar involution, the poset $(W_\mathrm{aff}^+,\leq_\ell)$, and the basis $\{ B_w \mid w \in W_\mathrm{aff} \}$.
	In other words, we need to verify that $\overline{B_w} - B_w$ is an $\mathcal{A}$-linear combination of $\{ B_x \mid x \in W_\mathrm{aff}^+ \text{ with } x <_\ell w \}$, for all $w \in W_\mathrm{aff}^+$.
	To that end, observe that $\overline{B_w} - B_w = \sum_{y \in W_\mathrm{aff}} \overline{P_{y,w} \cdot H_y} - P_{y,w} \cdot H_y$, and let us write
	\[ \overline{P_{y,w} \cdot H_y} - P_{y,w} \cdot H_y = \sum_{z \in W_\mathrm{aff}} q_{z,y} \cdot H_y \]
	with $q_{z,y} \in \mathcal{A}$ for all $z \in W_\mathrm{aff}$.
	Writing $q_z = \sum_{y \in W_\mathrm{aff}} q_{z,y}$, we have $\overline{B_w} - B_w = \sum_{z \in W_\mathrm{aff}} q_z \cdot H_z$, and the definition of the standard basis in Conjecture \ref{conj:standardbasis} implies that
	\[ \overline{B_w} - B_w = \sum_{z \in W_\mathrm{aff}^+} q_z \cdot B_z . \]
	We claim that $q_z = 0$ for all $z \in W_\mathrm{aff}^+$ with $\ell(z) \geq \ell(w)$.
	To prove the claim, we consider the Laurent polynomials $q_{z,y}$ for $y \in W_\mathrm{aff}$.
	\begin{itemize}
		\item If $\ell(y) > \ell(w)$, then we have $P_{y,w} = 0$ by Hypothesis \ref{hyp:standardbasisfiltration}, and therefore $q_{z,y} = 0$ for all $z \in W_\mathrm{aff}$.
		\item If $\ell(y) < \ell(w)$, then we have $H_y , \overline{H}_y \in \mathcal{H}^{\leq \ell(w)-1}$, whence $q_{z,y} = 0$ for all $z \in W_\mathrm{aff}$ with $z \in W_\mathrm{aff}$ with $\ell(z) \geq \ell(w)$.
		\item If $\ell(y) = \ell(w)$ and $y \in W_\mathrm{aff}^+$, then we have $P_{y,w} = \delta_{y,w}$.
		For $y \neq w$, we have $P_{y,w} = 0$ and therefore $q_{z,y} = 0$ for all $z \in W_\mathrm{aff}$.
		For $y \neq w$, we have $P_{w,w} = 1$ and $\overline{H_w} - H_w \in \mathcal{H}^{\leq \ell(w)-1}$ by \cite[Subsection 2.1]{LaniniRamSobaje}, and it follows that $q_{z,w} = 0$ for all $z \in W_\mathrm{aff}$ with $\ell(z) \geq \ell(w)$.
		\item If $\ell(y) = \ell(w)$ and $y \in W_\mathrm{aff} \setminus W_\mathrm{aff}^+$, then we have $\overline{H_y} - H_y \in \mathcal{H}^{\ell(w)-1}$ (again by \cite[Subsection 2.1]{LaniniRamSobaje}) and therefore
		\[ \overline{P_{y,w} \cdot H_y} - P_{y,w} \cdot H_y \in \mathcal{A} H_y + \mathcal{H}^{\leq \ell(w)-1} \]
		in particular, we have $q_{z,y} = 0$ for all $z \in W_\mathrm{aff}$ with $\ell(z) \geq \ell(w)$ and $z \neq y$.
	\end{itemize}
	In conclusion, we have $q_{z,y} = 0$ for all $y \in W_\mathrm{aff}$ and $z \in W_\mathrm{aff}^+$ with $\ell(z) \geq \ell(w)$, and therefore
	\[ q_z = \sum_{y \in W_\mathrm{aff}} q_{z,y} = 0 \]
	for all $z \in W_\mathrm{aff}^+$ with $\ell(z) \geq \ell(w)$, as claimed.
	As
	\[ \overline{B_w} - B_w = \sum_{z \in W_\mathrm{aff}^+} q_z \cdot B_z , \]
	we conclude that $\overline{B_w} - B_w$ is an $\mathcal{A}$-linear combination of $\{ B_x \mid x \in W_\mathrm{aff}^+ \text{ with } x <_\ell w \}$, as required.
\end{proof}

\begin{Remark} \label{rem:KLbasisB}
	The elements $\{ \underline{\tilde B}_w \mid w \in W_\mathrm{aff}^+ \}$ form an $\mathcal{A}$-basis of $\mathcal{B}$ because $\{ B_w \mid w \in W_\mathrm{aff}^+ \}$ is an $\mathcal{A}$-basis of $\mathcal{B}$ and the matrix $(\tilde R_{x,w})_{x,w}$ is triangular with all diagonal entries equal to $1$.
\end{Remark}

\begin{Definition}
	We call $\{ \underline{\tilde B}_w \mid w \in W_\mathrm{aff}^+ \}$ the \emph{Kazhdan--Lusztig basis} of $\mathcal{B}$.
\end{Definition}

\subsection{}

Our next goal is to establish certain properties of the coefficients that arise when we write elements of the Kazhdan--Luszitg basis of $\mathcal{B}$ as linear combinations of the Kazhdan--Lusztig basis of $\mathcal{H}$.
We will use the \emph{inverse Kazhdan--Lusztig polynomials}, which we recall below.

\begin{Remark} \label{rem:inverseKLpolynomials}
	The \emph{inverse Kazhdan--Lusztig polynomials} $h^{x,y} \in \mathcal{A}$ are defined via
\[ H_x = \sum_{y \in W_\mathrm{aff}} (-1)^{\ell(y)+\ell(x)} \cdot h^{x,y} \cdot \underline{H}_y \]
for all $x \in W_\mathrm{aff}$, see \cite[Section 3]{SoergelKL}.
It is straightforward to see that we have $h^{x,y} = 0$ unless $x \geq y$ and $h^{x,x} = 1$ and $h^{x,y} \in v \cdot \Z[v]$ for $x > y$.
\end{Remark}

\begin{Proposition} \label{prop:KLbasisproperties}
	Let $w \in W_\mathrm{aff}^+$ and write
	\[ \underline{\tilde B}_w = \sum_{x \in W_\mathrm{aff}} \tilde Q_{x,w} \cdot \underline{H}_x , \]
	with $\tilde Q_{x,w} \in \mathcal{A}$ for $x \in W_\mathrm{aff}$.
	Then the Laurent polynomials $\tilde Q_{x,w}$ have the following properties.
	\begin{enumerate}
		\item $\tilde Q_{x,w}$ is self-dual for all $x \in W_\mathrm{aff}$, that is, we have $\overline{\tilde Q_{x,w}} = \tilde Q_{x,w}$;
		\item we have $\tilde Q_{x,w} = 0$ for all $x \in W_\mathrm{aff}$ with $\ell(x) > \ell(w)$;
		\item we have $\tilde Q_{w,w} = 1$ and $\tilde Q_{x,w} = 0$ for all $x \in W_\mathrm{aff}^+$ with $x \neq w$ and $\ell(x) = \ell(y)$.
	\end{enumerate}
\end{Proposition}
\begin{proof}
	As $\underline{\tilde B}_w$ and $\underline{H}_x$ are self-dual for all $w \in W_\mathrm{aff}^+$ and $x \in W_\mathrm{aff}$, we have
	\[ \sum_{x \in W_\mathrm{aff}} \tilde Q_{x,w} \cdot \underline{H}_x = \tilde B_w = \overline{\underline{\tilde B}_w} = \sum_{x \in W_\mathrm{aff}} \overline{ \tilde Q_{x,w} \cdot \underline{H}_x } = \sum_{x \in W_\mathrm{aff}} \overline{ \tilde Q_{x,w} } \cdot \underline{H}_x , \]
	and this implies (1) because $\{ \underline{H}_x \mid x \in W_\mathrm{aff} \}$ is an $\mathcal{A}$-basis of $\mathcal{H}$.
	In order to prove (2) and (3), observe that
	\begin{multline*}
		\sum_{x \in W_\mathrm{aff}} \tilde Q_{x,w} = \underline{\tilde B}_w = \sum_{y \in W_\mathrm{aff}^+} \tilde R_{y,w} \cdot B_y = \sum_{y \in W_\mathrm{aff}^+} \sum_{z \in W_\mathrm{aff}} \tilde R_{y,w} \cdot P_{z,y} \cdot H_z \\
		 = \sum_{y \in W_\mathrm{aff}^+} \sum_{z \in W_\mathrm{aff}} \sum_{x \in W_\mathrm{aff}} (-1)^{\ell(x)+\ell(z)} \tilde R_{y,w} \cdot P_{z,y} \cdot h^{z,x} \cdot \underline{H}_x ,
	\end{multline*}
	and therefore
	\[ \tilde Q_{x,w} = \sum_{y \in W_\mathrm{aff}^+} \sum_{z \in W_\mathrm{aff}} (-1)^{\ell(x)+\ell(z)} \tilde R_{y,w} \cdot P_{z,y} \cdot h^{z,x} \]
	for all $x \in W_\mathrm{aff}$.
	Recall from Theorem \ref{thm:KLbasisB} and Remarks \ref{rem:Bwproperties} and \ref{rem:inverseKLpolynomials} that we have $\tilde R_{y,w} = 0$ unless $\ell(y) \leq \ell(w)$ and $P_{z,y} = 0$ unless $\ell(z) \leq \ell(y)$ and $h^{z,x} = 0$ unless $z \geq x$.
	If $\ell(x) > \ell(w)$ then it follows that $\tilde Q_{x,w} = 0$, as claimed in (2).
	If $\ell(x) = \ell(w)$, then for all $y \in W_\mathrm{aff}^+$ and $z \in W_\mathrm{aff}$ such that $\tilde R_{y,w} \cdot P_{z,y} \cdot h^{z,x} \neq 0$, we must have $\ell(x) = \ell(y) = \ell(z)$ and $z \geq x$, and it follows that $z=x$.
	Furthermore, we must have $y=w$ by Theorem \ref{thm:KLbasisB}, and we conclude that $\tilde Q_{x,w} = \tilde R_{w,w} \cdot P_{x,w} \cdot h^{x,x} = P_{x,w}$.
	Now (3) follows from the definition of the Laurent polynomials $P_{x,w}$ in Conjecture \ref{conj:standardbasis}.
\end{proof}

The following corollary is an analogue of Lemma \ref{lem:Bcoefficientnonzero}, with the standard basis of $\mathcal{H}$ replaced by the Kazhdan--Lusztig basis.

\begin{Corollary} \label{cor:BcoefficientnonzeroKL}
	Let $b \in \mathcal{B}$ and write
	\[ b = \sum_{x \in W_\mathrm{aff}} a_x \cdot \underline{H}_x , \]
	with $a_x \in \mathcal{A}$ for $x \in W_\mathrm{aff}$.
	If $a_x = 0$ for all $x \in W_\mathrm{aff}^+$ then $b=0$.
\end{Corollary}
\begin{proof}
	By Remark \ref{rem:KLbasisB}, we can write $b = \sum_y c_y \cdot \underline{\tilde B}_y$, with $c_y \in \mathcal{A}$ for $y \in W_\mathrm{aff}^+$, and it follows that
	\[ a_x = \sum_{y \in W_\mathrm{aff}^+} c_y \cdot \tilde Q_{x,y} \]
	for all $x \in W_\mathrm{aff}$.
	If $b \neq 0$, then we can choose $x \in W_\mathrm{aff}^+$ of maximal length with the property that $c_x \neq 0$, so $c_y = 0$ for all $y \in W_\mathrm{aff}^+$ with $\ell(y) > \ell(x)$.
	For $y \in W_\mathrm{aff}^+$ with $\ell(y) \leq \ell(x)$, we have $\tilde Q_{x,y} = 0$ if $x \neq y$ and $\tilde Q_{x,x} = 1$ by Proposition \ref{prop:KLbasisproperties}, and we conclude that $a_x = c_x \neq 0$, as required.
\end{proof}

\subsection{}

Now we are ready to establish the existence of what we call the \emph{canonical basis} of $\mathcal{B}$.

\begin{Corollary} \label{cor:canonicalbasisexistence}
	For all $w \in W_\mathrm{aff}^+$, there is a unique element $\underline{B}_w \in \mathcal{B}$ whose expansion
	\[ \underline{B}_w = \sum_{x \in W_\mathrm{aff}} Q_{x,w} \cdot \underline{H}_x \]
	in terms of the Kazhdan--Lusztig basis of $\mathcal{H}$ satisfies $Q_{w,w} = 1$ and $Q_{x,w} = 0$ for all $w \neq x \in W_\mathrm{aff}^+$.
	Furthermore, $\underline{B}_w$ is self-dual.
\end{Corollary}
\begin{proof}
	By Proposition \ref{prop:KLbasisproperties}, we can define $\underline{B}_w$ by induction on $\ell(w)$, via $\underline{B}_e = \underline{\tilde B}_e = H_e$ and
	\[ \underline{B}_w = \underline{\tilde B}_w - \sum_{ \substack{ x \in W_\mathrm{aff}^+ \\ \ell(x) < \ell(w) }  }  \tilde Q_{x,w} \cdot \underline{B}_x \]
	for $w \in W_\mathrm{aff}^+$ with $\ell(w)>0$.
	The uniqueness of $\underline{B}_w$ follows from Corollary \ref{cor:BcoefficientnonzeroKL} (compare the proof of Lemma \ref{lem:Bwunique}), and $\underline{B}_w$ is self-dual because $\underline{\tilde B}_w$ and $\tilde Q_{x,w}$ are self-dual for $x \in W_\mathrm{aff}$, by Theorem \ref{thm:KLbasisB} and Proposition \ref{prop:KLbasisproperties}.
\end{proof}

\begin{Corollary} \label{cor:canonicalbasisproperties}
	For $w \in W_\mathrm{aff}^+$ and $x \in W_\mathrm{aff}$, we have $Q_{x,w} = 0$ unless $\ell(x) \leq \ell(w)$.
\end{Corollary}
\begin{proof}
	By the proof of Corollary \ref{cor:canonicalbasisexistence}, we have
	\[ \underline{B}_w = \underline{\tilde B}_w - \sum_{ \substack{ x \in W_\mathrm{aff}^+ \\ \ell(x) < \ell(w) }  }  \tilde Q_{x,w} \cdot \underline{B}_x , \]
	and the claim follows by induction on $\ell(w)$ using Proposition \ref{prop:KLbasisproperties}.
\end{proof}

\begin{Corollary} \label{cor:canonicalbasis}
	The elements $\{ \underline{B}_w \mid w \in W_\mathrm{aff}^+ \}$ form an $\mathcal{A}$-basis of $\mathcal{B}$.
\end{Corollary}
\begin{proof}
	The elements $\{ \underline{B}_w \mid w \in W_\mathrm{aff}^+ \}$ are linearly independent by the linear independence of the Kazhdan--Lusztig basis $\{ \underline{H}_x \mid x \in W_\mathrm{aff} \}$ of $\mathcal{H}$, and they span $\mathcal{B}$ because the elements of the Kazhdan--Lusztig basis $\{ \underline{\tilde B}_x \mid x \in W_\mathrm{aff}^+ \}$ of $\mathcal{B}$ can be written as $\underline{\tilde B}_x = \sum_w \tilde Q_{w,x} \cdot \underline{B}_w$ for all $x \in W_\mathrm{aff}^+$ by the proof of Corollary \ref{cor:canonicalbasisexistence}.
\end{proof}

\begin{Definition} \label{def:canonicalbasis}
	We call $\{ \underline{B}_x \mid x \in W_\mathrm{aff}^+ \}$ the \emph{canonical basis} of $\mathcal{B}$.
\end{Definition}

For $w \in W_\mathrm{aff}^+$ with reduced expression $w = s_{i_1} s_{i_2} \cdots s_{i_r}$, we sometimes write $\underline{B}_w = \underline{B}_{i_1 i_2 \cdots i_r}$.

\subsection{}
\label{subsec:canonicalbasisBext}

The existence of the Kazhdan--Lusztig basis and the canonical basis in $\mathcal{B}$ implies that there are analogous bases in the subalgebra $\mathcal{B}_\mathrm{ext}$ of $\mathcal{H}_\mathrm{ext}$.
Indeed the bar involution on $\mathcal{H}$ readily extends to a bar involution on $\mathcal{H}_\mathrm{ext}$, and the Kazhdan--Lusztig basis of $\mathcal{H}$ gives rise to a Kazhdan--Lusztig basis in $\mathcal{H}_\mathrm{ext}$ via $\underline{H}_{x \omega} = \underline{H}_x \omega$ for $x \in W_\mathrm{aff}$ and $\omega \in \Omega$.
As in Lemma \ref{lem:Bbarinvolution}, one sees that $\mathcal{B}_\mathrm{ext}$ is stable under the bar involution, and the Kazhdan--Lusztig basis $\{ \underline{\tilde B}_x \mid x \in W_\mathrm{ext}^+ \}$ and the canonical basis $\{ \underline{B}_x \mid x \in W_\mathrm{ext}^+ \}$ of $\mathcal{B}_\mathrm{ext}$ can be defined via $\underline{\tilde B}_{y \omega} = \underline{\tilde B}_y \omega$ and $\underline{B}_{y\omega} = \underline{B}_y \omega$ for $y \in W_\mathrm{aff}^+$ and $\omega \in \Omega$ (see Lemma \ref{lem:Bomega}).
One easily checks that the Kazhdan--Lusztig basis satisfies the conditions from Theorem \ref{thm:KLbasisB}, and that the canonical basis satisfies the conditions from Corollary \ref{cor:canonicalbasisexistence}, with $W_\mathrm{aff}$ replaced by $W_\mathrm{ext}$.
In particular, if
\[ \underline{B}_x = \sum_{y \in W_\mathrm{ext}} Q_{y,x} \cdot \underline{H}_x , \]
with $Q_{y,x} \in \mathcal{A}$ for $y \in W_\mathrm{ext}$, then we have $Q_{x,x} = 1$ and $Q_{y,x} = 0$ for $x neq y \in W_\mathrm{aff}^+$.

\subsection{}

Next we want to demonstrate that the canonical basis does not in general coincide with the Kazhdan--Lusztig basis of $\mathcal{B}$.
This will follow from Corollary \ref{cor:canonicalbasisequalsKLbasiscondition} and Example \ref{ex:A2canonicalbasis} below.
For $w \in W_\mathrm{aff}^+$, let us write
\begin{equation} \label{eq:canonicalbasistostandardbasis}
	\underline{B}_w = \sum_{y \in W_\mathrm{aff}^+} R_{y,w} \cdot B_y
\end{equation}
with $R_{y,w} \in \mathcal{A}$ for all $y \in W_\mathrm{aff}^+$.

\begin{Lemma} \label{lem:canonicalbasistostandardbasis}
	For all $w,y \in W_\mathrm{aff}^+$, we have
	\[ R_{y,w} = \sum_{x \in W_\mathrm{aff}} h_{y,x} \cdot Q_{x,w} \]
	In particular, we have $R_{y,w} = 0$ unless $y \leq_\ell w$ and $R_{w,w} = 1$.
\end{Lemma}
\begin{proof}
	For all $w \in W_\mathrm{aff}^+$, we have
	\[ \sum_{x,z \in W_\mathrm{aff}} h_{z,x} \cdot Q_{x,w} \cdot H_z = \sum_{x \in W_\mathrm{aff}} Q_{x,w} \cdot \underline{H}_x = \underline{B}_w = \sum_{y \in W_\mathrm{aff}^+} R_{y,w} \cdot B_y = \sum_{z \in W_\mathrm{aff}} \sum_{y \in W_\mathrm{aff}^+} P_{z,y} \cdot R_{y,w} \cdot H_z , \]
	and it follows that
	\[ \sum_{y \in W_\mathrm{aff}^+} P_{z,y} \cdot R_{y,w} = \sum_{x \in W_\mathrm{aff}} h_{z,x} \cdot Q_{x,w} \]
	for all $z \in W_\mathrm{aff}$.
	For $y,z \in W_\mathrm{aff}^+$, we have $P_{z,y} = \delta_{y,z}$ by the definition in Conjecture \ref{conj:standardbasis}, and so $R_{y,w} = \sum_x h_{y,x} \cdot Q_{x,w}$, as claimed.
	If $R_{y,w} \neq 0$ then there is an element $x \in W_\mathrm{aff}$ such that $h_{y,x} \neq 0$ and $Q_{x,w} \neq 0$, and it follows that $y \leq x$ and $\ell(x) \leq \ell(w)$ by Corollary \ref{cor:canonicalbasisproperties}.
	By Corollary \ref{cor:canonicalbasisexistence}, we further have $x = w$ or $x \notin W_\mathrm{aff}^+$, and as $y \in W_\mathrm{aff}^+$, we conclude that $y \leq_\ell w$.	
	For $y = w$, we obtain $R_{w,w} = h_{w,w} \cdot Q_{w,w} = 1$.
\end{proof}

\begin{Corollary} \label{cor:canonicalbasisequalsKLbasiscondition}
	For $w \in W_\mathrm{aff}^+$, we have $\underline{B}_w = \underline{\tilde B}_w$ if and only if $R_{y,w} \in v \cdot \Z[v]$ for all $y \in W_\mathrm{aff}^+ \setminus \{ w \}$.	
\end{Corollary}
\begin{proof}
	This follows from the uniqueness of $\underline{\tilde B}_w$ in Theorem \ref{thm:KLbasisB}, together with the facts that $\underline{B}_w$ is self dual (see Corollary \ref{cor:canonicalbasisexistence}) and that $R_{y,w} = 0$ unless $y \leq_\ell w$ and $R_{w,w} = 1$ (see Lemma \ref{lem:canonicalbasistostandardbasis}).
\end{proof}

\subsection{}
\label{subsec:examplescanonicalbasis}

For affine Hecke algebras of type $\mathsf{A}_1$, we can explicitly compute the canonical basis and the Kazhdan--Lusztig basis of $\mathcal{B}$.
We also compute some example in type $\mathsf{A}_2$ below.

\begin{Example} \label{ex:A1canonicalbasis}
	Let $\Phi$ be of type $\mathsf{A}_1$, and recall from Example \ref{ex:A1standardbasis} that we write
	\[ W_\mathrm{aff}^+ = \{ x_m \mid m \geq 0 \} , W_\mathrm{aff} = \{ e , x_m , y_m \mid m \geq 1 \} \]
	with $\ell(x_m) = \ell(y_m) = m$ for all $m \geq 0$.
	The Kazhdan--Lusztig basis for $W_\mathrm{aff}$ is given by
	\begin{align*}
		\underline{H}_{x_m} & = H_{x_m} + \sum_{i=1}^{m-1} v^i \cdot ( H_{x_{m-i}} + H_{y_{m-i}} ) + v^m \cdot H_e , \\
		\underline{H}_{y_m} & = H_{y_m} + \sum_{i=1}^{m-1} v^i \cdot ( H_{x_{m-i}} + H_{y_{m-i}} ) + v^m \cdot H_e .
	\end{align*}
	Now we have
	\begin{multline*}
		\qquad\qquad
		\underline{H}_{x_m} + \underline{H}_{y_m} = ( H_{x_m} + H_{y_m} ) + \sum_{i=1}^{m-1} 2 v^i \cdot ( H_{x_{m-i}} + H_{y_{m-i}} ) + 2 v^m \cdot H_e \\
		 = B_{x_m} + \sum_{i=1}^m 2 v^i \cdot B_{x_{m-i}} \in \mathcal{B} ,
		\qquad\qquad
	\end{multline*}
	and this implies that $\underline{B}_{x_m} = \underline{\tilde B}_{x_m} = \underline{H}_{x_m} + \underline{H}_{y_m}$ for all $m \geq 1$.
\end{Example}

\begin{Example} \label{ex:A2canonicalbasis}
	Let $\Phi$ be of type $\mathsf{A}_2$.
	We compute the elements $\underline{B}_w$ of the canonical basis of $\mathcal{B}$ and the elements $\underline{\tilde B}_w$ of the Kazhdan--Lusztig basis of $\mathcal{B}$ for $w \in W_\mathrm{aff}^+$ with $\ell(w) \leq 4$.
	First observe that we have
	\[ \underline{H}_0 + \underline{H}_1 + \underline{H}_2 = ( H_0 + H_1 + H_2 ) + 3v \cdot H_e = B_0 + 3v \cdot B_e \in \mathcal{B} \]
	by Example \ref{ex:A2standardbasis}, and so
	\[ \underline{B}_0 = \underline{\tilde B}_0 = \underline{H}_0 + \underline{H}_1 + \underline{H}_2 . \]
	Analogously, we obtain
	\begin{multline*}
		\qquad \underline{H}_{01} + \underline{H}_{12} + \underline{H}_{20} = ( H_{01} + H_{12} + H_{20} ) + 2v \cdot ( H_0 + H_1 + H_2 ) + 3v^2 \cdot H_e \\ = B_{01} + 2v \cdot B_0 + 3 v^2 \cdot B_e \in \mathcal{B} , \qquad
	\end{multline*}
	and it follows that
	\[ \underline{B}_{01} = \underline{\tilde B}_{01} = \underline{H}_{01} + \underline{H}_{12} + \underline{H}_{20} , \]
	and by symmetry
	\[ \underline{B}_{02} = \underline{\tilde B}_{02} = \underline{H}_{02} + \underline{H}_{21} + \underline{H}_{10} . \]
	Alternatively, one easily checks that $\underline{B}_{02} = \underline{B}_0 \cdot \underline{B}_0 - \underline{B}_{01} - (v+v^{-1}) \cdot \underline{B}_0$.
	Next, an explicit computation shows that
	\[ \underline{B}_{01} \cdot \underline{B}_0 - (v+v^{-1}) \cdot \underline{B}_{01} - \underline{B}_0 = \underline{H}_{012} + \underline{H}_{120} + \underline{H}_{201} + \underline{H}_{010} + \underline{H}_{121} + \underline{H}_{202} , \]
	and we conclude that
	\begin{align*}
		\underline{B}_{012} & = \underline{H}_{012} + \underline{H}_{120} + \underline{H}_{201} + \underline{H}_{010} + \underline{H}_{121} + \underline{H}_{202} , \\
		\underline{B}_{021} & = \underline{H}_{021} + \underline{H}_{210} + \underline{H}_{102} + \underline{H}_{010} + \underline{H}_{121} + \underline{H}_{202} .
	\end{align*}
	We further have $\underline{B}_{012} = B_{012} + 3v \cdot B_{01} + 5v^2 \cdot B_0 + 6v^3 \cdot B_e$, whence $\underline{B}_{012} = \underline{\tilde B}_{012}$ and $\underline{B}_{021} = \underline{\tilde B}_{021}$.
	Analogously, we compute
	\[ \underline{B}_{01} \cdot \underline{B}_{01} = ( \underline{H}_{0120} + \underline{H}_{1201} + \underline{H}_{2012} ) + (v+v^{-1}) \cdot B_{012} + B_{01} , \]
	and it follows that
	\begin{align*}
		\underline{B}_{0120} & = \underline{H}_{0120} + \underline{H}_{1201} + \underline{H}_{2012} , \\
		\underline{B}_{0210} & = \underline{H}_{0210} + \underline{H}_{1021} + \underline{H}_{2102} .
	\end{align*}
	By computing the Kazhdan--Lusztig basis elements in these formulas, one sees that
	\[ \underline{B}_{0120} = B_{0120} + 2v \cdot B_{012} + 3 v^2 \cdot B_{01} + 2 v^2 \cdot B_{02} + (3v^3+3v) \cdot B_0 + 3 \cdot (v^4+v^2) \cdot B_e , \]
	whence $\underline{B}_{0120} = \underline{\tilde B}_{0120}$ and similarly $\underline{B}_{0210} = \underline{\tilde B}_{0210}$.
	Finally, we compute
	\begin{align*}
		\underline{B}_{012} \cdot \underline{B}_{0} & = \underline{H}_{0121} + \underline{H}_{1202} + \underline{H}_{2010} + \underline{H}_{1210} + \underline{H}_{2021} + \underline{H}_{0102} + \underline{H}_{0120} + \underline{H}_{1201} + \underline{H}_{2012} \\
		& \hspace{2cm} + (v+v^{-1}) \cdot ( \underline{H}_{012} + \underline{H}_{120} + \underline{H}_{201} ) + 2 \cdot (v+v^{-1}) \cdot ( \underline{H}_{010} + \underline{H}_{121} + \underline{H}_{202} ) \\
		& \hspace{2cm} + ( \underline{H}_{01} + \underline{H}_{12} + \underline{H}_{20} ) \\
		& = \underline{H}_{0121} + \underline{H}_{1202} + \underline{H}_{2010} + \underline{H}_{1210} + \underline{H}_{2021} + \underline{H}_{0102} + (v+v^{-1}) \cdot ( \underline{H}_{010} + \underline{H}_{121} + \underline{H}_{202} ) \\
		& \hspace{2cm} + \underline{B}_{0120} + (v+v^{-1}) \cdot \underline{B}_{012} + \underline{B}_{01} ,
	\end{align*}
	and it follows that
	\[ \underline{B}_{0121} = \underline{H}_{0121} + \underline{H}_{1202} + \underline{H}_{2010} + \underline{H}_{1210} + \underline{H}_{2021} + \underline{H}_{0102} + (v+v^{-1}) \cdot ( \underline{H}_{010} + \underline{H}_{121} + \underline{H}_{202} ) . \]
	We want to compute the Laurent polynomial $R_{y,w}$ from \eqref{eq:canonicalbasistostandardbasis} for $w = 0120$ and $y = 01$.
	One easily checks that
	\begin{alignat*}{2}
		h_{y,x} & = v^2 \qquad && \text{for } x \in \{ 0121 , 2010 , 2021 , 0102 \} , \\
		h_{y,x} & = 0 \qquad && \text{for } x \in \{ 1202 , 1210 , 121 , 202 \} , \\
		h_{y,x} & = v \qquad && \text{for } x = 010 ,
	\end{alignat*}
	and using Lemma \ref{lem:canonicalbasistostandardbasis}, it follows that
	\[ R_{y,w} = 4 \cdot v^2 + v \cdot (v+v^{-1}) = 5 v^2 + 1 . \]
	As $R_{y,w} \notin v \cdot \Z[v]$, Corollary \ref{cor:canonicalbasisequalsKLbasiscondition} implies that $\underline{B}_{0121} \neq \underline{\tilde B}_{0121}$.
	More precisely, we have
	\[ \underline{B}_{0121} = B_{0121} + 2v \cdot ( B_{012} + B_{021} ) + (5v^2+1) \cdot ( B_{01} + B_{02} ) + ( 7v^3 + v ) \cdot B_0 + ( 9 v^4 + 3 v^2 ) \cdot B_e , \]
	and it follows that
	\[ \underline{\tilde B}_{0121} = \underline{B}_{0121} - \underline{B}_{01} - \underline{B}_{02} . \]
	Further note that this implies that the Laurent polynomials in an expansion of $\underline{\tilde B}_w$ in terms of the Kazhdan--Lusztig basis of $\mathcal{H}$ need not have non-negative coefficients.
\end{Example}

\subsection{}
\label{subsec:positivityconjectures}

Motivated by the computations in Examples \ref{ex:A1canonicalbasis} and \ref{ex:A2canonicalbasis}, we propose the following positivity conjectures (the first and third of which have already been stated in the introduction).

\begin{Conjecture} \label{conj:positivity}
\begin{enumerate}
	\item The Laurent polynomials $Q_{x,w} \in \mathcal{A}$ in the expansion
	\[ \underline{B}_w = \sum_{x \in W_\mathrm{aff}} Q_{x,w} \cdot \underline{H}_x \]
	have non-negative coefficients for all $x \in W_\mathrm{aff}$ and $w \in W_\mathrm{aff}^+$.
	\item The polynomials $\tilde R_{y,w} \in \Z[v]$ in the expansion
	\[ \underline{\tilde B}_w = \sum_{y \in W_\mathrm{aff}^+} \tilde R_{y,w} \cdot B_y \]
	have non-negative coefficients for all $w,y \in W_\mathrm{aff}^+$.
	\item The Laurent polynomials $\mu_{w,w'}^z \in \mathcal{A}$ in the expansion
	\[ \underline{B}_w \cdot \underline{B}_{w'} = \sum_{z \in W_\mathrm{aff}^+} \mu_{w,w'}^z \cdot \underline{B}_z \]
	have non-negative coefficients for all $w,w',z \in W_\mathrm{aff}^+$.
\end{enumerate}
\end{Conjecture}

Observe that parts (1) and (2) of Conjecture \ref{conj:positivity} are valid
\begin{itemize}
	\item in type $\mathrm{A}_1$ by Example \ref{ex:A1canonicalbasis},
	\item in type $\mathrm{A}_2$ for elements $w \in W_\mathrm{aff}^+$ with $\ell(w) \leq 4$ by Example \ref{ex:A2canonicalbasis}.%
	\footnote{In fact, we have checked part (1) of Conjecture \ref{conj:positivity} in type $\mathrm{A}_2$ for elements $w \in W_\mathrm{aff}^+$ with $\ell(w) \leq 12$ using a computer.}
\end{itemize}
By the following proposition, this implies that part (3) of Conjecture \ref{conj:positivity} is also valid in type $\mathrm{A}_1$ (and in type $\mathrm{A}_2$ for elements $w,w' \in W_\mathrm{aff}$ with $\ell(w) , \ell(w') \leq 4$).

\begin{Proposition}
	Let $w,w' \in W_\mathrm{aff}^+$ and suppose that the Laurent polynomials $Q_{x,w}$ and $Q_{y,w'}$ have non-negative coefficients for all $x,y \in W_\mathrm{aff}$.
	Then the Laurent polynomial $\mu_{w,w'}^z$ has non-negative coefficients for all $z \in W_\mathrm{aff}^+$.
	In particular, part (1) of Conjecture \ref{conj:positivity} implies part (3) of Conjecture \ref{conj:positivity}.
\end{Proposition}
\begin{proof}
	For $x,y \in W_\mathrm{aff}$, let us write $\underline{H}_x \cdot \underline{H}_y = \sum_z h_{x,y}^z \cdot \underline{H}_z$, with $h_{x,y}^z \in \mathcal{A}$ for all $z \in W_\mathrm{aff}$.
	The Laurent polynomials $h_{x,y}^z$ have non-negative coefficients for all $x,y,z \in W_\mathrm{aff}$ by Corollary 1.2 in \cite{EliasWilliamsonHodgetheory}.
	For $w , w' \in W_\mathrm{aff}^+$, we have
	\[ \underline{B}_w \cdot \underline{B}_{w'} = \sum_{ x,y \in W_\mathrm{aff} } Q_{x,w} \cdot Q_{y,w'} \cdot \underline{H}_x \cdot \underline{H}_y = \sum_{z \in W_\mathrm{aff}} \Big( \sum_{ x,y \in W_\mathrm{aff} } Q_{x,w} \cdot Q_{y,w'} \cdot h_{x,y}^z \Big) \cdot \underline{H}_z , \]
	and using the definition of the canonical basis (see Corollary \ref{cor:canonicalbasisexistence}), it follows that
	\[ \mu_{w,w'}^z = \sum_{ x,y \in W_\mathrm{aff} } Q_{x,w} \cdot Q_{y,w'} \cdot h_{x,y}^z \]
	for all $z \in W_\mathrm{aff}^+$.
	In particular, if the Laurent polynomials $Q_{x,w}$ and $Q_{y,w'}$ have non-negative coefficients for all $x,y \in W_\mathrm{aff}$ then the Laurent polynomial $\mu_{w,w'}^z$ has non-negative coefficients for all $z \in W_\mathrm{aff}^+$.
\end{proof}

\section{The center of the affine Hecke algebra}
\label{sec:center}

Recall that the center $Z(\mathcal{H}_\mathrm{ext})$ of the extended affine Hecke algebra $\mathcal{H}_\mathrm{ext}$ is contained in the subalgebra $\mathcal{B}_\mathrm{ext}$ of $\mathcal{H}_\mathrm{ext}$ by Corollary \ref{cor:centerinB}.
In this section, we explain how the chain of inclusions
\[ Z(\mathcal{H}_\mathrm{ext}) \subseteq \mathcal{B}_\mathrm{ext} \subseteq \mathcal{H}_\mathrm{ext} \]
can be used to study combinatorial properties of the center.
More precisely, we will show that the problem of expressing the elements of a certain canonical basis of $Z(\mathcal{H}_\mathrm{ext})$ in terms of the Kazhdan--Lusztig basis of $\mathcal{H}_\mathrm{ext}$ can be reduced to computing the canonical basis of $\mathcal{B}$ (if it exists).
In order to state the problem more explicitly, we first recall some results about $Z(\mathcal{H}_\mathrm{ext})$.

\subsection{}
\label{subsec:centerHext}

Recall from Subsection \ref{subsec:Bcontainscenter} that the center $Z(\mathcal{H}_\mathrm{ext})$ of $\mathcal{H}_\mathrm{ext}$ can be described as the ring of $W_\mathrm{fin}$-invariants in $\mathcal{A}[Y]$, that is
\[ Z(\mathcal{H}_\mathrm{ext}) = \mathcal{A}[Y]^{W_\mathrm{fin}} = \Big\{ \sum\nolimits_\lambda a_\lambda \cdot Y_\lambda \in \mathcal{A}[Y] \mathop{\Big|} a_\lambda = a_{w(\lambda)} \text{ for all } w \in W_\mathrm{fin} \text{ and } \lambda \in Y \Big\} . \]
Thus, the characters of finite-dimensional representations of the complex simple Lie algebra $\mathfrak{g}^\vee$ with root system $\Phi^\vee$ afford canonical elements of $Z(\mathcal{H}_\mathrm{ext})$.
More precisely, for a dominant coweight $\lambda \in Y^+$, let $V(\lambda)$ be the simple $\mathfrak{g}^\vee$-module of highest weight $\lambda$ and write $V(\lambda) = \bigoplus_\mu V(\lambda)_\mu$ for the weight space decomposition of $V(\lambda)$.
Then we have
\[ \chi_\lambda \coloneqq \sum_{\mu \in X} \dim V(\lambda)_\mu \cdot Y_\mu \in Z(\mathcal{H}_\mathrm{ext}) . \]
The elements $\{ \chi_\lambda \mid \lambda \in Y^+ \}$ form an $\mathcal{A}$-basis of $Z(\mathcal{H}_\mathrm{ext})$, and they are self-dual with respect to the bar involution on $\mathcal{H}_\mathrm{ext}$ (see Corollary 8.8 in \cite{LusztigSpherical} and its proof).
Furthermore, they are related to the Kazhdan--Lusztig basis of $\mathcal{H}_\mathrm{ext}$ via the fact that
\begin{equation} \label{eq:chilambdaKLbasis}
	\underline{H}_{w_0} \cdot \chi_\lambda = \underline{H}_{w_0 t_\lambda}
\end{equation}
where $\lambda \in Y^+$ and $w_0 \in W_\mathrm{fin}$ denotes the longest element; see Theorem 6.2 in \cite{KnopSpherical}.

For $\lambda \in X^+$ and $x \in W_\mathrm{ext}$, we can define a symmetric Laurent polynomial $a_{x,\lambda} \in \mathcal{A}$ via
\[ \chi_\lambda = \sum_{x \in W_\mathrm{ext}} a_{x,\lambda} \cdot \underline{H}_x . \]
As explained in \cite[Section 5.3.4]{AcharRichecentral}, computing the Laurent polynomials $a_{x,\lambda}$ is equivalent to the problem of finding the graded composition multiplicities in (a mixed version of) Gaitsgory's \emph{central sheaves} on affine flag varieties \cite{Gaitsgorycentral}.

\subsection{}
\label{subsec:chilambdaviacanonicalbasis}

For the rest of this section, we assume the existence of the canonical basis of $\mathcal{B}$.
More precisely, we work under the following hypothesis:

\begin{Hypothesis} \label{hyp:canonicalbasis}
	We assume that for all $w \in W_\mathrm{aff}^+$, there is an element $\underline{B}_w \in \mathcal{B}$ whose expansion
	\[ \underline{B}_w = \sum_{x \in W_\mathrm{aff}} Q_{x,w} \cdot \underline{H}_x \]
	in terms of the Kazhdan--Lusztig basis of $\mathcal{H}$ satisfies $Q_{w,w} = 1$ and $Q_{x,w} = 0$ for all $w \neq x \in W_\mathrm{aff}^+$.
\end{Hypothesis}

Note that Hypothesis \ref{hyp:standardbasisfiltration} implies Hypothesis \ref{hyp:canonicalbasis} by Corollary \ref{cor:canonicalbasisexistence}.
In particular, Hypothesis \ref{hyp:canonicalbasis} holds for root systems of type $\mathsf{A}_n$ and $\mathsf{B}_2$ and $\mathsf{G}_2$.

The elements $\{ \underline{B}_w \mid w \in W_\mathrm{aff}^+ \}$ form the \emph{canonical basis} of $\mathcal{B}$ (see Corollary \ref{cor:canonicalbasis} and Definition \ref{def:canonicalbasis}).
As in Subsection \ref{subsec:canonicalbasisBext}, we can also define a canonical basis $\{ \underline{B}_w \mid w \in W_\mathrm{ext}^+ \}$ of $\mathcal{B}_\mathrm{ext}$ with $\underline{B}_{w\omega} = \underline{B}_w \omega$ for all $w \in W_\mathrm{ext}^+$ and $\omega \in \Omega$.
By Corollary \ref{cor:centerinB}, we have $Z(\mathcal{H}_\mathrm{ext}) \subseteq \mathcal{B}_\mathrm{ext}$, and so we can define Laurent polynomials $b_{w,\lambda} \in \mathcal{A}$, for $w \in W_\mathrm{ext}^+$ and $\lambda \in X^+$, via
\[ \chi_\lambda = \sum_{w \in W_\mathrm{ext}^+} b_{w,\lambda} \cdot \underline{B}_w . \]
Then we have
\[ \chi_\lambda = \sum_{w \in W_\mathrm{ext}^+} b_{w,\lambda} \cdot \underline{B}_w = \sum_{x \in W_\mathrm{ext}} \Big( \sum_{w \in W_\mathrm{ext}^+} Q_{x,w} \cdot b_{w,\lambda} \Big) \cdot \underline{H}_x \]
for all $\lambda \in X^+$, whence the Laurent polynomials $a_{x,\lambda}$ from Subsection \ref{subsec:centerHext} can be computed via
\[ a_{x,\lambda} = \sum_{w \in W_\mathrm{ext}^+} Q_{x,w} \cdot b_{w,\lambda} \]
for $x \in W_\mathrm{ext}$ and $\lambda \in X^+$.
Note that for $w' \in W_\mathrm{ext}^+$, we have $Q_{w',w} = \delta_{w,w'}$ by the above hypothesis, and therefore $a_{w',\lambda} = b_{w',\lambda}$.
We will prove in Theorem \ref{thm:centercanonicalbasiscoefficients} below that the Laurent polynomials $b_{w,\lambda} = a_{w,\lambda}$ (for $w \in W_\mathrm{ext}^+$) can be expressed explicitly in terms of parabolic Kazhdan--Lusztig polynomials, so that the problem of computing the Laurent polynomials $a_{x,\lambda}$ (for $x \in W_\mathrm{ext}$) can be reduced to the problem of computing the Laurent polynomials $Q_{x,w}$.
This, in our view, underlines the importance of understanding the canonical basis of $\mathcal{B}_\mathrm{ext}$.

\subsection{}
\label{subsec:negativeKL}

Before we can explain how the Laurent polynomials $b_{w,\lambda} = a_{w,\lambda}$ can be computed for $w \in W_\mathrm{ext}^+$ and $\lambda \in X^+$, we need to introduce additional notation and recall some classical results about Hecke algebras.
Let us write $d = \overline{\phantom{A}} \colon \mathcal{H}_\mathrm{ext} \to \mathcal{H}_\mathrm{ext}$ for the bar involution and consider the involutive anti-automorphisms $i$ and $a$ of $\mathcal{H}_\mathrm{ext}$ with
\[ i(v) = a(v) = v , \qquad i(H_x) = H_{x^{-1}} , \qquad a(H_x) = (-1)^{\ell(x)} \cdot H_x^{-1} \]
for all $x \in W_\mathrm{ext}$, as in the proof of Theorem 2.7 in \cite{SoergelKL}.
The maps $d$, $i$ and $a$ commute with one another pairwise, and the composite $dia$ satisfies
\[ dia(H_x) = (-1)^{\ell(x)} \cdot H_x . \]
Apart from the Kazhdan--Lusztig basis, there is a second canonical $\mathcal{A}$-basis $\{ \underline{\tilde H}_x \mid x \in W_\mathrm{ext} \}$ in $\mathcal{H}_\mathrm{ext}$, which is uniquely determined by the properties that $\underline{\tilde H}_x$ is self dual and
\[ \underline{\tilde H}_x \in H_x + \sum_{y<x} v^{-1} \cdot \Z[v^{-1}] \cdot H_y \]
for all $x \in W_\mathrm{ext}$.
It is related to the Kazhdan-Lusztig basis and the standard basis of $_\mathrm{ext}$ via
\begin{equation} \label{eq:negativeKL}
	\underline{\tilde H}_x = (-1)^{\ell(x)} \cdot dia( \underline{H}_x ) = \sum_y (-1)^{\ell(x)+\ell(y)} \cdot \overline{h_{y,x}} \cdot H_y ,
\end{equation}
see again the proof of Theorem 2.7 in \cite{SoergelKL}.

Next, for a weight $\lambda \in Y$ and $\mu,\nu \in Y^+$ such that $\lambda = \mu - \nu$, we have
\[ dia( Y_\lambda ) = dia( H_{t_\mu} \cdot H_{t_\nu}^{-1} ) = (-1)^{\ell(t_\mu) - \ell(t_\nu)} H_{t_\mu} \cdot H_{t_\nu}^{-1} = (-1)^{\ell(t_\lambda)} Y_\lambda ; \]
see Subsection \ref{subsec:Bcontainscenter} for the definition of $Y_\lambda$.
Since for $\lambda \in Y^+$, all weights of the simple $\mathfrak{g}^\vee$-module $V(\lambda)$ belong to the same $\Z\Phi^\vee$-coset in $Y$, it follows that $dia(\chi_\lambda) = (-1)^{\ell(t_\lambda)} \cdot \chi_\lambda$, and we obtain
\begin{equation} \label{eq:chilambdanegativeKLbasis}
	\chi_\lambda = (-1)^{\ell(t_\lambda)} \cdot dia(\chi_\lambda) = \sum_{x \in W_\mathrm{ext}} (-1)^{\ell(t_\lambda)+\ell(x)} \cdot a_{x,\lambda} \cdot \underline{\tilde H}_x .
\end{equation}

\subsection{}
\label{subsec:sphericalandantisphericalmodule}

By the quadratic relation $(H_s + v) (H_s - v^{-1}) = 0$, there are two distinct $\mathcal{H}$-module structures on $\mathcal{A} = \Z[v,v^{-1}]$, namely the trivial module $\mathcal{A}_\mathrm{triv} = \mathcal{A}$, where $H_s$ acts by multiplication with $v^{-1}$ for all $s \in S$, and the sign module $\mathcal{A}_\mathrm{sign} = \mathcal{A}$, where $H_s$ acts by multiplication with $-v$ for all $s \in S$.
The \emph{spherical module} and the \emph{anti-spherical module} are the right $\mathcal{H}_\mathrm{ext}$-modules defined via
\[ \mathcal{M} \coloneqq \mathcal{A}_\mathrm{triv} \otimes_{\mathcal{H}_\mathrm{fin}} \mathcal{H}_\mathrm{ext} , \hspace{2cm} \mathcal{N} \coloneqq \mathcal{A}_\mathrm{sign} \otimes_{\mathcal{H}_\mathrm{fin}} \mathcal{H}_\mathrm{ext} , \]
respectively, where $\mathcal{H}_\mathrm{fin}$ denotes the finite Hecke algebra, generated by $H_s$ for $s \in S_\mathrm{fin}$.
Writing $M_w = 1 \otimes H_w \in \mathcal{M}$ and $N_w = 1 \otimes H_w \in \mathcal{N}$ for $w \in W_\mathrm{ext}^+$, the standard bases of $\mathcal{M}$ and $\mathcal{N}$ are given by $\{ M_w \mid w \in W_\mathrm{ext}^+ \}$ and $\{ N_w \mid w \in W_\mathrm{ext}^+ \}$, respectively.
The bar involution on $\mathcal{H}_\mathrm{ext}$ induces a bar involution on $\mathcal{M}$ and on $\mathcal{N}$, and there are Kazhdan--Luszitg bases $\{ \underline{M}_w \mid w \in W_\mathrm{ext}^+ \}$ of $\mathcal{M}$ and $\{ \underline{N}_w \mid w \in W_\mathrm{ext}^+ \}$ of $\mathcal{N}$, analogous to the Kazhdan--Lusztig basis of $\mathcal{H}_\mathrm{ext}$.
For $x \in W_\mathrm{ext}$, we have
\begin{equation} \label{eq:antisphericalKLmultiplication}
	N_e \cdot \underline{H}_x = \begin{cases} \underline{N}_x & \text{if } x \in W_\mathrm{ext}^+ , \\ 0 & \text{otherwise} \end{cases}
\end{equation}
by the proof of Proposition 3.4 in \cite{SoergelKL}.
The spherical Kazhdan--Lusztig polynomials $m_{y,x}$ and the anti-sperical Kazhdan--Lusztig polynomials $n_{y,x}$ are defined via
\[ \underline{M}_x = \sum_{y \in W_\mathrm{ext}^+} m_{y,x} \cdot M_y , \hspace{2cm} \underline{N}_x = \sum_{y \in W_\mathrm{ext}^+} n_{y,x} \cdot N_y \]
for all $x \in W_\mathrm{ext}^+$, and the inverse anti-spherical Kazhdan--Lusztig polynomials $n^{x,y}$ are defined via
\begin{equation} \label{eq:inverseantisphericalKLpolynomials}
	N_x = \sum_{y \in W_\mathrm{ext}^+} (-1)^{\ell(y)+\ell(x)} \cdot n^{x,y} \cdot \underline{N}_y
\end{equation}
for all $x \in W_\mathrm{ext}^+$.

Like the Hecke algebra $\mathcal{H}_\mathrm{ext}$, the spherical module $\mathcal{M}$ and the anti-spherical module $\mathcal{N}$ each admit a second canonical basis (besides the Kazhdan--Lusztig basis), denoted by $\{ \underline{\tilde M}_x \mid x \in W_\mathrm{ext}^+ \} \subseteq \mathcal{M}$ and $\{ \underline{\tilde N}_x \mid x \in W_\mathrm{ext}^+ \} \subseteq \mathcal{N}$ and defined by the requirement that the base change coefficients with the standard basis are in $\Z[v^{-1}]$ (rather than $\Z[v]$); see Theorem 3.5 in \cite{SoergelKL}.
According to the proof of Theorem 3.5 in \cite{SoergelKL}, there is a $\Z$-linear bijection $\vartheta \colon \mathcal{M} \to \mathcal{N}$ such that
\begin{equation} \label{eq:thetadialinear}
	\vartheta(m \cdot h) = \vartheta(m) \cdot dia(h)
\end{equation}
for all $m \in \mathcal{M}$ and $h \in \mathcal{H}$, where $d$ is the bar involution and $i$ and $a$ are as in Subsection \ref{subsec:negativeKL}.
The bijection $\varphi$ satisfies
\begin{equation} \label{eq:thetapropertiesbases}
	\vartheta(M_x) = (-1)^{\ell(x)} \cdot N_x , \qquad \vartheta( \underline{M}_x ) = (-1)^{\ell(x)} \cdot \underline{\tilde N}_x , \qquad \vartheta( \underline{\tilde M}_x ) = (-1)^{\ell(x)} \cdot \underline{N}_x
\end{equation}
for all $x \in W_\mathrm{ext}^+$.
Furthermore, there is an isomorphism of right $\mathcal{H}$-modules $\psi \colon \mathcal{M} \to \underline{H}_{w_0} \cdot \mathcal{H}$ such that
\begin{equation} \label{eq:psipropertiesbases}
	\psi( M_x ) = \underline{H}_{w_0} H_x , \qquad \psi( \underline{M}_x ) = \underline{H}_{w_0x}
\end{equation}
for all $x \in W_\mathrm{ext}^+$, where $w_0 \in W_\mathrm{fin}$ denotes the longest element, see the proof of \cite[Proposition 3.4]{SoergelKL}.

\subsection{}

Now we explain how the Laurent polynomials $a_{w,\lambda} = b_{w,\lambda}$, for $w \in W_\mathrm{aff}^+$ and $\lambda \in Y^+$, can be computed in terms of parabolic Kazhdan--Lusztig polynomials.

\begin{Theorem} \label{thm:centercanonicalbasiscoefficients}
	For $w \in W_\mathrm{ext}^+$ and $\lambda \in Y^+$, we have
	\[ a_{w,\lambda} = b_{w,\lambda} = (-1)^{\ell(t_\lambda)+\ell(w)} \cdot \sum_{y \in W_\mathrm{ext}^+} \overline{ m_{y,t_\lambda} } \cdot n^{y,w} . \]
\end{Theorem}
\begin{proof}
	The equality $a_{w,\lambda} = b_{w,\lambda}$ was explained in Subsection \ref{subsec:chilambdaviacanonicalbasis}.
	Combining the observations from Subsections \ref{subsec:negativeKL} and \ref{subsec:sphericalandantisphericalmodule}, we compute
	{\allowdisplaybreaks
	\begin{alignat*}{2}
	\underline{\tilde N}_{t_\lambda} & = (-1)^{\ell(t_\lambda)} \cdot \vartheta( \underline{M}_{t_\lambda} ) \hspace{2cm} && \text{ by \eqref{eq:thetapropertiesbases}} \\
	& = (-1)^{\ell(t_\lambda)} \cdot \vartheta \psi^{-1}( \underline{H}_{w_0 t_\lambda} ) \hspace{2cm} && \text{ by \eqref{eq:psipropertiesbases}} \\
	& = (-1)^{\ell(t_\lambda)} \cdot \vartheta \psi^{-1}( \underline{H}_{w_0} \cdot \chi_\lambda ) \hspace{2cm} && \text{ by \eqref{eq:chilambdaKLbasis}} \\
	& = \vartheta \psi^{-1}\Big( \sum_{x \in W_\mathrm{ext}} (-1)^{\ell(x)} \cdot a_{x,\lambda} \cdot \underline{H}_{w_0} \cdot \underline{\tilde H}_{x} \Big) \hspace{2cm} && \text{ by \eqref{eq:chilambdanegativeKLbasis}} \\
	& = \vartheta\Big( \sum_{x \in W_\mathrm{ext}} (-1)^{\ell(x)} \cdot a_{x,\lambda} \cdot \psi^{-1}( \underline{H}_{w_0} ) \cdot \underline{\tilde H}_{x} \Big) \\
	& = \vartheta\Big( \sum_{x \in W_\mathrm{ext}} (-1)^{\ell(x)} \cdot a_{x,\lambda} \cdot M_e \cdot \underline{\tilde H}_{x} \Big) \hspace{2cm} && \text{ by \eqref{eq:psipropertiesbases}} \\
	& = \sum_{x \in W_\mathrm{ext}} (-1)^{\ell(x)} \cdot a_{x,\lambda} \cdot N_e \cdot dia( \underline{\tilde H}_{x} ) \hspace{2cm} && \text{ by \eqref{eq:thetadialinear} and \eqref{eq:thetapropertiesbases}} \\
	& = \sum_{x \in W_\mathrm{ext}} a_{x,\lambda} \cdot N_e \cdot \underline{H}_{x} \hspace{2cm} && \text{ by \eqref{eq:negativeKL}} \\
	& = \sum_{x \in W_\mathrm{ext}^+} a_{x,\lambda} \cdot \underline{N}_x \hspace{2cm} && \text{ by \eqref{eq:antisphericalKLmultiplication}} .
\end{alignat*}
	}
	We further have
	\[ \underline{\tilde N}_{t_\lambda} = \sum_{y \in W_\mathrm{ext}^+} (-1)^{\ell(t_\lambda) + \ell(y)} \cdot \overline{ m_{y,t_\lambda} } \cdot N_y = \sum_{x \in W_\mathrm{ext}^+} \Big( \sum_{y \in W_\mathrm{ext}^+} (-1)^{\ell(t_\lambda) + \ell(x)} \cdot \overline{ m_{y,t_\lambda} } \cdot n^{y,x} \Big) \cdot \underline{N}_x , \]
	 by Theorem 3.5 in \cite{SoergelKL} and by \eqref{eq:inverseantisphericalKLpolynomials}, and it follows that
	\[ a_{w,\lambda} = (-1)^{\ell(t_\lambda)+\ell(w)} \cdot \sum_{y \in W_\mathrm{ext}^+} \overline{ m_{y,t_\lambda} } \cdot n^{y,w} \]
	for all $w \in W_\mathrm{ext}^+$, as claimed.
\end{proof}

\begin{Remark}
	For $w \in W_\mathrm{ext}^+$ and $\lambda \in Y^+$, the equality
	\[ a_{w,\lambda} = (-1)^{\ell(t_\lambda)+\ell(w)} \cdot \sum_{y \in W_\mathrm{ext}^+} \overline{ m_{y,t_\lambda} } \cdot n^{y,w} \]
	still holds without assuming the existence of the canonical basis of $\mathcal{B}$.
\end{Remark}

\appendix

\section{The standard basis in rank 2}
\label{app:rank2}

In this appendix, we prove the existence of the standard basis of $\mathcal{B}$ (see Conjecture \ref{conj:standardbasis}) and the validity of Hypothesis \ref{hyp:standardbasisfiltration} for root systems of type $\mathsf{B}_2$ and $\mathsf{G}_2$.
We start with some additional notation pertaining to the affine Weyl group and with some considerations that reduce Conjecture \ref{conj:standardbasis} to a finite computational problem for any given root system.

\subsection{Alcove geometry}

The extended affine Weyl group $W_\mathrm{ext} = W_\mathrm{fin} \ltimes Y$ acts on $Y_\R = Y \otimes_\Z \R$ via
\[ t_\gamma w(y) = w(y) + \gamma \]
for $w \in W_\mathrm{fin}$, $\gamma \in Y$ and $y \in Y_\R$, and the fixed points of a reflection $s_{\alpha,r} = t_{r\alpha^\vee} s_\alpha$ form an affine hyperplane $H_{\alpha,r} = \{ y \in Y_\R \mid (\alpha,y) = r \}$ for $\alpha \in \Phi$ and $r \in \Z$.
The connected components of
\[ Y_\R \setminus \bigcup_{\alpha,r} H_{\alpha,r} \]
are called \emph{alcoves}.
Every alcove $A \subseteq Y_\R$ can be described as
\[ A = \{ y \in Y_\R \mid n_\alpha \leq (\alpha,y) \leq n_\alpha+1 \text{ for all } \alpha \in \Phi^+ \} \]
for uniquely determined integers $n_\alpha = n_\alpha(A)$ with $\alpha \in \Phi^+$, and the alcove
\[ A_\mathrm{fund} = \{ y \in Y_\R \mid 0 \leq (\alpha,y) \leq 1 \text{ for all } \alpha \in \Phi^+ \} \]
is called the \emph{fundamental alcove}.
The set of alcoves is acted upon by $W_\mathrm{ext}$ since the action of $W_\mathrm{ext}$ on $Y_\R$ permutes the reflection hyperplanes $H_{\alpha,r}$.
It is well known that $W_\mathrm{aff}$ acts freely on the set of alcoves, so that the map $x \mapsto x(A_\mathrm{fund})$ gives rise to a bijection between $W_\mathrm{aff}$ and the set of alcoves.
For the action of $W_\mathrm{ext}$ on the set of alcoves, the stabilizer of $A_\mathrm{fund}$ is the subgroup $\Omega \subseteq W_\mathrm{ext}$ defined in Subsection \ref{subsec:WaffWext}:
\[ \Omega = \{ \omega \in W_\mathrm{ext} \mid \omega( A_\mathrm{fund} ) = A_\mathrm{fund} \} \]

\subsection{Dominant and restricted alcoves}

An element $y \in Y_\R$ is called dominant if $(y,\alpha) > 0$ for all $\alpha \in \Phi^+$, and an alcove $A \subseteq Y_\R$ is called \emph{dominant} if it contains a dominant element.
Using the fact that every $W_\mathrm{fin}$-orbit in $Y_\R$ contains a unique dominant element, it is straightforward to see that an element $w \in W_\mathrm{ext}$ belongs to the subset $W_\mathrm{ext}^+ \subseteq W_\mathrm{ext}$ of minimal $W_\mathrm{fin}$-coset representatives if and only if the alcove $w(A_\mathrm{fund})$ is dominant.

An alcove $A \subseteq Y_\R$ is called \emph{restricted} if $n_\alpha(A) = 0$ for all simple roots $\alpha \in \Pi$, or equivalently, if $0<(y,\alpha)<1$ for all $y \in A$ and $\alpha \in \Pi$.
We write $\mathcal{P}$ for the set of restricted alcoves and note that these are precisely the alcoves that are contained in the parallelepiped spanned by the fundamental dominant coweights $\varpi_\alpha^\vee$ with $\alpha \in \Pi$.
Then $\mathcal{P}$ is a set of representatives for the action of the translation subgroup $Y \subseteq W_\mathrm{ext}$ on the set of alcoves.
Writing
\[ \mathcal{P}_\mathrm{ext} = \{ w \in W_\mathrm{ext} \mid w(A_\mathrm{fund}) \subseteq \mathcal{P} \} , \hspace{2cm} \mathcal{P}_\mathrm{aff} = \{ w \in W_\mathrm{aff} \mid w(A_\mathrm{fund}) \subseteq \mathcal{P} \} , \]
we have $\mathcal{P}_\mathrm{ext} = \mathcal{P}_\mathrm{aff} \cdot \Omega$, and every element $x \in W_\mathrm{ext}$ can be written uniquely as a product $x = t_\gamma w$ with $\gamma \in Y$ and $w \in \mathcal{P}_\mathrm{ext}$.
Also note that $\mathcal{P}_\mathrm{ext} \subseteq W_\mathrm{ext}^+$ and $\mathcal{P}_\mathrm{aff} \subseteq W_\mathrm{aff}^+$.

\subsection{Length function and minimal coset representatives}

In order to prove our reduction theorem for Conjecture \ref{conj:standardbasis}, we need to establish some (presumably well-known) auxiliary results on the length function $\ell \colon W_\mathrm{ext} \to \Z_{\geq 0}$ and the set of minimal $W_\mathrm{fin}$-coset representatives $W_\mathrm{ext}^+$.

\begin{Lemma} \label{lem:lengthadditiveminimalcosetrep}
	Let $\gamma \in Y$ and $x \in W_\mathrm{ext}$ such that $t_\gamma x \in W_\mathrm{ext}^+$ and $\ell(t_\gamma x) = \ell(t_\gamma) + \ell(x)$.
	Then we have $\gamma \in Y^+$ and $x \in W_\mathrm{ext}^+$.
\end{Lemma}
\begin{proof}
	For all $y,z \in W_\mathrm{ext}$ with $yz \in W_\mathrm{ext}^+$ and $\ell(yz) = \ell(y) + \ell(z)$, we must have $y \in W_\mathrm{ext}^+$ by the definition of $W_\mathrm{ext}^+$ in Subsection \ref{subsec:standardbasisB0}.
	In particular, we have $t_\gamma \in W_\mathrm{ext}^+$ and therefore $\gamma \in Y^+$.
	Next let us write $x = t_\nu w$ with $\nu \in Y$ and $w \in W_\mathrm{fin}$, and observe that
	\[ t_\gamma x = t_{\gamma + \nu} w = t_\nu w t_{w^{-1}(\gamma)} = x t_{w^{-1}(\gamma)} , \]
	where $\ell(x t_{w^{-1}(\gamma)}) = \ell(t_\gamma x) = \ell(t_\gamma) + \ell(x) = \ell(x) + \ell(t_{w^{-1}(\gamma)})$.
	As $x t_{w^{-1}(\gamma)} = t_\gamma x \in W_\mathrm{ext}^+$, we conclude that $x \in W_\mathrm{ext}^+$, as required.
\end{proof}

We also have the following converse of Lemma \ref{lem:lengthadditiveminimalcosetrep}:

\begin{Lemma} \label{lem:dominantweightandminimalcosetrep}
	Let $\gamma \in Y^+$ and $x \in W_\mathrm{ext}^+$.
	Then $t_\gamma x \in W_\mathrm{ext}^+$ and $\ell(t_\gamma x) = \ell(t_\gamma) + \ell(x)$.
\end{Lemma}
\begin{proof}
	Since $x(A_\mathrm{fund})$ is a dominant alcove and $\gamma \in Y^+$, we have that $x(A_\mathrm{fund}) + \gamma$ is a dominant alcove and hence $t_\gamma x \in W_\mathrm{ext}^+$.
	The claim that $\ell(t_\gamma x) = \ell(t_\gamma) + \ell(x)$ is proven in \cite[Lemma 4.6]{GruberGeneric}.
\end{proof}

\subsection{The reduction theorem}

As in Subsection \ref{subsec:filtrationHaff}, we consider the filtration
\[ \{0\} = \mathcal{H}^{\leq -1} \subseteq \mathcal{H}^{\leq 0} \subseteq \mathcal{H}^{\leq 1} \subseteq \cdots \subseteq \mathcal{H} \]
defined via $\mathcal{H}^{\leq i} = \mathrm{span}_\mathcal{A}\{ H_w \mid w \in W_\mathrm{aff} \text{ with } \ell(w) \leq i \}$ and the analogous filtration
\[ \{0\} = \mathcal{H}_\mathrm{ext}^{\leq -1} \subseteq \mathcal{H}_\mathrm{ext}^{\leq 0} \subseteq \mathcal{H}_\mathrm{ext}^{\leq 1} \subseteq \cdots \subseteq \mathcal{H}_\mathrm{ext} \]
with $\mathcal{H}_\mathrm{ext}^{\leq i} = \mathrm{span}_\mathcal{A}\{ H_w \mid w \in W_\mathrm{ext} \text{ with } \ell(w) \leq i \} = \mathcal{H}^{\leq i} \cdot \Omega$.
Recall that in Conjecture \ref{conj:standardbasis}, we have conjectured the existence of certain elements $B_w \in \mathcal{B}_\mathrm{aff}$ for $w \in W_\mathrm{aff}^+$ which (assuming their existence) form the standard basis $\{ B_w \mid w \in W_\mathrm{aff}^+ \}$ of $\mathcal{B}$, see Lemma \ref{lem:Bwbasis} and Definition \ref{def:standardbasis}.
If we additionally impose the condition $B_w \in \mathcal{H}^{\leq \ell(w)}$ on the elements from Conjecture \ref{conj:standardbasis} (as in Hypothesis \ref{hyp:standardbasisfiltration}), then the following reduction theorem reduces the conjecture to proving the existence of $B_w$ for finitely many $w \in W_\mathrm{aff}^+$.
(Note that in type $\mathsf{A}_n$, the condition $B_w \in \mathcal{H}^{\leq \ell(w)}$ is satisfied for all $w \in W_\mathrm{aff}^+$ by Remark \ref{rem:Bwproperties}.)

\begin{Theorem} \label{thm:reductionstandardbasis}
	Suppose that for all $w \in \mathcal{P}_\mathrm{aff}$, there is an element $B_w \in \mathcal{B}$ such that
	\begin{enumerate}
		\item the expansion
	\[ B_w = \sum_{x \in W_\mathrm{aff}} P_{x,w} \cdot H_x \]
	in terms of the standard basis of $\mathcal{H}$ satisfies $P_{w,w} = 1$ and $P_{x,w} = 0$ for all $w \neq x \in W_\mathrm{aff}^+$;
		\item we have $B_w \in \mathcal{H}^{\leq \ell(w)}$.
	\end{enumerate}
	Then, for all $w \in W_\mathrm{aff}^+$, there is an element $B_w \in \mathcal{B}$ that satisfies (1) and (2).
	In particular, Conjecture \ref{conj:standardbasis} holds and the Hypothesis \ref{hyp:standardbasisfiltration} is satisfied.
\end{Theorem}
\begin{proof}
	First observe that since $\mathcal{P}_\mathrm{ext} = \mathcal{P}_\mathrm{aff} \cdot \Omega$ and $\mathcal{B}_\mathrm{ext} = \mathcal{B}_\mathrm{aff} \cdot \Omega$ (see Lemma \ref{lem:Bomega}), we may assume that for all $w \in \mathcal{P}_\mathrm{ext}$, there is an element $B_w \in \mathcal{B}_\mathrm{ext}$ such that
	\begin{itemize}
		\item[$(1')$] the expansion
		\[ B_w = \sum_{x \in W_\mathrm{ext}} P_{x,w} \cdot H_x \]
		in terms of the standard basis of $\mathcal{H}_\mathrm{ext}$ satisfies $P_{w,w} = 1$ and $P_{x,w} = 0$ for all $w \neq x \in W_\mathrm{ext}^+$;
		\item[$(2')$] we have $B_w \in \mathcal{H}_\mathrm{ext}^{\leq \ell(w)}$.
	\end{itemize}
	For an arbitrary element $w \in W_\mathrm{ext}^+$, we will prove the existence of $B_w \in \mathcal{B}_\mathrm{ext}$ satisfying $(1')$ and $(2')$ by induction on $\ell(w)$, starting from the observation that $B_\omega = \omega$ satisfies these conditions for all $\omega \in \Omega$.
	Now suppose that $\ell(w)>0$ and let $\lambda \in Y^+$ and $u \in \mathcal{P}_\mathrm{ext}$ such that $w = t_\lambda u$.
	Consider the element $B_u \in \mathcal{B}_\mathrm{ext}$ satisfying $(1')$ and $(2')$ given by the assumption and the element $\chi_\lambda \in Z( \mathcal{H}_\mathrm{ext} )$ defined in Subsection \ref{subsec:centerHext}, and observe that $\chi_\lambda \in \mathcal{B}_\mathrm{ext}$ by Corollary \ref{cor:centerinB}.
	We can write
	\[ \chi_\lambda \cdot B_u = \sum_{x \in W_\mathrm{ext}} P'_{x,w} \cdot H_x \]
	for certain Laurent polynomials $P'_{x,w} \in \mathcal{A}$.
	For all $y \in W_\mathrm{ext}^+$ with $\ell(y)<\ell(w)$, we may assume by induction that there exists an element $B_y = \sum_{x \in W_\mathrm{ext}} P_{y,x} \cdot H_y \in \mathcal{B}_\mathrm{ext}$ satisfying the conditions $(1')$ and $(2')$.
	We claim that the element
	\begin{equation} \label{eq:constructionBw}
		B_w \coloneqq \chi_\lambda \cdot B_u - \sum_{ \substack{ y \in W_\mathrm{ext}^+ \\ \ell(y) < \ell(w) } } P'_{y,w} \cdot B_y \in \mathcal{B}_\mathrm{ext}
	\end{equation}
	also satisfies the conditions $(1')$ and $(2')$.
	
	Indeed, we have $\chi_\lambda \in \mathcal{H}_\mathrm{ext}^{\leq \ell(t_\lambda)}$ by construction, whereas $B_u \in \mathcal{H}_\mathrm{ext}^{\leq \ell(u)}$ and $\mathcal{B}_y \in \mathcal{H}_\mathrm{ext}^{\leq \ell(y)}$ for all $y \in W_\mathrm{ext}^+$ with $\ell(y) < \ell(w)$ by $(2')$.
	As $\ell(w) = \ell(t_\lambda) + \ell(u)$ by Lemma \ref{lem:dominantweightandminimalcosetrep}, we conclude that $B_w \in \smash{\mathcal{H}_\mathrm{ext}^{\leq \ell(w)}}$, and so $B_w$ satisfies $(2')$.
	Now consider the expansion
	\[ B_w = \sum_{x \in W_\mathrm{aff}} P_{x,w} \cdot H_x \]
	in terms of the standard basis of $\mathcal{H}_\mathrm{ext}$.
	By the definition of $B_w$ in \eqref{eq:constructionBw}, we have $P_{y,w} = 0$ for all $y \in W_\mathrm{ext}^+$ with $\ell(y) < \ell(w)$ and $P_{x,w} = P'_{x,w}$ for all $x \in W_\mathrm{ext}$ with $\ell(x) = \ell(w)$, and so it remains to show that $P'_{x,w} = \delta_{x,w}$ for $x \in W_\mathrm{ext}^+$ with $\ell(x) = \ell(w)$.
	To that end, observe that for $y,z \in W_\mathrm{ext}$, we have $H_y \cdot H_z = H_{yz}$ if $\ell(yz) = \ell(y) + \ell(z)$ and $H_y \cdot H_z \in \smash{\mathcal{H}_\mathrm{ext}^{\leq \ell(y)+\ell(z)-1}}$ if $\ell(yz) < \ell(y) + \ell(z)$.
	It is straightforward to see from the definition in Subsection \ref{subsec:centerHext} that we can write
	\[ \chi_\lambda = \sum_{\lambda' \in W_\mathrm{fin}(\lambda)} H_{t_{\lambda'}} + R \]
	with $R \in \mathcal{H}_\mathrm{ext}^{\leq \ell(t_\lambda)-1}$, and so for $x \in W_\mathrm{ext}^+$ with $\ell(x) = \ell(w) = \ell(t_\lambda) + \ell(u)$ and $P'_{x,w} \neq 0$, there must be elements $\lambda' \in W_\mathrm{fin}(\lambda)$ and $u' \in W_\mathrm{ext}$ such that $P_{u',u} \neq 0$ and $x = t_{\lambda'} u'$ with $\ell(x) = \ell(t_{\lambda'})+\ell(u')$.
	By Lemma \ref{lem:lengthadditiveminimalcosetrep}, this implies that $\lambda' \in Y^+$ and $u' \in W_\mathrm{ext}^+$, and we conclude that $\lambda' = \lambda$ and $u' = u$ because $B_u$ satisfies $(1')$.
	Thus $x = t_\lambda u = w$, and we further obtain $P'_{w,w} = 1$ because $P_{u,u} = 1$.
	In conclusion, we have $P'_{x,w} = \delta_{x,w}$ for $x \in W_\mathrm{ext}^+$ with $\ell(x) = \ell(w)$, and it follows that $B_w$ satisfies $(1')$.
	
	Finally, it remains to note that for $w \in W_\mathrm{aff}^+$, the element $B_w \in \mathcal{B}_\mathrm{ext}$ constructed above actually belongs to $\mathcal{B}_\mathrm{aff}$ by Lemmas \ref{lem:Bomega} and \ref{lem:Bcoefficientnonzero}, and so $B_w$ satisfies (1) and (2).
\end{proof}

\subsection{Type \texorpdfstring{$\mathsf{B}_2$}{B2}}
\label{subsec:standardbasisB2}

Let $\Phi$ be the root system of type $\mathsf{B}_2$, with simple roots $\Pi = \{ \alpha , \beta \}$ such that $\alpha$ is short.
The positive roots are given by $\Phi^+ = \{ \alpha , \beta , \alpha+\beta , 2\alpha+\beta \}$, and the affine Weyl group $W_\mathrm{aff}$ is generated by the simple reflections $S_\mathrm{aff} = \{ s , t , u \}$ with $u = s_\alpha$, $t = s_\beta$ and $s = s_0$ the affine simple root.
The braid relations in $W_\mathrm{aff}$ are as follows:
\[ st = ts , \hspace{2cm} susu=usus , \hspace{2cm} tutu=utut \]
We have $\mathcal{P}_\mathrm{aff} = \{ e , s , su , sut \}$, and the corresponding set $\mathcal{P}$ of restricted alcoves is displayed in Figure \ref{fig:alcovesB2}.
\begin{figure}[htbp]
	\begin{tikzpicture}[scale=1.5]
	\draw[fill,gray!30] (0,0) -- (0,2) -- (1,3) -- (1,1) -- cycle;
	\draw[thick] (0,0) -- (0,3);
	\draw[thick] (0,0) -- (3,3);
	\draw[thick] (0,1) -- (1,1);
	\draw[thick] (1,1) -- (0,2);
	\draw[thick] (1,1) -- (1,3);
	\draw[thick] (0,2) -- (2,2);
	\draw[thick] (0,2) -- (1,3);
	\draw[thick] (2,2) -- (1,3);
	\draw[thick] (2,2) -- (2,3);
	\draw[thick] (0,3) -- (3,3);
	
	\node at (.325,.675) {\footnotesize $e$};
	\node at (.325,1.325) {\footnotesize $s$};
	\node at (.675,1.675) {\footnotesize $su$};
	\node at (.675,2.325) {\footnotesize $sut$};
	\node at (.325,2.675) {\footnotesize $sutu$};
	\node at (1.325,1.675) {\footnotesize $sus$};
	\node at (1.325,2.325) {\footnotesize $suts$};
	\end{tikzpicture}
	\caption{Some dominant alcoves for $\Phi$ of type $\mathsf{B}_2$, labeled by the corresponding elements of $W_\mathrm{aff}$.
	The set $\mathcal{P}$ of restricted alcoves is shaded in gray.}
	\label{fig:alcovesB2}
\end{figure}
The highest root $\alpha_\mathrm{h} = 2\alpha+\beta$ has coroot $\alpha_\mathrm{h}^\vee = \alpha^\vee + \beta^\vee$, and so Lemma \ref{lem:Bs0} implies that
\[ B_s \coloneqq H_s + H_t + H_u \in \mathcal{B}_\mathrm{aff} . \]
Next, we compute that
\[ B_s \cdot B_s = H_{su} + 2 H_{st} + H_{us} + H_{tu} + H_{ut} + (v^{-1}-v) \cdot ( H_s + H_t + H_u ) + 3 H_e , \]
and it follows that
\[ B_{su} \coloneqq B_s \cdot B_s - (v^{-1}-v) \cdot B_s - 3 H_e = H_{su} + 2 H_{st} + H_{us} + H_{tu} + H_{ut} \in \mathcal{B}_\mathrm{aff} . \]
In order to define $B_{sut}$, we first construct the element $B_{sus}$ using the element $\chi_{\varpi_\beta^\vee} \in Z(\mathcal{H}_\mathrm{ext})$ defined in Subsection \ref{subsec:centerHext}.
Note that $\Omega = \{ 1 , \omega \}$ is a cyclic group of order $2$ and that $t_{\varpi_\beta^\vee} = sus \omega$.
We compute
\begin{multline*}
	\chi_{\varpi_\beta^\vee} \cdot \omega = H_{sus} + H_{tut} + H_{uts} + H_{tsu} + (v-v^{-1}) \cdot ( H_{su} + H_{st} + H_{us} + H_{tu} + H_{ut} ) \\ + (v^2 - 2 + v^{-2}) \cdot (H_s + H_t + H_u) + (v^3 - 2v + 2v^{-1} - v^{-3}) \cdot H_e ,
\end{multline*}
and since $\chi_{\varpi_\beta^\vee} \in \mathcal{B}_{\mathrm{ext}}$ by Corollary \ref{cor:centerinB} and $\omega \in \mathcal{B}_\mathrm{ext}$ by Lemma \ref{lem:Bomega}, we conclude that
\begin{multline*}
	B_{sus} \coloneqq \chi_{\varpi_\beta^\vee} \cdot \omega - (v-v^{-1}) B_{su} - (v^2-2+v^{-2}) \cdot B_s - (v^3-2v+2v^{-1}-v^{-3}) \cdot H_e \\
	= H_{sus} + H_{tut} + H_{uts} + H_{tsu} - (v-v^{-1}) \cdot H_{st} \in \mathcal{B}_\mathrm{aff} .
\end{multline*}
Finally, we compute
\begin{multline*}
	 B_{su} \cdot B_s = H_{sut} + H_{sus} + H_{tus} + 2 H_{uts} + H_{usu} + 2 H_{tsu} + H_{utu} + H_{tut} \\ + (v^{-1}-v) \cdot ( H_{us} + 4 H_{st} + H_{su} + H_{tu} + H_{ut} ) + 3 H_s + 3 H_t + 2 H_u ,
\end{multline*}
and it follows that
\begin{align*}
	B_{sut} \coloneqq & B_{su} \cdot B_s - B_{sus} - (v^{-1}-v) \cdot B_{su} - 3 B_s \\
	= & H_{sut} + H_{tus} + H_{uts} + H_{usu} + H_{tsu} + H_{utu} + (v^{-1}-v) \cdot H_{st} - H_u \in \mathcal{B}_\mathrm{aff} .
\end{align*}
The elements $B_s,B_{su},B_{sut} \in \mathcal{B}_\mathrm{aff}$ clearly satisfy the conditions (1) and (2) from Theorem \ref{thm:reductionstandardbasis}, and we conclude that Conjecture \ref{conj:standardbasis} holds and that Hypothesis \ref{hyp:standardbasisfiltration} is satisfied for the root system of type $\mathsf{B}_2$.

\subsection{Type \texorpdfstring{$\mathsf{G}_2$}{G2}}

Let $\Phi$ be the root system of type $\mathsf{G}_2$, with simple roots $\Pi = \{ \alpha , \beta \}$ such that $\alpha$ is short.
The positive roots are given by $\Phi^+ = \{ \alpha , \beta , \alpha+\beta , 2\alpha+\beta , 3\alpha+\beta , 3\alpha + 2\beta \}$, and the affine Weyl group $W_\mathrm{aff}$ is generated by the simple reflections $S_\mathrm{aff} = \{ s , t , u \}$ with $u = s_\alpha$, $t = s_\beta$ and $s = s_0$ the affine simple reflection.
The braid relations in $W_\mathrm{aff}$ are as follows:
\[ sts = tst , \hspace{2cm} su=us , \hspace{2cm} tututu=ututut \]
The set $\mathcal{P}$ of restricted alcoves is displayed in Figure \ref{fig:alcovesG2}, and the corresponding subset of $W_\mathrm{aff}$ is
\[ \mathcal{P}_\mathrm{aff} = \{ e , s , st , stu ,stut , stutu , stuts , stutsu , stutsut , stutsutu , stutsutut , stutsututs \} . \]
\begin{figure}[htbp]%
\begin{tikzpicture}[scale=2.5]
	\pgftransformcm{cos(60)}{sin(60)}{0}{sqrt(3)/2}{\pgfpoint{0cm}{0cm}}
	\draw[fill,gray!30] (0,0) -- (3,0) -- (3,2) -- (0,2) -- cycle;
	
	\draw[thick] (0,0) -- (5,0);
	\draw[thick] (0,0) -- (0,5);
	\draw[thick] (5,0) -- (0,5);
	
	\draw[thick] (1,0) -- (0,2);
	\draw[thick] (2,0) -- (0,4);
	\draw[thick] (3,0) -- (1,4);
	\draw[thick] (4,0) -- (3,2);
	
	\draw[thick] (1,0) -- (0,1);
	\draw[thick] (2,0) -- (0,2);
	\draw[thick] (3,0) -- (0,3);
	\draw[thick] (4,0) -- (0,4);
	
	\draw[thick] (0,2) -- (1.5,0);
	\draw[thick] (0,2) -- (3,0);
	
	\draw[thick] (0,4) -- (3,0);
	\draw[thick] (0,4) -- (3,2);
	
	\draw[thick] (0,2) -- (3,2);
	\draw[thick] (0,4) -- (1,4);
	
	\draw[thick] (3,0) -- (3,2);
	\draw[thick] (3,2) -- (4.5,0);
	
	\node at (.325,.325) {\tiny $e$};
	\node at (.325,1) {\tiny $s$};
	\node at (.9,.6) {\tiny $st$};
	\node at (1.3,.5) {\tiny $stu$};
	\node at (1.3,.9) {\tiny $stut$};
	\node[rotate=30] at (2.1,.4) {\tiny $stuts$};
	\node[rotate=30] at (.9,1.6) {\tiny $stutu$};
	\node at (1.75,1.1) {\tiny $stutsu$};
	\node at (1.75,1.4) {\tiny $stutsut$};
	\node[rotate=-30] at (2.1,1.4) {\tiny $stutsutu$};
	\node[rotate=-60] at (2.65,1) {\tiny $stutsutut$};
	\node[rotate=60] at (2.65,1.7) {\tiny $stutsututs$};
	
	\node[rotate=60] at (.35,2.3) {\tiny $stutut$};
	\node[rotate=-60] at (.35,3) {\tiny $stututs$};
	\node[rotate=-30] at (.9,2.6) {\tiny $stututst$};
	\node at (1.1,2.8) {\tiny $stututstu$};
	\node[rotate=60] at (3.4,.25) {\tiny $stutsututu$};
\end{tikzpicture}
	\caption{Some dominant alcoves for $\Phi$ of type $\mathrm{G}_2$, labeled by the corresponding elements of $W_\mathrm{aff}$.
	The set $\mathcal{P}$ of restricted alcoves is shaded in gray.}
	\label{fig:alcovesG2}
\end{figure}%
In order to prove the existence of elements $B_w \in \mathcal{B}_\mathrm{aff}$, for $w \in \mathcal{P}_\mathrm{aff}$, satisfying the conditions (1) and (2) in Theorem \ref{thm:reductionstandardbasis}, we proceed as in Subsection \ref{subsec:standardbasisB2}.
\begin{itemize}
	\item Since $\alpha_\mathrm{h}^\vee = \alpha^\vee + 2 \beta^\vee$, we have
	\[ B_s \coloneqq H_s + 2 H_t + H_u \in \mathcal{B}_\mathrm{aff} \]
	by Lemma \ref{lem:Bs0}.
	\item One computes that
	\[ B_s \cdot B_s - (v^{-1}-v) \cdot B_s - 6 H_e = 2 \cdot \big( H_{st} + H_{ts} + H_{su} + H_{tu} + H_{ut} + (v^{-1}-v) \cdot H_t \big) , \]
	and it follows that
	\[ B_{st} \coloneqq H_{st} + H_{ts} + H_{su} + H_{tu} + H_{ut} + (v^{-1}-v) \cdot H_t \in \mathcal{B}_\mathrm{aff} . \]
	Similarly, we have
	\begin{align*}
		B_{stu} & \coloneqq B_s B_{st} + 2 \cdot (v-v^{-1}) B_{st} - 3 B_3 + 2 \cdot (v-v^{-1}) \cdot B_e \\
		& = 2 H_{ t u t } + 2 H_{ u s t } + 3 H_{ t s t } + 1 H_{ u t u } + 
  1 H_{ s t u } + 2 H_{ t u s } + 1 H_{ u t s } -4 H_{ t }
	\end{align*}
	and
	\begin{align*}
		B_s B_{stu} & + (v-v^{-1}) \cdot B_{stu} - 4 B_{st} + 8 B_e  \\
		 & = 2 \cdot \big( H_{ s t u t } + H_{ u t u t } +  2 H_{ t u s t } + 2 H_{ u t s t } + H_{ t u t u } + H_{ u s t u } + 2 H_{ t s t u } + H_{ u t u s } + H_{ t u t s } \\
		 & \hspace{1cm} - (v-v^{-1}) \cdot H_{ t u t } - (v-v^{-1}) \cdot H_{ u s t } - (3v-3v^{-1}) \cdot H_{ t s t } - (v-v^{-1}) \cdot H_{ t u s } - H_{ u t } \big) ,
	\end{align*}
	and it follows that
	\begin{align*}
		B_{stut} & \coloneqq   H_{ s t u t } + H_{ u t u t } +  2 H_{ t u s t } + 2 H_{ u t s t } + H_{ t u t u } + H_{ u s t u } + 2 H_{ t s t u } + H_{ u t u s } + H_{ t u t s } \\
		 & \hspace{1cm} - (v-v^{-1}) \cdot H_{ t u t } - (v-v^{-1}) \cdot H_{ u s t } - (3v-3v^{-1}) \cdot H_{ t s t } - (v-v^{-1}) \cdot H_{ t u s } - H_{ u t } \in \mathcal{B}_\mathrm{aff} .
	\end{align*}
	\item Next, one checks by brute-force computation that the following element of $\mathcal{H}$ belongs to $\mathcal{B}_\mathrm{aff}$:
	\begin{align} \label{eq:G2Bstuts}
	\begin{split}
	B_{stuts} & \coloneqq H_{ s t u t s } + H_{ t s t u t } + H_{ t u s t u } +  H_{ t u t s t } + H_{ t u t u s } + H_{ u t u t u } + H_{ u t u s t } + H_{ u t s t u } + H_{ u s t u t } \\
	& \hspace{2cm} + (-v+v^{-1}) \cdot H_{ t u s t } + (-v+v^{-1}) \cdot H_{ u t s t } + (-v+v^{-1}) \cdot H_{ t s t u } \\
	& \hspace{2cm} -2 H_{ t u t } -2 H_{ u s t } + (v^2-5+v^{-2}) \cdot H_{ t s t } - H_{ u t u } -2 H_{ t u s } \\
	& \hspace{2cm} + (v-v^{-1}) \cdot H_{ u t } + (v-v^{-1}) \cdot H_{ t u } + (-v^2+7-v^{-2}) \cdot H_{ t }
	\end{split}
	\end{align}
	More details of the computation are provided in Subsection \ref{subsec:detailsG2Bstuts}.
	We do not know any conceptual way of constructing $B_{stuts}$, e.g.\ from a central element of $\mathcal{H}$ as in Subsection \ref{subsec:standardbasisB2}, since the filtration piece $\mathcal{H}^{\leq 5}$ does not contain any central elements besides multiples of $1$.
\end{itemize}
Having found $B_{stuts}$, the remaining elements $B_x$ for $x \in \mathcal{P}_\mathrm{aff}$ can easily be computed, following the same strategy as above.
We omit the details of the computations and only outline the main steps.
\begin{itemize}
	\item $B_{stutu}$ is as a linear combination of $B_{stut} B_s$ and the elements $B_x$ for $x \in \{ e , s , st , stu , stut , stuts \}$.
	\item $B_{stutsu}$ is a linear combination of $B_{stuts} B_s$ and $B_x$ for $x \in W_\mathrm{aff}^+$ with $\ell(x) \leq 5$.
	\item $B_{stutut}$ is as a linear combination of $B_{stutu} B_s$ and $B_x$ for $x \in W_\mathrm{aff}^+$ with $\ell(x) \leq 5$ or $x = stutsu$.
	\item $B_{stututs}$ can be constructed from $B_{stutut} B_s$, $B_{stutsut}$ can be constructed from $B_{stutsu} B_s$, and $B_{stututst}$ can be constructed from $B_{stututs} B_s$.
	\item $B_{stutsutu}$ is a linear combination of $B_{stutsut} B_s$ and $B_x$ for $x \in W_\mathrm{aff}^+$ with $x \leq stututst$.
	\item $B_{stututstu}$ can be constructed from $B_{stututst} \cdot B_s$.
	\item $B_{stutsutut}$ is a linear combination of $B_{stutsutu} B_s$ and $B_x$ for $x \in W_\mathrm{aff}^+$ with $x \leq stututstu$.
	\item $B_{stutsututu}$ is a linear combination of the central element $\chi_{\varpi_{\alpha}^\vee}$ with the elements $B_x$ for $x \in W_\mathrm{aff}^+$ with $\ell(x) \leq 9$.
	(Note that $t_{\varpi_{\alpha}^\vee} = stutsututu$.)
	\item $B_{stutsututs}$ is a linear combination of $B_{stutsutut} B_s$ and the elements $B_x$ for $x \in W_\mathrm{aff}^+$ with $\ell(x) \leq 9$ or $x = stutsututu$.
\end{itemize}
In summary, we have constructed for all $w \in \mathcal{P}_\mathrm{aff}$ an element $B_w \in \mathcal{B}_\mathrm{aff}$ that satisfies the conditions (1) and (2) from Theorem \ref{thm:reductionstandardbasis}.
We conclude that Conjecture \ref{conj:standardbasis} holds and Hypothesis \ref{hyp:standardbasisfiltration} is satisfied for the root system of type $\mathsf{G}_2$.
\clearpage

\begin{landscape}

\subsection{Details on the computation of \texorpdfstring{$B_{stuts}$}{Bstuts}}
\label{subsec:detailsG2Bstuts}

On the following pages, we compute $\varphi(H_x X_{\varpi_\alpha})$ and $\varphi(H_x X_{\varpi_\beta})$ for certain elements $x \in W_\mathrm{aff}$ with $\ell(x) \leq 5$.
Since the element $B_{stuts} \in \mathcal{H}$ from \eqref{eq:G2Bstuts} is a linear combination of the standard basis elements $H_x$ for these $x \in W_\mathrm{aff}$, one can directly check that $\varphi(B_{stuts} X_{\varpi_\alpha}) = 0$ and $\varphi(B_{stuts} X_{\varpi_\beta}) = 0$, and therefore $B_{stuts} \in \mathcal{B}_\mathrm{aff}$ by Lemma \ref{lem:BgeneratingsetX}.
\scriptsize

\begingroup
\scriptsize
\allowdisplaybreaks

\begin{align*}
\varphi( H_{ s t u t s } X_{\varpi_\alpha} ) & = H_{ s t u t s } + (v-v^{-1}) \cdot H_{ s t u t } + (v-v^{-1}) \cdot H_{ t s t u } + (v-v^{-1}) \cdot H_{ u t s t } + (v-v^{-1}) \cdot H_{ t u t s } \\
& \qquad + (v^2-2+v^{-2}) \cdot H_{ s t u } + (v^2-2+v^{-2}) \cdot H_{ u s t } + (v^2-2+v^{-2}) \cdot H_{ t u t } + (v^2-2+v^{-2}) \cdot H_{ t s t } + (v^2-2+v^{-2}) \cdot H_{ t u s } + (v^2-2+v^{-2}) \cdot H_{ u t s } \\
& \qquad + (v^3-3v+3v^{-1}-v^{-3}) \cdot H_{ s t } + (v^3-3v+3v^{-1}-v^{-3}) \cdot H_{ u s } + (v^3-3v+3v^{-1}-v^{-3}) \cdot H_{ t u } + (v^3-3v+3v^{-1}-v^{-3}) \cdot H_{ u t } + (v^3-3v+3v^{-1}-v^{-3}) \cdot H_{ t s } \\
& \qquad + (v^4-3v^2+4-3v^{-2}+v^{-4}) \cdot H_{ s } + (v^4-4v^2+6-4v^{-2}+v^{-4}) \cdot H_{ t } + (v^4-3v^2+4-3v^{-2}+v^{-4}) \cdot H_{ u } + (v^5-3v^3+6v-6v^{-1}+3v^{-3}-v^{-5}) \cdot H_{ e } \\
\varphi( H_{ t s t u t } X_{\varpi_\alpha} ) & = H_{ t s t u t } + (-v+v^{-1}) \cdot H_{ t u s t } + (v-v^{-1}) \cdot H_{ s t u t } \\
& \qquad + (v^2-2+v^{-2}) \cdot H_{ t s t } + (-v^2+2-v^{-2}) \cdot H_{ t u t } + (v^3-3v+3v^{-1}-v^{-3}) \cdot H_{ s t } + (-v+v^{-1}) \cdot H_{ t s } \\
& \qquad + (-v^2+2-v^{-2}) \cdot H_{ t } + (-v^2+2-v^{-2}) \cdot H_{ s } \\
\varphi( H_{ t u s t u } X_{\varpi_\alpha} ) & =  H_{ t u s t u } + (-v+v^{-1}) \cdot H_{ t u s t } + (-2v+2v^{-1}) \cdot H_{ t s t u } + (-v+v^{-1}) \cdot H_{ u s t u } \\
& \qquad + (-v^2+2-v^{-2}) \cdot H_{ t u t } + (v^2-2+v^{-2}) \cdot H_{ t s t } + (-v^2+2-v^{-2}) \cdot H_{ s t u } \\
& \qquad + (v^3-3v+3v^{-1}-v^{-3}) \cdot H_{ s t } + (-v+v^{-1}) \cdot H_{ t s } + (-v^2+2-v^{-2}) \cdot H_{ t } + (-v^2+2-v^{-2}) \cdot H_{ s } \\
\varphi( H_{ t u t s t } X_{\varpi_\alpha} ) & = H_{ t u t s t } + (v-v^{-1}) \cdot H_{ t u s t } + (v-v^{-1}) \cdot H_{ u t s t } \\
& \qquad + (v^2-2+v^{-2}) \cdot H_{ t s t } + (v^2-2+v^{-2}) \cdot H_{ u s t } + (v-v^{-1}) \cdot H_{ t u } + (v^3-v+v^{-1}-v^{-3}) \cdot H_{ s t } \\
& \qquad + (2v^2-4+2v^{-2}) \cdot H_{ t } + (v^2-2+v^{-2}) \cdot H_{ u } + (v^3-3v+3v^{-1}-v^{-3}) \cdot H_{ e } \\
\varphi( H_{ t u t u s } X_{\varpi_\alpha} ) & = H_{ t u t u s } + (v-v^{-1}) \cdot H_{ t u t u } + (-v+v^{-1}) \cdot H_{ t u t s } + (v-v^{-1}) \cdot H_{ u t u s } \\
& \qquad + (-v^2+2-v^{-2}) \cdot H_{ t u t } + (v^2-2+v^{-2}) \cdot H_{ u t u } + (-v^2+2-v^{-2}) \cdot H_{ t u s } + (-v^2+2-v^{-2}) \cdot H_{ u t s } \\
& \qquad + (-v^3+3v-3v^{-1}+v^{-3}) \cdot H_{ t u } + (-v^3+3v-3v^{-1}+v^{-3}) \cdot H_{ u t } + (-v^3+3v-3v^{-1}+v^{-3}) \cdot H_{ t s } + (-v^3+4v-4v^{-1}+v^{-3}) \cdot H_{ u s } \\
& \qquad + (-v^4+4v^2-6+4v^{-2}-v^{-4}) \cdot H_{ t } + (-v^4+5v^2-8+5v^{-2}-v^{-4}) \cdot H_{ u } + (-v^4+2v^2-2+2v^{-2}-v^{-4}) \cdot H_{ s } \\
& \qquad + (-v^5+3v^3-4v+4v^{-1}-3v^{-3}+v^{-5}) \cdot H_{ e } \\
\varphi( H_{ u t u t u } X_{\varpi_\alpha} ) & = H_{ u t u t u } + (-v+v^{-1}) \cdot H_{ u t u t } + (-v+v^{-1}) \cdot H_{ t u t u } \\
& \qquad + (2v^2-4+2v^{-2}) \cdot H_{ t u t } + (-v^3+2v-2v^{-1}+v^{-3}) \cdot H_{ u t } + (-v+v^{-1}) \cdot H_{ t u } \\
& \qquad + (v^4-2v^2+2-2v^{-2}+v^{-4}) \cdot H_{ t } + (2v^2-4+2v^{-2}) \cdot H_{ u } + (-v^3+v-v^{-1}+v^{-3}) \cdot H_{ e } \\
\varphi( H_{ u t u s t } X_{\varpi_\alpha} ) & = H_{ u t u s t } + (v-v^{-1}) \cdot H_{ u t u t } + (-v+v^{-1}) \cdot H_{ u t s t } + (v-v^{-1}) \cdot H_{ t u s t } \\
& \qquad + (v^2-2+v^{-2}) \cdot H_{ t u t } + (-v^2+2-v^{-2}) \cdot H_{ u s t } + (-2v^2+4-2v^{-2}) \cdot H_{ t s t } \\
& \qquad + (-2v^3+6v-6v^{-1}+2v^{-3}) \cdot H_{ s t } + (-v^2+2-v^{-2}) \cdot H_{ u } + (-2v^3+6v-6v^{-1}+2v^{-3}) \cdot H_{ e } \\
\varphi( H_{ u t s t u } X_{\varpi_\alpha} ) & = H_{ u t s t u } + (-v+v^{-1}) \cdot H_{ u t s t } + (-v+v^{-1}) \cdot H_{ u t u s } + (v-v^{-1}) \cdot H_{ u s t u } + (v-v^{-1}) \cdot H_{ t s t u } \\
& \qquad + (-v^2+2-v^{-2}) \cdot H_{ u s t } + (-2v^2+4-2v^{-2}) \cdot H_{ t s t } + (-v^2+2-v^{-2}) \cdot H_{ u t u } + (-v^2+2-v^{-2}) \cdot H_{ t u s } + (v^2-2+v^{-2}) \cdot H_{ s t u } \\
& \qquad + (-2v^3+6v-6v^{-1}+2v^{-3}) \cdot H_{ s t } + (-v^3+3v-3v^{-1}+v^{-3}) \cdot H_{ t u } + (-v^2+2-v^{-2}) \cdot H_{ u } + (-2v^3+6v-6v^{-1}+2v^{-3}) \cdot H_{ e } \\
\varphi( H_{ u s t u t } X_{\varpi_\alpha} ) & = H_{ u s t u t } + (-2v+2v^{-1}) \cdot H_{ s t u t } + (v^2-2+v^{-2}) \cdot H_{ u s t } \\
& \qquad + (v^3-3v+3v^{-1}-v^{-3}) \cdot H_{ u t } + (-v^3+2v-2v^{-1}+v^{-3}) \cdot H_{ s t } + (-v+v^{-1}) \cdot H_{ u s } \\
& \qquad + (-v^4+3v^2-4+3v^{-2}-v^{-4}) \cdot H_{ t } + (-v^2+2-v^{-2}) \cdot H_{ u } + (v^2-2+v^{-2}) \cdot H_{ s } + (v^3-3v+3v^{-1}-v^{-3}) \cdot H_{ e } \\
\varphi( H_{ t u s t } X_{\varpi_\alpha} ) & =  H_{ t u s t } + (v-v^{-1}) \cdot H_{ t u t } + (-v+v^{-1}) \cdot H_{ t s t } \\
& \qquad + (-v^2+2-v^{-2}) \cdot H_{ s t } + (-v^2+2-v^{-2}) \cdot H_{ e } \\
\varphi( H_{ u t s t } X_{\varpi_\alpha} ) & =  H_{ u t s t } + (v-v^{-1}) \cdot H_{ u s t } + (2v-2v^{-1}) \cdot H_{ t s t } \\
& \qquad + (2v^2-4+2v^{-2}) \cdot H_{ s t } + (v-v^{-1}) \cdot H_{ u } + (2v^2-4+2v^{-2}) \cdot H_{ e } \\
\varphi( H_{ t s t u } X_{\varpi_\alpha} ) & = H_{ t s t u } + (-v+v^{-1}) \cdot H_{ t s t } + (-v+v^{-1}) \cdot H_{ t u s } + (v-v^{-1}) \cdot H_{ s t u } \\
& \qquad + (-v^2+2-v^{-2}) \cdot H_{ s t } + (-v^2+2-v^{-2}) \cdot H_{ t u } + (-v^2+2-v^{-2}) \cdot H_{ e } \\
\varphi( H_{ u s t } X_{\varpi_\alpha} ) & =  H_{ u s t } + (v-v^{-1}) \cdot H_{ u t } + (-v+v^{-1}) \cdot H_{ s t } + (-v^2+2-v^{-2}) \cdot H_{ t } \\
\varphi( H_{ u t u } X_{\varpi_\alpha} ) & =  H_{ u t u } + (-v+v^{-1}) \cdot H_{ u t } + (-2v+2v^{-1}) \cdot H_{ t u } + (v^2-2+v^{-2}) \cdot H_{ t } + (-v+v^{-1}) \cdot H_{ e } \\
\varphi( H_{ t u t } X_{\varpi_\alpha} ) & =  H_{ t u t } + (-v+v^{-1}) \cdot H_{ u t } + (v^2-2+v^{-2}) \cdot H_{ t } + (-v+v^{-1}) \cdot H_{ e } \\
\varphi( H_{ t u s } X_{\varpi_\alpha} ) & =  H_{ t u s } + (v-v^{-1}) \cdot H_{ t u } + (-v+v^{-1}) \cdot H_{ t s } + (-v^2+2-v^{-2}) \cdot H_{ t } + (-v^2+2-v^{-2}) \cdot H_{ s } + (-v^3+3v-3v^{-1}+v^{-3}) \cdot H_{ e } \\
\varphi( H_{ t s t } X_{\varpi_\alpha} ) & =  H_{ t s t } + (v-v^{-1}) \cdot H_{ s t } + (v-v^{-1}) \cdot H_{ e } \\
\varphi( H_{ u t } X_{\varpi_\alpha} ) & =  H_{ u t } + (-v+v^{-1}) \cdot H_{ t } \\
\varphi( H_{ t u } X_{\varpi_\alpha} ) & =  H_{ t u } + (-v+v^{-1}) \cdot H_{ t } + (-v+v^{-1}) \cdot H_{ u } \\
\varphi( H_{ t } X_{\varpi_\alpha} ) & =  H_{ t } \\
\end{align*}

\begin{align*}
\varphi( H_{ s t u t s } X_{\varpi_\beta} ) & = H_{ s t u t s } + (2v-2v^{-1}) \cdot H_{ s t u t } + (v-v^{-1}) \cdot H_{ t s t u } + (2v-2v^{-1}) \cdot H_{ u t s t } + (v-v^{-1}) \cdot H_{ t u t s } \\
& \qquad + (v^2-2+v^{-2}) \cdot H_{ s t u } + (2v^2-4+2v^{-2}) \cdot H_{ u s t } + (v^2-2+v^{-2}) \cdot H_{ t u t } + (3v^2-6+3v^{-2}) \cdot H_{ t s t } + (v^2-2+v^{-2}) \cdot H_{ t u s } + (v^2-2+v^{-2}) \cdot H_{ u t s } \\
& \qquad + (3v^3-9v+9v^{-1}-3v^{-3}) \cdot H_{ s t } + (v^3-3v+3v^{-1}-v^{-3}) \cdot H_{ u s } + (v^3-3v+3v^{-1}-v^{-3}) \cdot H_{ t u } \\
& \qquad + (v^3-3v+3v^{-1}-v^{-3}) \cdot H_{ u t } + (v^3-3v+3v^{-1}-v^{-3}) \cdot H_{ t s } \\
& \qquad + (v^4-4v^2+6-4v^{-2}+v^{-4}) \cdot H_{ t } + (v^4-4v^2+6-4v^{-2}+v^{-4}) \cdot H_{ s } + (v^4-2v^2+2-2v^{-2}+v^{-4}) \cdot H_{ u } + (v^5-2v^3+4v-4v^{-1}+2v^{-3}-v^{-5}) \cdot H_{ e } \\
\varphi( H_{ t s t u t } X_{\varpi_\beta} ) & = H_{ t s t u t } + (-v+v^{-1}) \cdot H_{ t s t u } + (-2v+2v^{-1}) \cdot H_{ t u s t } + (v-v^{-1}) \cdot H_{ s t u t } \\
& \qquad + (v^2-2+v^{-2}) \cdot H_{ t u s } + (-2v^2+4-2v^{-2}) \cdot H_{ s t u } + (3v^2-6+3v^{-2}) \cdot H_{ t s t } + (-2v^2+4-2v^{-2}) \cdot H_{ t u t } + (v^2-2+v^{-2}) \cdot H_{ u s t } \\
& \qquad + (v^3-3v+3v^{-1}forSi-v^{-3}) \cdot H_{ t u } + (v^3-3v+3v^{-1}-v^{-3}) \cdot H_{ u s } + (2v^3-6v+6v^{-1}-2v^{-3}) \cdot H_{ s t } + (-3v+3v^{-1}) \cdot H_{ t s } + (v^3-4v+4v^{-1}-v^{-3}) \cdot H_{ u t } \\
& \qquad + (v^4-5v^2+8-5v^{-2}+v^{-4}) \cdot H_{ u } + (-v^4+v^2+v^{-2}-v^{-4}) \cdot H_{ t } + (-3v^2+6-3v^{-2}) \cdot H_{ s } \\
\varphi( H_{ t u s t u } X_{\varpi_\beta} ) & = H_{ t u s t u } + (v-v^{-1}) \cdot H_{ t u t u } + (-3v+3v^{-1}) \cdot H_{ t s t u } + (-v+v^{-1}) \cdot H_{ u s t u } \\
& \qquad + (v^2-2+v^{-2}) \cdot H_{ t u s } + (-v^2+2-v^{-2}) \cdot H_{ u t u } + (-3v^2+6-3v^{-2}) \cdot H_{ s t u } \\
& \qquad + (2v^3-6v+6v^{-1}-2v^{-3}) \cdot H_{ t u } + (v^3-3v+3v^{-1}-v^{-3}) \cdot H_{ u s } + (-v+v^{-1}) \cdot H_{ t s } \\
& \qquad + (-2v^2+4-2v^{-2}) \cdot H_{ t } + (v^4-7v^2+12-7v^{-2}+v^{-4}) \cdot H_{ u } + (-v^2+2-v^{-2}) \cdot H_{ s } + (-v^3+3v-3v^{-1}+v^{-3}) \cdot H_{ e } \\
\varphi( H_{ t u t s t } X_{\varpi_\beta} ) & = H_{ t u t s t } + (-v+v^{-1}) \cdot H_{ t u t s } + (2v-2v^{-1}) \cdot H_{ t u s t } + (v-v^{-1}) \cdot H_{ u t s t } \\
& \qquad + (-v^2+2-v^{-2}) \cdot H_{ t u t } + (-v^2+2-v^{-2}) \cdot H_{ t u s } + (-2v^2+4-2v^{-2}) \cdot H_{ u t s } + (v^2-2+v^{-2}) \cdot H_{ u s t } \\
& \qquad + (-v^3+4v-4v^{-1}+v^{-3}) \cdot H_{ t u } + (-2v^3+6v-6v^{-1}+2v^{-3}) \cdot H_{ u t } + (-3v^3+9v-9v^{-1}+3v^{-3}) \cdot H_{ t s } \\
& \qquad + (-2v^3+6v-6v^{-1}+2v^{-3}) \cdot H_{ u s } + (3v-3v^{-1}) \cdot H_{ s t } \\
& \qquad + (-3v^4+15v^2-24+15v^{-2}-3v^{-4}) \cdot H_{ t } + (-2v^4+9v^2-14+9v^{-2}-2v^{-4}) \cdot H_{ u } + (-3v^4+9v^2-12+9v^{-2}-3v^{-4}) \cdot H_{ s } \\
& \qquad + (-3v^5+12v^3-21v+21v^{-1}-12v^{-3}+3v^{-5}) \cdot H_{ e } \\
\varphi( H_{ t u t u s } X_{\varpi_\beta} ) & = H_{ t u t u s } + (2v-2v^{-1}) \cdot H_{ t u t u } + (2v-2v^{-1}) \cdot H_{ u t u s } \\
& \qquad + (2v^2-4+2v^{-2}) \cdot H_{ u t u } + (2v^2-4+2v^{-2}) \cdot H_{ t u s } \\
& \qquad + (2v^3-6v+6v^{-1}-2v^{-3}) \cdot H_{ t u } + (v^3-v^{-3}) \cdot H_{ u s } + (v-v^{-1}) \cdot H_{ t s } \\
& \qquad + (v^4-v^2-v^{-2}+v^{-4}) \cdot H_{ u } + (v^2-2+v^{-2}) \cdot H_{ t } + (2v^2-4+2v^{-2}) \cdot H_{ s } + (2v^3-6v+6v^{-1}-2v^{-3}) \cdot H_{ e } \\
\varphi( H_{ u t u t u } X_{\varpi_\beta} ) & = H_{ u t u t u } + (-3v+3v^{-1}) \cdot H_{ t u t u } + (3v^2-6+3v^{-2}) \cdot H_{ u t u } \\
& \qquad + (-5v^3+13v-13v^{-1}+5v^{-3}) \cdot H_{ t u } + (-v+v^{-1}) \cdot H_{ u t } + (v^4+3v^2-8+3v^{-2}+v^{-4}) \cdot H_{ u } + (3v^2-6+3v^{-2}) \cdot H_{ t } + (-v^3+v^{-3}) \cdot H_{ e } \\
\varphi( H_{ u t u s t } X_{\varpi_\beta} ) & = H_{ u t u s t } + (-v+v^{-1}) \cdot H_{ u t u s } + (v-v^{-1}) \cdot H_{ u t u t } + (-3v+3v^{-1}) \cdot H_{ u t s t } \\
& \qquad + (-2v^2+4-2v^{-2}) \cdot H_{ u t u } + (-3v^2+6-3v^{-2}) \cdot H_{ t u s } + (-6v^2+12-6v^{-2}) \cdot H_{ t s t } + (-5v^2+10-5v^{-2}) \cdot H_{ u s t } \\
& \qquad + (-3v^3+9v-9v^{-1}+3v^{-3}) \cdot H_{ t u } + (-2v^3+6v-6v^{-1}+2v^{-3}) \cdot H_{ u s } + (-2v^3+6v-6v^{-1}+2v^{-3}) \cdot H_{ u t } + (-5v^3+14v-14v^{-1}+5v^{-3}) \cdot H_{ s t } \\
& \qquad + (-2v^4+5v^2-6+5v^{-2}-2v^{-4}) \cdot H_{ u } + (-v^2+2-v^{-2}) \cdot H_{ s } + (v^4-5v^2+8-5v^{-2}+v^{-4}) \cdot H_{ t } + (-7v^3+21v-21v^{-1}+7v^{-3}) \cdot H_{ e } \\
\varphi( H_{ u t s t u } X_{\varpi_\beta} ) & = H_{ u t s t u } + (-v+v^{-1}) \cdot H_{ u t u s } + (2v-2v^{-1}) \cdot H_{ u s t u } + (3v-3v^{-1}) \cdot H_{ t s t u } \\
& \qquad + (-v^2+2-v^{-2}) \cdot H_{ u t u } + (-3v^2+6-3v^{-2}) \cdot H_{ t u s } + (3v^2-6+3v^{-2}) \cdot H_{ s t u } \\
& \qquad + (-3v^3+9v-9v^{-1}+3v^{-3}) \cdot H_{ t u } + (-2v^3+6v-6v^{-1}+2v^{-3}) \cdot H_{ u s } \\
& \qquad + (-2v^4+10v^2-16+10v^{-2}-2v^{-4}) \cdot H_{ u } + (-v^2+2-v^{-2}) \cdot H_{ s } + (-v^3+4v-4v^{-1}+v^{-3}) \cdot H_{ e } \\
\varphi( H_{ u s t u t } X_{\varpi_\beta} ) & = H_{ u s t u t } + (-v+v^{-1}) \cdot H_{ u s t u } + (-v+v^{-1}) \cdot H_{ u t u t } + (-3v+3v^{-1}) \cdot H_{ s t u t } \\
& \qquad + (-v^2+2-v^{-2}) \cdot H_{ u t u } + (3v^2-6+3v^{-2}) \cdot H_{ s t u } + (2v^2-4+2v^{-2}) \cdot H_{ u s t } + (3v^2-6+3v^{-2}) \cdot H_{ t u t } \\
& \qquad + (-v^3+v^{-3}) \cdot H_{ u s } + (3v^3-9v+9v^{-1}-3v^{-3}) \cdot H_{ t u } + (-2v^3+4v-4v^{-1}+2v^{-3}) \cdot H_{ s t } \\
& \qquad + (-2v^4+5v^2-6+5v^{-2}-2v^{-4}) \cdot H_{ u } + (4v^2-8+4v^{-2}) \cdot H_{ s } + (-v^2+2-v^{-2}) \cdot H_{ t } + (3v^3-9v+9v^{-1}-3v^{-3}) \cdot H_{ e } \\
& \qquad + (3v^3-9v+9v^{-1}-3v^{-3}) \cdot H_{ t } + (2v^3-6v+6v^{-1}-2v^{-3}) \cdot H_{ u } + (3v^3-6v+6v^{-1}-3v^{-3}) \cdot H_{ s } + (3v^4-9v^2+12-9v^{-2}+3v^{-4}) \cdot H_{ e } \\
\varphi( H_{ t u s t } X_{\varpi_\beta} ) & = H_{ t u s t } + (-v+v^{-1}) \cdot H_{ t u s } + (v-v^{-1}) \cdot H_{ t u t } + (-3v+3v^{-1}) \cdot H_{ t s t } + (-v+v^{-1}) \cdot H_{ u s t } \\
& \qquad + (-2v^2+4-2v^{-2}) \cdot H_{ t u } + (-v^2+2-v^{-2}) \cdot H_{ u s } + (-v^2+2-v^{-2}) \cdot H_{ u t } + (-3v^2+6-3v^{-2}) \cdot H_{ s t } \\
& \qquad + (-v^3+3v-3v^{-1}+v^{-3}) \cdot H_{ u } + (-3v^2+6-3v^{-2}) \cdot H_{ e } \\
\varphi( H_{ u t s t } X_{\varpi_\beta} ) & = H_{ u t s t } + (-v+v^{-1}) \cdot H_{ u t s } + (2v-2v^{-1}) \cdot H_{ u s t } + (3v-3v^{-1}) \cdot H_{ t s t } \\
& \qquad + (-v^2+2-v^{-2}) \cdot H_{ u t } + (-v^2+2-v^{-2}) \cdot H_{ u s } + (-3v^2+6-3v^{-2}) \cdot H_{ t s } + (3v^2-6+3v^{-2}) \cdot H_{ s t } \\
& \qquad + (-v^3+4v-4v^{-1}+v^{-3}) \cdot H_{ u } + (-3v^3+9v-9v^{-1}+3v^{-3}) \cdot H_{ t } + (-3v^3+9v-9v^{-1}+3v^{-3}) \cdot H_{ s } + (-3v^4+15v^2-24+15v^{-2}-3v^{-4}) \cdot H_{ e } \\
\varphi( H_{ t s t u } X_{\varpi_\beta} ) & = H_{ t s t u } + (-v+v^{-1}) \cdot H_{ t u s } + (2v-2v^{-1}) \cdot H_{ s t u } \\
& \qquad + (-v^2+2-v^{-2}) \cdot H_{ t u } + (-v^2+2-v^{-2}) \cdot H_{ u s } + (-v^3+4v-4v^{-1}+v^{-3}) \cdot H_{ u } \\
\varphi( H_{ u s t } X_{\varpi_\beta} ) & = H_{ u s t } + (-v+v^{-1}) \cdot H_{ u s } + (v-v^{-1}) \cdot H_{ u t } + (-3v+3v^{-1}) \cdot H_{ s t } + (-2v^2+4-2v^{-2}) \cdot H_{ u } + (-3v^2+6-3v^{-2}) \cdot H_{ t } \\
\varphi( H_{ u t u } X_{\varpi_\beta} ) & = H_{ u t u } + (-3v+3v^{-1}) \cdot H_{ t u } + (v^2-2+v^{-2}) \cdot H_{ u } + (-v+v^{-1}) \cdot H_{ e } \\
\varphi( H_{ t u t } X_{\varpi_\beta} ) & = H_{ t u t } + (-v+v^{-1}) \cdot H_{ t u } + (-2v+2v^{-1}) \cdot H_{ u t } + (v^2-2+v^{-2}) \cdot H_{ u } + (3v^2-6+3v^{-2}) \cdot H_{ t } + (-3v+3v^{-1}) \cdot H_{ e } \\
\varphi( H_{ t u s } X_{\varpi_\beta} ) & = H_{ t u s } + (2v-2v^{-1}) \cdot H_{ t u } + (v-v^{-1}) \cdot H_{ u s } + (v^2-2+v^{-2}) \cdot H_{ u } \\
\varphi( H_{ t s t } X_{\varpi_\beta} ) & = H_{ t s t } + (-v+v^{-1}) \cdot H_{ t s } + (2v-2v^{-1}) \cdot H_{ s t } + (-v^2+2-v^{-2}) \cdot H_{ t } + (-v^2+2-v^{-2}) \cdot H_{ s } + (-v^3+4v-4v^{-1}+v^{-3}) \cdot H_{ e } \\
\varphi( H_{ u t } X_{\varpi_\beta} ) & = H_{ u t } + (-v+v^{-1}) \cdot H_{ u } + (-3v+3v^{-1}) \cdot H_{ t } \\
\varphi( H_{ t u } X_{\varpi_\beta} ) & = H_{ t u } + (-v+v^{-1}) \cdot H_{ u } \\
\varphi( H_{ t } X_{\varpi_\beta} ) & = H_{ t } + (-v+v^{-1}) \cdot H_{ e } \\
\end{align*}

\endgroup

\noindent
The Laurent polynomials that appear as coefficients of the standard basis elements $H_x$ in the equations listed above can be arranged in two $34 \times 20$ matrices.
One checks that the column vector (with 20 rows) whose entries are the Laurent polynomials that appear as coefficients of the standard basis elements in \eqref{eq:G2Bstuts} belongs to the kernel of both of these matrices, and so $\varphi( B_{stuts} X_{\varpi_\beta} ) = 0$ and $\varphi( B_{stuts} X_{\varpi_\beta} ) = 0$, as claimed.
As an example, the coefficient of $H_e$ in $\varphi(B_{stuts} X_{\varpi_\alpha})$ is given by
\begin{align*}
	& (v^5-3v^3+6v-6v^{-1}+3v^{-3}-v^{-5}) + 0 + 0 + (v^3-3v+3v^{-1}-v^{-3}) + (-v^5+3v^3-4v+4v^{-1}-3v^{-3}+v^{-5}) \\
	& \hspace{2cm} + (-v^3+v-v^{-1}+v^{-3}) + (-2v^3+6v-6v^{-1}+2v^{-3}) + (-2v^3+6v-6v^{-1}+2v^{-3}) + (v^3-3v+3v^{-1}-v^{-3}) \\
	& \hspace{2cm} + (-v+v^{-1}) \cdot (-v^2+2-v^{-2}) + (-v+v^{-1}) \cdot (2v^2-4+2v^{-2}) + (-v+v^{-1}) \cdot (-v^2+2-v^{-2}) \\
	& \hspace{2cm}  -2 \cdot (-v+v^{-1}) -2 \cdot 0 + (v^2-5+v^{-2}) \cdot (v-v^{-1}) - (-v+v^{-1}) - 2 \cdot (-v^3+3v-3v^{-1}+v^{-3}) \\
	& \hspace{2cm} + (v-v^{-1}) \cdot 0 + (v-v^{-1}) \cdot 0 + (-v^2+7-v^{-2}) \cdot 0 \\
	& \hspace{2cm} = 0 ,
\end{align*}
and the coefficient of $H_e$ in $\varphi(B_{stuts} X_{\varpi_\beta})$ is given by
\begin{align*}
	& (v^5-2v^3+4v-4v^{-1}+2v^{-3}-v^{-5}) + 0 + (-v^3+3v-3v^{-1}+v^{-3}) + (-3v^5+12v^3-21v+21v^{-1}-12v^{-3}+3v^{-5}) + (2v^3-6v+6v^{-1}-2v^{-3}) \\
	& \hspace{2cm} + (-v^3+v^{-3}) + (-7v^3+21v-21v^{-1}+7v^{-3}) + (-v^3+4v-4v^{-1}+v^{-3}) + (3v^3-9v+9v^{-1}-3v^{-3}) \\
	& \hspace{2cm} + (-v+v^{-1}) \cdot (-3v^2+6-3v^{-2}) + (-v+v^{-1}) \cdot (-3v^4+15v^2-24+15v^{-2}-3v^{-4}) + (-v+v^{-1}) \cdot 0 \\
	& \hspace{2cm}  - 2 \cdot (-3v+3v^{-1}) -2 \cdot 0 + (v^2-5+v^{-2}) \cdot (-v^3+4v-4v^{-1}+v^{-3}) - (-v+v^{-1}) - 2 \cdot 0 \\
	& \hspace{2cm} + (v-v^{-1}) \cdot 0 + (v-v^{-1}) \cdot 0 + (-v^2+7-v^{-2}) \cdot (-v+v^{-1}) \\
	& \hspace{2cm} = 0 .
\end{align*}

\end{landscape}

\bibliographystyle{alpha}
\bibliography{centralizer}

\end{document}